\documentclass[12pt]{scrartcl}  
\usepackage{a4}   
\usepackage{amsmath}    
\usepackage{paralist}
\usepackage{amssymb} 
\usepackage{amsfonts}  
\usepackage{mathrsfs}  
\usepackage{dsfont}
\usepackage{latexsym} 
\usepackage{xcolor}
\usepackage{bbm,exscale}
\definecolor{Myblue}{rgb}{0,0,0.6}  
\usepackage[colorlinks,citecolor=Myblue,linkcolor=Myblue,urlcolor=Myblue,pdfpagemode=None]{hyperref}
\usepackage{amsthm}
\usepackage{accents}
\usepackage[square,numbers,sort&compress]{natbib} 
\usepackage[all,cmtip]{xy}
\usepackage{ifthen} 
\usepackage{bbding}
\usepackage{stmaryrd}  
\usepackage{wasysym}
\usepackage{verbatim}
\usepackage{bbding} 
\usepackage{soul}  
\usepackage[yyyymmdd,hhmmss]{datetime}
\usepackage{booktabs}
\usepackage{enumitem}
\usepackage{color}
\usepackage{tikz}
\usepackage{textcomp}
\usepackage{gensymb}
\usepackage{tikz-cd}
\usepackage{tikz-3dplot}
\usepackage{pgfplots}
	\pgfplotsset{width=7cm,compat=1.8}
	\usetikzlibrary{decorations.pathreplacing}
	\usetikzlibrary{decorations.markings}
	\usetikzlibrary{patterns}

\tikzset{
    string/.style={draw=#1, postaction={decorate}, decoration={markings,mark=at position .51 with {\arrow[color=#1]{>}}}},
    costring/.style={draw=#1, postaction={decorate}, decoration={markings,mark=at position .51 with {\arrow[draw=#1]{<}}}},
    ostring/.style={draw=#1, postaction={decorate}, decoration={markings,mark=at position .47 with {\arrow[draw=#1]{>}}}},
    ustring/.style={draw=#1, postaction={decorate}, decoration={markings,mark=at position .56 with {\arrow[draw=#1]{>}}}},
    oostring/.style={draw=#1, postaction={decorate}, decoration={markings,mark=at position .43 with {\arrow[draw=#1]{>}}}},
    uustring/.style={draw=#1, postaction={decorate}, decoration={markings,mark=at position .59 with {\arrow[draw=#1]{>}}}},
    directed/.style={string=blue!50!black}, 
    odirected/.style={ostring=blue!50!black}, 
    udirected/.style={ustring=blue!50!black}, 
    oodirected/.style={oostring=blue!50!black}, 
    uudirected/.style={uustring=blue!50!black},     
    redirected/.style={costring= blue!50!black},
    redirectedgreen/.style={costring= green!50!black},
    directedgreen/.style={string= green!50!black},
    redirectedlightgreen/.style={costring= green!65!black},
    directedlightgreen/.style={string= green!65!black},
}

\tikzset{-dot-/.style={decoration={
  markings,
  mark=at position 0.5 with {\fill circle (1.875pt);}},postaction={decorate}}}

\tikzset{
	Fdot/.style={circle, draw, fill, inner sep=0pt}, 
	Odot/.style={circle, draw, inner sep=0.1pt, minimum size=0.1cm}
	}

\newcommand\tikzzbox[1]
{#1}

  \vfuzz \hfuzz
\makeatletter
\newcommand{\raisemath}[1]{\mathpalette{\raisem@th{#1}}}
\newcommand{\raisem@th}[3]{\raisebox{#1}{$#2#3$}}
\makeatother

\renewcommand{\geq}{\geqslant}
\newcommand{\rp}[1]{\widetilde{#1}}

\newcommand{\D}{\mathds{D}}

\newcommand{\R}{\mathds{R}}
\newcommand{\Z}{\mathds{Z}}

\def\1{\ifmmode\mathrm{1\!l}\else\mbox{\(\mathrm{1\!l}\)}\fi}
\newcommand{\one}{\mathbbm{1}}
\newcommand{\be}{\begin{equation}}
\newcommand{\ee}{\end{equation}}
\newcommand{\bes}{\begin{equation*}}
\newcommand{\ees}{\end{equation*}}

\newcommand{\DC}{\mathds{D}^{\mathcal C}}

\newcommand{\interior}[1]{%
  {\kern0pt#1}^{\mathrm{o}}%
}

\newcommand{\id}{\text{id}}

\newcommand{\Hom}{\operatorname{Hom}}
\newcommand{\End}{\operatorname{End}}

\def\lra{\longrightarrow}
\def\lmt{\longmapsto}

\newcommand{\Borddc}{\widehat{\operatorname{Bord}}{}^{\mathrm{def}}_3(\DC)}
\newcommand{\Borddefn}[1] {\operatorname{Bord}^{\mathrm{def}}_{#1}}
\newcommand{\Bordribn}[1] {\operatorname{Bord}^{\mathrm{rib}}_{#1}}
\newcommand{\Bordriben}[1] {\widehat{\operatorname{Bord}}{}^{\mathrm{rib}}_{#1}}
\newcommand{\Borddefen}[1] {\widehat{\operatorname{Bord}}{}^{\mathrm{def}}_{#1}}

\newcommand{\zz}{\mathcal{Z}}
\newcommand{\zzc}{\mathcal{Z}^{\mathcal C}}
\newcommand{\zrt}{\mathcal{Z}^{\text{RT,}\mathcal C}}

\newcommand{\Vect}{\operatorname{Vect}}
\newcommand{\Vectk}{\operatorname{Vect}_\Bbbk}

\newcommand{\eps}{\varepsilon}

\newcommand{\Cat}[1]         {\operatorname{\mathcal{#1}}}

\newcommand{\FEnd}{\mathcal{E}\hspace*{-.7pt}nd}

\definecolor{DarkViolet} {rgb}{0.580392,0.000000,0.827450}

\newcommand\arxiv[2]      {\href{http://arXiv.org/abs/#1}{#2}}
\newcommand\doi[2]        {\href{http://dx.doi.org/#1}{#2}}
\newcommand\httpurl[2]    {\href{http://#1}{#2}}

\allowdisplaybreaks

\deffootnote[1em]{1em}{1em}{\textsuperscript{\thefootnotemark}}

\theoremstyle{definition}
\newtheorem{definition}{Definition}

\newtheorem{theorem}[definition]{Theorem}

\newtheorem{definitionlemma}[definition]{Definition and Lemma}
\newtheorem{lemma}[definition]{Lemma}

\newtheorem{remark}[definition]{Remark}

\newtheorem{example}[definition]{Example}
\newtheorem{construction}[definition]{Construction}

\numberwithin{equation}{section}
\numberwithin{definition}{section}
\numberwithin{figure}{section}

\newcommand\void[1]{}

\begin{document}

\title{Line and surface defects\\ 
in Reshetikhin-Turaev TQFT}

\author{%
Nils Carqueville$^*$ \quad
Ingo Runkel$^\#$ \quad
Gregor Schaumann$^*$%
\\[0.5cm]
   \normalsize{\texttt{\href{mailto:nils.carqueville@univie.ac.at}{nils.carqueville@univie.ac.at}}} \\  %
   \normalsize{\texttt{\href{mailto:ingo.runkel@uni-hamburg.de}{ingo.runkel@uni-hamburg.de}}} \\
   \normalsize{\texttt{\href{mailto:gregor.schaumann@univie.ac.at}{gregor.schaumann@univie.ac.at}}}\\[0.1cm]
   {\normalsize\slshape $^*$Fakult\"at f\"ur Mathematik, Universit\"at Wien, Austria}\\[-0.1cm]
   {\normalsize\slshape $^\#$Fachbereich Mathematik, Universit\"{a}t Hamburg, Germany}\\[-0.1cm]
}

\date{}
\maketitle

\begin{abstract}
A modular tensor category $\Cat{C}$ gives rise to a Reshetikhin-Turaev type topological quantum field theory which is defined on 3-dimensional bordisms with embedded $\Cat{C}$-coloured ribbon graphs. 
We extend this construction to include bordisms with surface defects which in turn can meet along line defects. 
The surface defects are labelled by $\Delta$-separable symmetric  
Frobenius algebras and the line defects by ``multi-modules'' 
which are equivariant with respect to a cyclic group action. 
Our invariant cannot distinguish non-isotopic embeddings of 2-spheres, but we give an example where it distinguishes non-isotopic embeddings of 2-tori.
\end{abstract}

\newpage

\tableofcontents

\section{Introduction}
\label{sec:intro}

The study of field theories with defects of various dimensions has seen much recent activity. 
Of particular relevance for the present paper are works concerned with general properties of topological quantum field theories (TQFTs) with defects, such as
\cite{
ks1012.0911,
KK1104.5047,
dkr1107.0495,
fsv1203.4568,
BMS,
Fuchs:2015daa,
CMS,
CRS1}.
An introduction to 2-dimensional TQFTs with defects can be found in \cite{2ddTQFTreview}.

The present paper is the second in a series whose aim is to study orbifolds of TQFTs via their defects. 
In the first paper \cite{CRS1} we developed a bordism category for $n$-dimensional TQFTs with defects, described an orbifold procedure in terms of defects, 
and gave the algebraic conditions for a collection of defects to serve as an input for this procedure.
	In the third paper \cite{CRS3} we will present several examples of such orbifolds for Reshetikhin-Turaev theory, including those coming from group extensions of tensor categories, and we will identify Turaev-Viro theory as the ``orbifold of the trivial theory''. 
The present paper lays the groundwork for 
	these applications, 
and to our knowledge provides the first systematic construction of Reshetikhin-Turaev TQFT with surface defects
	as a functor from stratified and suitably decorated bordisms to vector spaces.
This has potential applications beyond the orbifold construction, in particular in topological quantum computation \cite{KitaevTQC, BJQ, fs1310.1329}, 
	which is another motivation for the present work.

\medskip

We will be concerned with 3-dimensional TQFTs with defects. 
The relevant bordism category has morphisms which are equivalence classes of 
oriented stratified 3-manifolds \cite{CMS,CRS1}. For each $j \in \{1,2,3 \}$, the $j$-strata (i.\,e.~the connected components of the prescribed submanifolds of dimension~$j$) carry a label from a set $D_j$ of defect labels. We denote this bordism category by 
\be\label{eq:intro-borddef}
 \Borddefn{3}(\D) \, ,
\ee 
where $\D$, the set of ``defect data'', contains the three chosen sets $D_1, D_2,D_3$ and additional information on allowed adjacencies between the defects of various codimension. 
All this is reviewed in Section~\ref{sec:defectTQFTs}.

Reshetikhin-Turaev TQFTs \cite{retu2,tur} are defined in terms of a modular tensor category $\Cat{C}$ and are 3-dimensional TQFTs that can be evaluated on bordisms with embedded $\Cat{C}$-coloured ribbon tangles.
Given $\Cat{C}$, we define a set of defect data $\D^{\Cat{C}}$ and we give a prescription in terms of the underlying Reshetikhin-Turaev TQFT which produces a symmetric monoidal functor 
\be\label{eq:intro-RTdef}
\zzc \colon 
\Borddefen{3}(\D^{\Cat{C}}) \lra \Vect \, .
\ee
The hat refers to the extension of the bordism category necessary to absorb the gluing anomaly \cite{tur} (if~$\Cat{C}$ has trivial anomaly, the functor \eqref{eq:intro-RTdef} factors as $\Borddefen{3}(\D^{\Cat{C}}) \to \Borddefn{3}(\D^{\Cat{C}}) \to \Vect$),
see  Sections~\ref{sec:MTCRT} and~\ref{sec:RTdef}. 
This extension of Reshetikhin-Turaev TQFT to include
surface defects is the main contribution of the present paper. 
 
\medskip 
 
In the remainder of the introduction we will outline the defect data $\D^{\Cat{C}}$ and the construction of the functor $\zzc$ in \eqref{eq:intro-RTdef}. Let thus $\Cat{C}$ be a modular tensor category. We start by describing the sets $D_j^{\Cat{C}}$ of defect labels which decorate the $j$-strata of morphisms in $\Borddefen{3}(\DC)$.
\begin{itemize}[leftmargin=24pt]
\item[$D_3^{\Cat{C}}$:]
These are the labels for the top-dimensional strata. The present formalism can only handle the situation that there is precisely one such label, namely~$\Cat{C}$, 
\be
D_3^{\Cat{C}} = \{ \Cat{C} \} \, , 
\ee
reflecting the fact that we consider only surface defects from a modular tensor category to itself.
Ultimately one would like a theory 
which includes surface defects between different modular tensor categories,  cf.\ Remark~\ref{rem:intro-whatelse} below.
\item[$D_2^{\Cat{C}}$:]
The labels for surface defects are given by Frobenius algebras in $\Cat{C}$ whose pairing is symmetric and whose coproduct is right inverse to the product, $\mu \circ \Delta = \id_A$, a condition we refer to as ``$\Delta$-separable''. Thus:
\be
D_2^{\Cat C} = \big\{ \Delta\text{-separable symmetric Frobenius algebras in } \Cat C \big\} \, .
\ee
This description of surface defects goes back to \cite{ks1012.0911}, where the construction in \cite{tft1} of consistent sets of correlators of 2-dimensional conformal field theories in terms of such algebras was re-interpreted using surface defects. 
It was studied in detail in \cite{fsv1203.4568}, where also a Morita invariant formulation in terms of module categories is given.
	In the present paper we work in the formulation of defect TQFTs as developed in \cite{CMS, CRS1}. 
\item[$D_1^{\Cat{C}}$:]
The description of line defects is 
	(original and) 
more complicated than the previous two cases. 
An arbitrary (but finite) number of surface defects 
can join at a line defect, much like the pages of a book join at the spine, see Figure~\ref{fig:intro-linedef}. 
If the algebras describing the surface defects are $A_1,A_2,\dots,A_n \in D_2^{\Cat{C}}$, 
the obvious guess for the line defect label is an $A_1 \otimes \cdots \otimes A_n$-module~$M$. 
This is almost correct, except for two effects that still need to be taken into account. 

The first effect is that not all surface defects meeting at the line defect need to be oriented in the same way. This will result in some of the~$A_i$ being replaced by their opposite algebra $A_i^{\mathrm{op}}$. 
We will write $A_i^+ = A_i$ and $A_i^-=A_i^\mathrm{op}$.

The second issue is that in our approach, the surface defects around a line defect are only cyclically ordered (via the orientations of the line defect and the surrounding 3-manifold), but they have no total order, see again Figure~\ref{fig:intro-linedef}. Assume for example that all $A_i$ are equal: $A_i = A$ for some $A \in D_2^{\Cat{C}}$ and all $i \in \{1,\dots,n\}$. 
One can define an action of the cyclic group $C_n$ on $A^{\otimes n}$-modules, and we label the corresponding line defects by modules which are $C_n$-equivariant, see Section~\ref{sec:modules}.
Of course, if~$n$ surface defects meet at a line defect, the maximal cyclic symmetry may be any divisor of~$n$, including $1$ (i.\,e.\ no symmetry). This leads us to the notion of a ``multi-module with cyclic symmetry'' (Definitions~\ref{def:symmetricmodule} and~\ref{def:cyclic-with-signs}).

A 1-stratum with no attached 0-stratum is labelled by an object $\Cat{C}$ with trivial twist. Altogether we thus set
$
D_1^{\Cat C} := \bigsqcup_{n \in \mathbb{Z}_{\geqslant 0}} L_n
$, 
where $L_0 = \{ X \in \Cat{C} \,\big|\, \theta_X = \id_X \}$,  
and, for $n>0$,
\begin{align}
	L_n =\, &\big\{ 
	\big((A_1,\eps_1), (A_2,\eps_2), \dots, (A_n,\eps_n), M\big) \;\big|\; 
					A_i \in D_2^{\Cat C} , \; 
					\eps_i \in \{ \pm \} , \; 
	 \nonumber\\ & \qquad
					M \text{ is an $A_1^{\eps_1} \otimes \cdots \otimes A_n^{\eps_n}$-module equivariant for the maximal}
	 \nonumber\\[-.2em] & \qquad
					\text{cyclic symmetry of }
					\big((A_1,\eps_1), (A_2,\eps_2), \dots, (A_n,\eps_n)\big)
					\big\} \ .
	\label{eq:D1C}
\end{align}
\end{itemize}

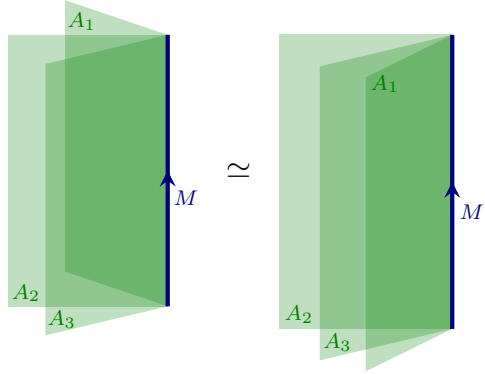
\begin{figure}[tb]
$$
\tikzzbox{\begin{tikzpicture}[very thick,scale=1.2,color=blue!50!black, baseline=-1.9cm]
%
\fill [green!50!black,opacity=0.25] (0,0) -- ($(0,0)+(130:1.75 and 0.5)$) -- ($(0,-3)+(130:1.75 and 0.5)$) -- (0,-3) -- (0,0);
\fill [green!50!black,opacity=0.25] (0,0) -- ($(0,0)+(180:1.75 and 0.5)$) -- ($(0,-3)+(180:1.75 and 0.5)$) -- (0,-3) -- (0,0);
\fill [green!50!black,opacity=0.25] (0,0) -- ($(0,0)+(220:1.75 and 0.5)$) -- ($(0,-3)+(220:1.75 and 0.5)$) -- (0,-3) -- (0,0);
%
 \draw[
	color=blue!50!black, 
	ultra thick, 
	>=stealth,
	decoration={markings, mark=at position 0.5 with {\arrow{>}},
					}, postaction={decorate}
	] 
(0,-3) -- (0,0);
%
\draw[line width=1, color=blue!50!black] (0.2, -1.8) node[line width=0pt] (beta) {{\scriptsize$M$}};
\draw[line width=1, color=green!50!black] (-0.94, 0.15) node[line width=0pt] (beta) {{\scriptsize$A_1$}};
\draw[line width=1, color=green!50!black] (-1.55, -2.85) node[line width=0pt] (beta) {{\scriptsize$A_2$}};
\draw[line width=1, color=green!50!black] (-1.16, -3.14) node[line width=0pt] (beta) {{\scriptsize$A_3$}};
\end{tikzpicture}}
\; 
\simeq
\; \, 
\tikzzbox{\begin{tikzpicture}[very thick,scale=1.3,color=blue!50!black, baseline=-1.9cm]
%
\fill [green!50!black,opacity=0.25] (0,0) -- ($(0,0)+(180:1.75 and 0.5)$) -- ($(0,-3)+(180:1.75 and 0.5)$) -- (0,-3) -- (0,0);
\fill [green!50!black,opacity=0.25] (0,0) -- ($(0,0)+(220:1.75 and 0.5)$) -- ($(0,-3)+(220:1.75 and 0.5)$) -- (0,-3) -- (0,0);
\fill [green!50!black,opacity=0.25] (0,0) -- ($(0,0)+(240:1.75 and 0.5)$) -- ($(0,-3)+(240:1.75 and 0.5)$) -- (0,-3) -- (0,0);
%
 \draw[
	color=blue!50!black, 
	ultra thick, 
	>=stealth,
	decoration={markings, mark=at position 0.5 with {\arrow{>}},
					}, postaction={decorate}
	] 
(0,-3) -- (0,0);
%
\draw[line width=1, color=blue!50!black] (0.2, -1.8) node[line width=0pt] (beta) {{\scriptsize$M$}};
\draw[line width=1, color=green!50!black] (-1.55, -2.85) node[line width=0pt] (beta) {{\scriptsize$A_2$}};
\draw[line width=1, color=green!50!black] (-1.16, -3.14) node[line width=0pt] (beta) {{\scriptsize$A_3$}};
\draw[line width=1, color=green!50!black] (-0.68, -0.5) node[line width=0pt] (beta) {{\scriptsize$A_1$}};
\end{tikzpicture}}
$$
\caption{Surface defects labelled by algebras $A_1,A_2,A_3 \in D_2^{\Cat{C}}$ meeting at a line defect labelled $M$.
The arrangement on the left may be isotoped to the arrangement on the right by flipping the $A_1$-labelled surface clockwise 
around~$M$ from the back to the front. Line defects only know the cyclic ordering of the surface defects adjacent to them; there is no total order.}
\label{fig:intro-linedef}
\end{figure}

The construction of the functor $\zzc$ in \eqref{eq:intro-RTdef} works in two steps. Given a bordism in $\Borddefen{3}(\D^{\Cat{C}})$, we first build a 3-manifold with embedded ribbon tangle which is constructed form the data of the defect strata. 
As in \cite{tft1,ks1012.0911,fsv1203.4568}, a surface defect labelled $A$ is replaced by a network of ribbons labelled $A$. Each $M$-labelled 
line defect is replaced by a ribbon with the same label, 
and joined to the network of $A$-ribbons of an adjacent surface defect via the module action (recall \eqref{eq:D1C}). 
Secondly, the resulting 3-manifold with ribbon tangle is evaluated in the Reshetikhin-Turaev TQFT for $\Cat{C}$. The value of the functor $\zzc$ on objects is defined via a limit construction. We refer to Section~\ref{sec:RTdef} for the details. Our main result then is (cf.\ Theorem~\ref{thm:mainresult}):

\begin{theorem}
\label{theorem:Intro}
Via the construction outlined above, a modular tensor category~$\Cat{C}$ 
gives rise to
a defect TQFT, that is, a symmetric monoidal functor $\zzc$ as in \eqref{eq:intro-RTdef}.
\end{theorem}

\begin{remark}\label{rem:intro-whatelse}
\begin{enumerate}
\item
It would clearly be desirably (and natural) to extend the formalism presented here to $D_3$ being all modular tensor categories. Even without knowing the details of this extension, from \cite{fsv1203.4568} it is already clear
that between two 3-strata labelled $\Cat{C}$ and $\Cat{D}$ there can be surface defects if and only if $\Cat{C}$ and $\Cat{D}$ are in the same Witt class, i.\,e.\ if $\Cat{C}\boxtimes\Cat{D}^{\mathrm{rev}}$ is braided equivalent to the Drinfeld centre of some fusion category \cite{Davydov:2010}.
\item
In \cite{fsv1203.4568} a Morita invariant description of surface defects in terms of module categories is given. 
In our formalism, Morita equivalent algebras $A \in \Cat{C}$ do not quite describe the same surface defect. 
For example, a 2-sphere labelled~$A$ evaluates to $\dim(A)$, the categorical dimension of $A$ in $\Cat{C}$ (see Section~\ref{sec:RTdef}), which is not a Morita invariant quantity. This difference between Morita equivalent algebras can be  phrased in terms of invertible surface defects which evaluate to Euler characteristic dependent constants. 
As detailed in \cite[Sect.\,2.5]{CRS1}, one can always complete a defect TQFT with respect to such constants
and use the Euler characteristic to normalise the value of the 2-sphere to obtain a Morita invariant description of surface defects. 
\item 
Line defects in Reshetikhin-Turaev theory with a nonzero number of incident surface defects are an original contribution of the present paper. 
Already in the case of only one label for 3-strata, the presence of such line defects makes possible the appearance of ``foams'' in Reshetikhin-Turaev theory, where many 2-strata may meet at 1-strata such that the union of all 1- and 2-strata does not form a manifold. 
Such foams are a crucial ingredient of the 3-dimensional orbifold construction of \cite{CRS1, CRS3}. 
Moreover, foams feature prominently in some constructions of homological link invariants \cite{KR0404189, msv0708.2228, QR1405.5920}, where all 3-strata are labelled ``trivially'' while 2-strata are labelled by certain Landau-Ginzburg models. 
In view of the relation to categorified quantum groups \cite{MY1306.6242}, it would be interesting to systematically study foams also in Reshetikhin-Turaev theory. 
\item
In the description of the set of defect data $\DC$ we did not include a label set $D_0^{\Cat{C}}$ for 0-strata. 
The reason for this is twofold. 
Firstly, their direct description is even more involved than that of line defects, see~\cite[Def.\,2.4]{CRS1} for the general case. 
Secondly, for a defect TQFT without labels for 0-strata, the set of defect data~$\D$ can always be canonically extended to a new set $\D^\bullet$ which now includes $D_0$. 
This is done by cutting out small balls and taking certain invariant vectors in
the corresponding state space of the TQFT, see~\cite[Sect.\,2.4]{CRS1} for details. 
\item 
The original Reshetikhin-Turaev construction assigns vector spaces to surfaces with marked points and linear maps to bordisms with embedded ribbons and coupons. 
One can make use of the extension to 0-strata mentioned in part (iv) to describe coupons, and in Remark~\ref{rem:relationToRT}(ii) we detail how the original Reshetikhin-Turaev construction embeds into the theory defined in the present paper. 
\item
The functor $\zzc$ also assigns state spaces to surfaces with marked points and decorated line defects, which correspond to boundaries of 2-strata in bordisms. 
Similar state spaces had previously been discussed in relation to topological phases of matter \cite{AiKongZheng}, using the language of factorisation homology. 
\item
Consider embeddings $\iota$ of an oriented surface~$\Sigma$ into a closed 3-manifold~$N$. 
For $\Sigma = S^2$ it is easy to convince oneself (as we will do in Section~\ref{sec:RTdef}) that any two embeddings~$\iota$ produce the same invariant $\zzc$, since one can collapse the network of algebra ribbons 
on any point of the sphere, cf.\ Section~\ref{sec:RTdef}. 
Some non-isotopic embeddings of surfaces of non-zero genus can however be distinguished by our invariant~$\zzc$. 
We illustrate this in Section~\ref{sec:RTdef} for $\Sigma = S^1 \times S^1$, $N = S^1 \times S^1 \times S^1$ and $\Cat{C}$ the modular tensor category obtained from the affine Lie algebra $\widehat{\mathfrak{sl}}(2)_k$.
\end{enumerate}
\end{remark}

This paper is organised as follows. 
In Section~\ref{sec:modules} we introduce multi-modules, give an action of the cyclic group on the category of multi-modules and discuss equivariance with respect to this action.
In Section~\ref{sec:defectTQFTs} we briefly review the category of 3-dimensional stratified bordisms and we define the sets of defect data which can be used to label the strata.
For a given modular tensor category~$\Cat{C}$,
in Section~\ref{sec:MTCRT} we recall bordisms with embedded $\Cat{C}$-coloured ribbon tangles and the formulation of Reshetikhin-Turaev TQFT for~$\Cat{C}$.
Finally, Section~\ref{sec:RTdef} gives our construction of Reshetikhin-Turaev theory with surface defects, and the proof of Theorem~\ref{theorem:Intro}, see Theorem~\ref{thm:mainresult}.

\subsubsection*{Acknowledgements} 

We would like to thank Alexei Davydov and J\"urgen Fuchs for helpful comments.
The work of N.\,C.~is partially supported by a grant from the Simons Foundation. 
N.\,C.~and G.\,S.~are partially supported by the stand-alone project P\,27513-N27 of the Austrian Science Fund. 
The authors acknowledge support by the Research Training Group 1670 of the German Research Foundation.

\section{Multi-modules with cyclic symmetry}
\label{sec:modules}

In this section we present our conventions for ribbon categories~$\Cat{C}$, and then introduce the notion of ``cyclically symmetric multi-modules'' over algebras in~$\Cat{C}$ that we will need in Section~\ref{sec:RTdef}.

\medskip

Let $\Cat{C}$ be a monoidal category which we will assume to be strict to simplify notation.
Recall that a \textsl{braiding} on~$\Cat{C}$ consists of a family of natural isomorphisms $c_{X,Y} \colon X \otimes Y \to Y \otimes X$ for all $X,Y \in \Cat{C}$ such that the hexagon identities and unit axioms are satisfied. 
In string diagrammatic language (read from bottom to top) we present the braiding and its inverse as 
\be
c_{X,Y}
= 
\begin{tikzpicture}[very thick,scale=0.5,color=blue!50!black, baseline]
\draw[color=blue!50!black] (-1,-1) node[below] (A1) {{\scriptsize$X$}};
\draw[color=blue!50!black] (1,-1) node[below] (A2) {{\scriptsize$Y$}};
\draw[color=blue!50!black] (-1,1) node[above] (B1) {{\scriptsize$Y$}};
\draw[color=blue!50!black] (1,1) node[above] (B2) {{\scriptsize$X$}};
\draw[color=blue!50!black] (A2) -- (B1);
\draw[color=white, line width=4pt] (A1) -- (B2);
\draw[color=blue!50!black] (A1) -- (B2);
\end{tikzpicture} 
\, , \quad
c^{-1}_{X,Y}
= 
\begin{tikzpicture}[very thick,scale=0.5,color=blue!50!black, baseline]
\draw[color=blue!50!black] (-1,-1) node[below] (A1) {{\scriptsize$Y$}};
\draw[color=blue!50!black] (1,-1) node[below] (A2) {{\scriptsize$X$}};
\draw[color=blue!50!black] (-1,1) node[above] (B1) {{\scriptsize$X$}};
\draw[color=blue!50!black] (1,1) node[above] (B2) {{\scriptsize$Y$}};
\draw[color=blue!50!black] (A1) -- (B2);
\draw[color=white, line width=4pt] (A2) -- (B1);
\draw[color=blue!50!black] (A2) -- (B1);
\end{tikzpicture} 
\, . 
\ee

A \textsl{twist} for a braided category~$\Cat{C}$ consists of a natural family of morphisms $\theta_{X}\colon X \to X$ for all $X \in \Cat{C}$ such that 
\be
\label{eq:twist}
\theta_{X \otimes Y} 
=
\begin{tikzpicture}[very thick,scale=0.75,color=blue!50!black, baseline=-0.2cm]
\draw[color=blue!50!black] (-1.2,0.8) node[above] (X) {{\scriptsize$X\otimes Y$}};
\draw[color=blue!50!black] (-1.2,-1.4) node[below] (X) {{\scriptsize$X\otimes Y$}};
\draw[color=blue!50!black] (-1.2,0.8) -- (-1.2,-1.4); 
\fill[color=blue!50!black] (-1.2,-0.2) circle (2.9pt) node[left] (meet2) {{\scriptsize$\theta_{X\otimes Y}$}};
\end{tikzpicture} 
= 
\begin{tikzpicture}[very thick,scale=0.75,color=blue!50!black, baseline=-0.2cm]
\draw[color=blue!50!black] (-1.2,0.8) node[above] (X) {{\scriptsize$X$}};
\draw[color=blue!50!black] (-0.5,0.8) node[above] (X) {{\scriptsize$Y$}};
\draw[color=blue!50!black] (-1.2,-1.4) node[below] (X) {{\scriptsize$X$}};
\draw[color=blue!50!black] (-0.5,-1.4) node[below] (X) {{\scriptsize$Y$}};
\draw[color=blue!50!black] (-1.2,0.8) -- (-1.2,0.4); 
\draw[color=blue!50!black] (-0.5,0.8) -- (-0.5,0.4); 
\draw[color=blue!50!black] (-1.2,0.4) .. controls +(0,-0.25) and +(0,0.25) .. (-0.5,-0.3);
\draw[color=white, line width=4pt] (-0.5,0.4) .. controls +(0,-0.25) and +(0,0.25) .. (-1.2,-0.3);
\draw[color=blue!50!black] (-0.5,0.4) .. controls +(0,-0.25) and +(0,0.25) .. (-1.2,-0.3);
\draw[color=blue!50!black] (-1.2,-0.3) .. controls +(0,-0.25) and +(0,0.25) .. (-0.5,-1);
\draw[color=white, line width=4pt] (-0.5,-0.3) .. controls +(0,-0.25) and +(0,0.25) .. (-1.2,-1);
\draw[color=blue!50!black] (-0.5,-0.3) .. controls +(0,-0.25) and +(0,0.25) .. (-1.2,-1);
\draw[color=blue!50!black] (-1.2,-1) -- (-1.2,-1.4); 
\draw[color=blue!50!black] (-0.5,-1) -- (-0.5,-1.4); 
\fill[color=blue!50!black] (-1.2,-1.1) circle (2.9pt) node[left] (meet2) {{\scriptsize$\theta_X$}};
\fill[color=blue!50!black] (-0.5,-1.1) circle (2.9pt) node[right] (meet2) {{\scriptsize$\theta_Y$}};
\end{tikzpicture} 
= 
c_{Y,X} \circ c_{X,Y} \circ (\theta_{X} \otimes \theta_{Y}) 
\, . 
\ee
Finally, a \textsl{ribbon category} is a monoidal category~$\Cat{C}$ which has a left and right dual $X^\vee$ for every object~$X$, a braiding~$c$ and a twist~$\theta$ such that 
$\theta_{X^{\vee}}=(\theta_{X})^{\vee}$ for all $X \in \Cat{C}$. 
For more details on ribbon categories, see e.\,g.\ \cite[Sect.\,2]{baki}.

\medskip

For the remainder of this section we fix a ribbon category~$\mathcal C$ to study algebras and their modules in $\mathcal C$. 
For an algebra $(A,\mu,\eta)$ with underlying object $A\in \Cat{C}$, multiplication $\mu \colon A \otimes A \to A$ and unit $\eta \colon \one \to A$, and for an $A$-module $(M,\rho)$ with object $M \in \Cat{C}$ and action $\rho \colon A\otimes M \to M$ we use the string diagram notation
\be
\begin{tikzpicture}[very thick,scale=0.75,color=blue!50!black, baseline]
\draw[color=green!50!black] (-0.85,0.2) -- (-0.85,1); 
\fill[color=green!50!black] (-0.85,0.2) circle (2.5pt) node (mult1) {};
\draw[color=green!50!black] (-1.2,-1) .. controls +(0,0.75) and +(-0.15,0) .. (-0.85,0.2);
\draw[color=green!50!black] (-0.5,-1) .. controls +(0,0.75) and +(0.15,0) .. (-0.85,0.2);
\draw[color=green!50!black] (-0.85,1) node[above] (A1) {{\scriptsize$A$}};
\draw[color=green!50!black] (-0.5,-1) node[below] (A1) {{\scriptsize$A$}};
\draw[color=green!50!black] (-1.2,-1) node[below] (A1) {{\scriptsize$A$}};
\draw[color=green!50!black] (-0.55,0.4) node (A1) {{\scriptsize$\mu$}};
\end{tikzpicture} 
\, , \qquad
\begin{tikzpicture}[very thick,scale=0.75,color=blue!50!black, baseline]
\draw[color=green!50!black] (0,-0.7) node[Odot] (D1) {}; 
\draw[color=green!50!black] (D1) -- (0,1); 
\draw[color=green!50!black] (0,1) node[above] (A1) {{\scriptsize$A$}};
\draw[color=green!50!black] (-0.25,-0.9) node (A1) {{\scriptsize$\eta$}};
\end{tikzpicture} 
\, , \qquad
\begin{tikzpicture}[very thick,scale=0.75,color=blue!50!black, baseline]
\draw (0,-1) node[below] (X) {};
\draw[color=green!50!black] (-1,-1) node[below] (A1) {};
\draw[color=green!50!black] (-0.5,-1) node[below] (A2) {};
\draw (0,1) node[right] (Xu) {};
\draw (0,-1) -- (0,1); 
\fill[color=blue!50!black] (0,0.5) circle (2.9pt) node[right] (meet2) {{\scriptsize$\rho$}};
\draw[color=green!50!black] (-1,-1) .. controls +(0,0.5) and +(-0.5,-0.5) .. (0,0.5);
\draw (0,1) node[above] (X) {{\scriptsize$M$}};
\draw (0,-1) node[below] (X) {{\scriptsize$M$}};
\draw[color=green!50!black] (-1,-1) node[below] (A1) {{\scriptsize$A$}};
\end{tikzpicture} 
\, . 
\ee
For an algebra 
$A \equiv (A,\mu,\eta)$ 
we denote by $A^{\text{op}}$ the algebra $(A, \mu^{\text{op}},\eta)$ with the convention 
\be
\label{eq:muop}
\mu^{\text{op}} := \mu \circ c_{A,A} = 
\begin{tikzpicture}[very thick,scale=0.75,color=blue!50!black, baseline]
\draw[color=green!50!black] (-0.85,0.2) -- (-0.85,1); 
\fill[color=green!50!black] (-0.85,0.2) circle (2.5pt) node (mult1) {};
\draw[color=green!50!black] (-1.2,-0.3) .. controls +(0,0.25) and +(0,0) .. (-0.85,0.2);
\draw[color=green!50!black] (-1.2,-0.3) .. controls +(0,-0.25) and +(0,0.25) .. (-0.5,-1);
\draw[color=green!50!black] (-0.5,-0.3) .. controls +(0,0.25) and +(0,0) .. (-0.85,0.2);
\draw[color=white, line width=4pt] (-0.5,-0.3) .. controls +(0,-0.25) and +(0,0.25) .. (-1.2,-1);
\draw[color=green!50!black] (-0.5,-0.3) .. controls +(0,-0.25) and +(0,0.25) .. (-1.2,-1);
\end{tikzpicture} 
\ee
for its multiplication. 

Recall that if $(A_1,\mu_1,\eta_1)$ and $(A_2,\mu_2,\eta_2)$ are algebras in~$\mathcal C$ then their tensor product $A_1 \otimes A_2$ also carries an algebra structure with multiplication 
\be
\mu_{A_1 \otimes A_2} = ( \mu_1 \otimes \mu_2 ) \circ ( 1_{A_1} \otimes c_{A_2,A_1} \otimes 1_{A_2}) 
= 
\begin{tikzpicture}[very thick,scale=0.75,color=blue!50!black, baseline]
\draw[color=green!50!black] (-1,-1) node[below] (A1) {{\scriptsize$A_1$}};
\draw[color=green!50!black] (0,-1) node[below] (A2) {{\scriptsize$A_2$}};
\draw[color=green!50!black] (1,-1) node[below] (A1r) {{\scriptsize$A_1$}};
\draw[color=green!50!black] (2,-1) node[below] (A2r) {{\scriptsize$A_2$}};
\draw[color=green!50!black] (0,1) node[above] (A1up) {{\scriptsize$A_1$}};
\draw[color=green!50!black] (1,1) node[above] (A2up) {{\scriptsize$A_2$}};
\fill[color=green!50!black] (0,0) circle (2.5pt) node (mult1) {};
\fill[color=green!50!black] (1,0) circle (2.5pt) node (mult2) {};
\draw[color=green!50!black] (A1) -- (0,0);
\draw[color=green!50!black] (A1r) -- (0,0);
\draw[color=white, line width=4pt] (A2) -- (1,0);
\draw[color=green!50!black] (A2) -- (1,0);
\draw[color=green!50!black] (A2r) -- (1,0);
\fill[color=green!50!black] (1,0) circle (2.5pt) node (mult2) {};
\draw[color=green!50!black] (0,0) -- (A1up);
\draw[color=green!50!black] (1,0) -- (A2up);
\end{tikzpicture} 
\ee
and unit $\eta_{A_1 \otimes A_2} = \eta_1 \otimes \eta_2$.
The use of the braiding instead of the inverse braiding is again a convention.

We are interested in modules over iterated tensor products $A_1\otimes \dots \otimes A_n$ of algebras $A_{i}$. 
It will be convenient to describe such modules in terms of compatible module structures over the individual factors $A_i$. 
The basic relation is as follows. 

\begin{lemma}
\label{lem:A1A2}
Let $A_1, A_2 \in \mathcal C$ be algebras and $M\in \mathcal C$. 
There is a 1-to-1 correspondence between 
\begin{enumerate}
\item
$A_1 \otimes A_2$-module structures on~$M$, and 
\item 
pairs of $A_1$- and $A_2$ module structures on~$M$ such that one of the following equivalent conditions holds: 
\be
\label{eq:A1A2comp}
\begin{tikzpicture}[very thick,scale=0.75,color=blue!50!black, baseline]
\draw (0,-1) node[below] (X) {{\scriptsize$M$}};
\draw[color=green!50!black] (-0.5,-1) node[below] (A1) {{\scriptsize$A_2$}};
\draw[color=green!50!black] (-1,-1) node[below] (A2) {{\scriptsize$A_1$}};
\draw (0,1) node[right] (Xu) {};
\draw[color=green!50!black] (A2) .. controls +(0,0.5) and +(-0.25,-0.25) .. (0,-0.25);
\draw[color=white, line width=4pt] (A1) .. controls +(0,0.5) and +(-0.5,-0.5) .. (0,0.6);
\draw[color=green!50!black] (A1) .. controls +(0,0.5) and +(-0.5,-0.5) .. (0,0.6);
\draw (0,-1) -- (0,1); 
\fill[color=blue!50!black] (0,-0.25) circle (2.9pt) node (meet) {};
\fill[color=blue!50!black] (0,0.6) circle (2.9pt) node (meet2) {};
\end{tikzpicture} 
= 
\begin{tikzpicture}[very thick,scale=0.75,color=blue!50!black, baseline]
\draw (0,-1) node[below] (X) {{\scriptsize$M$}};
\draw[color=green!50!black] (-0.5,-1) node[below] (X) {{\scriptsize$A_2$}};
\draw[color=green!50!black] (-1,-1) node[below] (X) {{\scriptsize$A_1$}};
\draw (0,1) node[right] (Xu) {};
\draw (0,-1) -- (0,1); 
\draw[color=green!50!black] (-0.5,-1) .. controls +(0,0.25) and +(-0.25,-0.25) .. (0,-0.25);
\draw[color=green!50!black] (-1,-1) .. controls +(0,0.5) and +(-0.5,-0.5) .. (0,0.6);
\fill[color=blue!50!black] (0,-0.25) circle (2.9pt) node (meet) {};
\fill[color=blue!50!black] (0,0.6) circle (2.9pt) node (meet2) {};
\end{tikzpicture} 
\, , \qquad 
\begin{tikzpicture}[very thick,scale=0.75,color=blue!50!black, baseline]
\draw (0,-1) node[below] (X) {{\scriptsize$M$}};
\draw[color=green!50!black] (-0.5,-1) node[below] (A1) {{\scriptsize$A_1$}};
\draw[color=green!50!black] (-1,-1) node[below] (A2) {{\scriptsize$A_2$}};
\draw (0,1) node[right] (Xu) {};
\draw[color=green!50!black] (A1) .. controls +(0,0.5) and +(-0.5,-0.5) .. (0,0.6);
\draw[color=white, line width=4pt] (A2) .. controls +(0,0.5) and +(-0.25,-0.25) .. (0,-0.25);
\draw[color=green!50!black] (A2) .. controls +(0,0.5) and +(-0.25,-0.25) .. (0,-0.25);
\draw (0,-1) -- (0,1); 
\fill[color=blue!50!black] (0,-0.25) circle (2.9pt) node (meet) {};
\fill[color=blue!50!black] (0,0.6) circle (2.9pt) node (meet2) {};
\end{tikzpicture} 
= 
\begin{tikzpicture}[very thick,scale=0.75,color=blue!50!black, baseline]
\draw (0,-1) node[below] (X) {{\scriptsize$M$}};
\draw[color=green!50!black] (-0.5,-1) node[below] (X) {{\scriptsize$A_1$}};
\draw[color=green!50!black] (-1,-1) node[below] (X) {{\scriptsize$A_2$}};
\draw (0,1) node[right] (Xu) {};
\draw (0,-1) -- (0,1); 
\draw[color=green!50!black] (-0.5,-1) .. controls +(0,0.25) and +(-0.25,-0.25) .. (0,-0.25);
\draw[color=green!50!black] (-1,-1) .. controls +(0,0.5) and +(-0.5,-0.5) .. (0,0.6);
\fill[color=blue!50!black] (0,-0.25) circle (2.9pt) node (meet) {};
\fill[color=blue!50!black] (0,0.6) circle (2.9pt) node (meet2) {};
\end{tikzpicture} 
\, .
\ee
\end{enumerate}
\end{lemma}
One can remember the conditions in \eqref{eq:A1A2comp} as ``$A_2$ always passes over $A_1$''.

\begin{proof}
Using the braiding and inverse braiding, respectively, it follows immediately that the two conditions in \eqref{eq:A1A2comp} are equivalent. 

(i) $\Rightarrow$ (ii): 
For an $A_1 \otimes A_2$-module structure  
\be
\label{eq:A1A2act}
\begin{tikzpicture}[very thick,scale=0.75,color=blue!50!black, baseline]
\draw (0,-1) node[below] (X) {{\scriptsize$M$}};
\draw[color=green!50!black] (-1,-1) node[below] (A1) {{\scriptsize$A_1$}};
\draw[color=green!50!black] (-0.5,-1) node[below] (A2) {{\scriptsize$A_2$}};
\draw (0,1) node[right] (Xu) {};
\draw (0,-1) -- (0,1); 
\draw[color=green!50!black] (-1,-1) .. controls +(0,0.5) and +(-0.5,-0.5) .. (0,0.55);
\draw[color=green!50!black] (-0.5,-1) .. controls +(0,0.25) and +(-0.25,-0.25) .. (0,0.45);
\fill[color=blue!50!black] (0,0.5) circle (4pt) node (meet2) {};
\end{tikzpicture} 
\colon A_1 \otimes A_2 \otimes M \lra M
\ee
we set 
\be
\label{eq:A1A2act2}
\begin{tikzpicture}[very thick,scale=0.75,color=blue!50!black, baseline]
\draw (0,-1) node[below] (X) {};
\draw[color=green!50!black] (-1,-1) node[below] (A1) {};
\draw[color=green!50!black] (-0.5,-1) node[below] (A2) {};
\draw (0,1) node[right] (Xu) {};
\draw (0,-1) -- (0,1); 
\draw[color=green!50!black] (-1,-1) .. controls +(0,0.5) and +(-0.5,-0.5) .. (0,0.5);
\draw (0,-1) node[below] (X) {{\scriptsize$M$}};
\draw[color=green!50!black] (-1,-1) node[below] (A1) {{\scriptsize$A_1$}};
\fill[color=blue!50!black] (0,0.5) circle (2.9pt) node (meet2) {};
\end{tikzpicture} 
\stackrel{\textrm{def}}{=}
\begin{tikzpicture}[very thick,scale=0.75,color=blue!50!black, baseline]
\draw (0,-1) node[below] (X) {};
\draw[color=green!50!black] (-1,-1) node[below] (A1) {};
\draw[color=green!50!black] (-0.5,-1) node[below] (A2) {};
\draw (0,1) node[right] (Xu) {};
\draw (0,-1) -- (0,1); 
\draw[color=green!50!black] (-1,-1) .. controls +(0,0.5) and +(-0.5,-0.5) .. (0,0.55);
\draw[color=green!50!black] (-0.35,-0.5) node[Odot] (D) {}; 
\draw[color=green!50!black] (D) .. controls +(0,0.25) and +(-0.25,-0.25) .. (0,0.45);
\draw (0,-1) node[below] (X) {{\scriptsize$M$}};
\draw[color=green!50!black] (-1,-1) node[below] (A1) {{\scriptsize$A_1$}};
\fill[color=blue!50!black] (0,0.5) circle (4pt) node (meet2) {};
\end{tikzpicture} 
\, ,
\quad 
\begin{tikzpicture}[very thick,scale=0.75,color=blue!50!black, baseline]
\draw (0,-1) node[below] (X) {};
\draw[color=green!50!black] (-1,-1) node[below] (A1) {};
\draw[color=green!50!black] (-0.5,-1) node[below] (A2) {};
\draw (0,1) node[right] (Xu) {};
\draw (0,-1) -- (0,1); 
\draw[color=green!50!black] (-1,-1) .. controls +(0,0.5) and +(-0.5,-0.5) .. (0,0.5);
\draw (0,-1) node[below] (X) {{\scriptsize$M$}};
\draw[color=green!50!black] (-1,-1) node[below] (A1) {{\scriptsize$A_2$}};
\fill[color=blue!50!black] (0,0.5) circle (2.9pt) node (meet2) {};
\end{tikzpicture} 
\stackrel{\textrm{def}}{=}
\begin{tikzpicture}[very thick,scale=0.75,color=blue!50!black, baseline]
\draw (0,-1) node[below] (X) {{\scriptsize$M$}};
\draw[color=green!50!black] (-0.5,-1) node[below] (A2) {{\scriptsize$A_2$}};
\draw (0,1) node[right] (Xu) {};
\draw (0,-1) -- (0,1); 
\draw[color=green!50!black] (-0.75,-0.5) node[Odot] (D) {}; 
\draw[color=green!50!black] (D) .. controls +(0,0.25) and +(-0.5,-0.5) .. (0,0.55);
\draw[color=green!50!black] (-0.5,-1) .. controls +(0,0.25) and +(-0.25,-0.25) .. (0,0.45);
\fill[color=blue!50!black] (0,0.5) circle (4pt) node (meet2) {};
\end{tikzpicture} 
\, . 
\ee
These are $A_1$- and $A_2$-module structures, respectively, as e.\,g. 
\be
\begin{tikzpicture}[very thick,scale=0.75,color=blue!50!black, baseline]
\draw (0,-1) node[below] (X) {{\scriptsize$M$}};
\draw[color=green!50!black] (-0.5,-1) node[below] (X) {{\scriptsize$A_1$}};
\draw[color=green!50!black] (-1,-1) node[below] (X) {{\scriptsize$A_1$}};
\draw (0,1) node[right] (Xu) {};
\draw (0,-1) -- (0,1); 
\draw[color=green!50!black] (-0.5,-1) .. controls +(0,0.25) and +(-0.25,-0.25) .. (0,-0.25);
\draw[color=green!50!black] (-1,-1) .. controls +(0,0.5) and +(-0.5,-0.5) .. (0,0.6);
\fill[color=blue!50!black] (0,-0.25) circle (2.9pt) node (meet) {};
\fill[color=blue!50!black] (0,0.6) circle (2.9pt) node (meet2) {};
\end{tikzpicture} 
= 
\begin{tikzpicture}[very thick,scale=0.75,color=blue!50!black, baseline]
\draw (0,1) node[right] (Xu) {};
\draw (0,-1) -- (0,1); 
\draw[color=green!50!black] (-0.5,-1) .. controls +(0,0.25) and +(-0.25,-0.25) .. (0,-0.2);
\draw[color=green!50!black] (-0.2,-0.8) node[Odot] (D1) {}; 
\draw[color=green!50!black] (D1) .. controls +(0.1,0.3) and +(-0.2,-0.2) .. (0,-0.3);
\draw[color=green!50!black] (-1.25,-1) .. controls +(0,0.5) and +(-0.5,-0.5) .. (0,0.65);
\draw[color=green!50!black] (-0.4,-0.1) node[Odot] (D2) {};
\draw[color=green!50!black] (D2) .. controls +(0.1,0.3) and +(-0.2,-0.2) .. (0,0.55);
\fill[color=blue!50!black] (0,-0.25) circle (4pt) node (meet) {};
\fill[color=blue!50!black] (0,0.6) circle (4pt) node (meet2) {};
\end{tikzpicture} 
= 
\begin{tikzpicture}[very thick,scale=0.75,color=blue!50!black, baseline]
\draw (0,1) node[right] (Xu) {};
\draw (0,-1) -- (0,1); 
\fill[color=green!50!black] (-0.9,-0.3) circle (2.5pt) node (mult1) {};
\fill[color=green!50!black] (-0.45,-0.3) circle (2.5pt) node (mult2) {};
\draw[color=green!50!black] (-0.5,-1) .. controls +(0,0.25) and +(0.25,-0.25) .. (-0.9,-0.3);
\draw[color=green!50!black] (-1.25,-1) .. controls +(0,0.5) and +(0,0) .. (-0.9,-0.3);
\draw[color=green!50!black] (-0.9,-0.3) .. controls +(0,0.5) and +(0,0) .. (0,0.6);
\draw[color=green!50!black] (-0.75,-0.8) node[Odot] (D2) {}; 
\draw[color=white, line width=4pt] (D2) .. controls +(0.1,0.2) and +(-0.2,-0.2) .. (-0.45,-0.3);
\fill[color=green!50!black] (-0.45,-0.3) circle (2.5pt) node (mult2) {};
\draw[color=green!50!black] (D2) .. controls +(0.1,0.2) and +(-0.2,-0.2) .. (-0.45,-0.3);
\draw[color=green!50!black] (-0.2,-0.8) node[Odot] (D3) {}; 
\draw[color=green!50!black] (D3) .. controls +(-0.1,0.2) and +(0.2,-0.2) .. (-0.45,-0.3);
\draw[color=green!50!black] (-0.45,-0.3) .. controls +(0,0.5) and +(0,0) .. (0,0.55);
\fill[color=blue!50!black] (0,0.6) circle (4pt) node (meet2) {};
\end{tikzpicture} 
= 
\begin{tikzpicture}[very thick,scale=0.75,color=blue!50!black, baseline]
\draw (0,1) node[right] (Xu) {};
\draw (0,-1) -- (0,1); 
\fill[color=green!50!black] (-0.9,-0.3) circle (2.5pt) node (mult1) {};
\draw[color=green!50!black] (-0.45,-0.3) node[Odot] (D2) {}; 
\draw[color=green!50!black] (-0.5,-1) .. controls +(0,0.25) and +(0.25,-0.25) .. (-0.9,-0.3);
\draw[color=green!50!black] (-1.25,-1) .. controls +(0,0.5) and +(0,0) .. (-0.9,-0.3);
\draw[color=green!50!black] (-0.9,-0.3) .. controls +(0,0.5) and +(0,0) .. (0,0.6);
\draw[color=green!50!black] (-0.45,-0.25) .. controls +(0,0.5) and +(0,0) .. (0,0.55);
\fill[color=blue!50!black] (0,0.6) circle (4pt) node (meet2) {};

\end{tikzpicture} 
= 
\begin{tikzpicture}[very thick,scale=0.75,color=blue!50!black, baseline]
\draw (0,1) node[right] (Xu) {};
\draw (0,-1) -- (0,1); 
\draw[color=green!50!black] (-0.5,-1) .. controls +(0,0.25) and +(0.25,-0.25) .. (-0.9,-0.3);
\draw[color=green!50!black] (-1.25,-1) .. controls +(0,0.5) and +(0,0) .. (-0.9,-0.3);
\draw[color=green!50!black] (-0.9,-0.3) .. controls +(0,0.5) and +(0,0) .. (0,0.6);
\draw (0,-1) node[below] (X) {{\scriptsize$M$}};
\draw[color=green!50!black] (-0.5,-1) node[below] (X) {{\scriptsize$A_1$}};
\draw[color=green!50!black] (-1.2,-1) node[below] (X) {{\scriptsize$A_1$}};
\fill[color=green!50!black] (-0.9,-0.3) circle (2.5pt) node (mult1) {};
\fill[color=blue!50!black] (0,0.6) circle (2.9pt) node (meet2) {};
\end{tikzpicture} 
\, . 
\ee
As direct consequences of the definition we get
\be\label{eq:A1A2-resolve}
\begin{tikzpicture}[very thick,scale=0.75,color=blue!50!black, baseline]
\draw (0,-1) node[below] (X) {{\scriptsize$M$}};
\draw[color=green!50!black] (-0.5,-1) node[below] (X) {{\scriptsize$A_2$}};
\draw[color=green!50!black] (-1,-1) node[below] (X) {{\scriptsize$A_1$}};
\draw (0,1) node[right] (Xu) {};
\draw (0,-1) -- (0,1); 
\draw[color=green!50!black] (-0.5,-1) .. controls +(0,0.25) and +(-0.25,-0.25) .. (0,-0.25);
\draw[color=green!50!black] (-1,-1) .. controls +(0,0.5) and +(-0.5,-0.5) .. (0,0.6);
\fill[color=blue!50!black] (0,-0.25) circle (2.9pt) node (meet) {};
\fill[color=blue!50!black] (0,0.6) circle (2.9pt) node (meet2) {};
\end{tikzpicture} 
= 
\begin{tikzpicture}[very thick,scale=0.75,color=blue!50!black, baseline]
\draw (0,1) node[right] (Xu) {};
\draw (0,-1) -- (0,1); 
\draw[color=green!50!black] (-0.5,-1) .. controls +(0,0.25) and +(-0.25,-0.25) .. (0,-0.3);
\draw[color=green!50!black] (-0.7,-0.8) node[Odot] (D1) {}; 
\draw[color=green!50!black] (D1) .. controls +(0.1,0.3) and +(-0.2,-0.2) .. (0,-0.2);
\draw[color=green!50!black] (-1.25,-1) .. controls +(0,0.5) and +(-0.5,-0.5) .. (0,0.65);
\draw[color=green!50!black] (-0.4,-0.1) node[Odot] (D2) {}; 
\draw[color=green!50!black] (D2) .. controls +(0.1,0.3) and +(-0.2,-0.2) .. (0,0.55);
\fill[color=blue!50!black] (0,-0.25) circle (4pt) node (meet) {};
\fill[color=blue!50!black] (0,0.6) circle (4pt) node (meet2) {};
\end{tikzpicture} 
= 
\begin{tikzpicture}[very thick,scale=0.75,color=blue!50!black, baseline]
\draw (0,1) node[right] (Xu) {};
\draw (0,-1) -- (0,1); 
\fill[color=green!50!black] (-0.9,-0.3) circle (2.5pt) node (mult1) {};
\fill[color=green!50!black] (-0.45,-0.3) circle (2.5pt) node (mult2) {};
\draw[color=green!50!black] (-0.4,-0.8) node[Odot] (D3) {}; 
\draw[color=green!50!black] (D3) .. controls +(-0.1,0.3) and +(0.25,-0.25) .. (-0.9,-0.3);
\draw[color=green!50!black] (-1.25,-1) .. controls +(0,0.5) and +(0,0) .. (-0.9,-0.3);
\draw[color=green!50!black] (-0.9,-0.3) .. controls +(0,0.5) and +(0,0) .. (0,0.6);
\draw[color=green!50!black] (-0.75,-0.8) node[Odot] (D2) {}; 
\draw[color=white, line width=4pt] (D2) .. controls +(0.1,0.2) and +(-0.2,-0.2) .. (-0.45,-0.3);
\fill[color=green!50!black] (-0.45,-0.3) circle (2.5pt) node (mult2) {};
\draw[color=green!50!black] (D2) .. controls +(0.1,0.2) and +(-0.2,-0.2) .. (-0.45,-0.3);
\draw[color=green!50!black] (-0.15,-1) .. controls +(0,0.4) and +(0.2,-0.2) .. (-0.45,-0.3);
\draw[color=green!50!black] (-0.45,-0.3) .. controls +(0,0.5) and +(0,0) .. (0,0.55);
\fill[color=blue!50!black] (0,0.6) circle (4pt) node (meet2) {};
\end{tikzpicture} 
= 
\begin{tikzpicture}[very thick,scale=0.75,color=blue!50!black, baseline]
\draw (0,-1) node[below] (X) {{\scriptsize$M$}};
\draw[color=green!50!black] (-1,-1) node[below] (A1) {{\scriptsize$A_1$}};
\draw[color=green!50!black] (-0.5,-1) node[below] (A2) {{\scriptsize$A_2$}};
\draw (0,1) node[right] (Xu) {};
\draw (0,-1) -- (0,1); 
\draw[color=green!50!black] (-1,-1) .. controls +(0,0.5) and +(-0.5,-0.5) .. (0,0.55);
\draw[color=green!50!black] (-0.5,-1) .. controls +(0,0.25) and +(-0.25,-0.25) .. (0,0.45);
\fill[color=blue!50!black] (0,0.6) circle (4pt) node (meet2) {};
\end{tikzpicture} 
\, ,
\ee
and
\be
\begin{tikzpicture}[very thick,scale=0.75,color=blue!50!black, baseline]
\draw (0,-1) node[below] (X) {{\scriptsize$M$}};
\draw[color=green!50!black] (-0.5,-1) node[below] (X) {{\scriptsize$A_1$}};
\draw[color=green!50!black] (-1,-1) node[below] (X) {{\scriptsize$A_2$}};
\draw (0,1) node[right] (Xu) {};
\draw (0,-1) -- (0,1); 
\draw[color=green!50!black] (-0.5,-1) .. controls +(0,0.25) and +(-0.25,-0.25) .. (0,-0.25);
\draw[color=green!50!black] (-1,-1) .. controls +(0,0.5) and +(-0.5,-0.5) .. (0,0.6);
\fill[color=blue!50!black] (0,-0.25) circle (2.9pt) node (meet) {};
\fill[color=blue!50!black] (0,0.6) circle (2.9pt) node (meet2) {};
\end{tikzpicture} 
\overset{\eqref{eq:A1A2act2}}= 
\begin{tikzpicture}[very thick,scale=0.75,color=blue!50!black, baseline]
\draw (0,1) node[right] (Xu) {};
\draw (0,-1) -- (0,1); 
\draw[color=green!50!black] (-0.5,-1) .. controls +(0,0.25) and +(-0.25,-0.25) .. (0,-0.2);
\draw[color=green!50!black] (-0.2,-0.8) node[Odot] (D1) {}; 
\draw[color=green!50!black] (D1) .. controls +(0.1,0.3) and +(-0.2,-0.2) .. (0,-0.3);
\draw[color=green!50!black] (-1.25,-1) .. controls +(0,0.5) and +(-0.5,-0.5) .. (0,0.55);
\draw[color=green!50!black] (-0.9,0) node[Odot] (D2) {}; 
\draw[color=green!50!black] (D2) .. controls +(0.1,0.3) and +(-0.2,-0.2) .. (0,0.65);
\fill[color=blue!50!black] (0,-0.25) circle (4pt) node (meet) {};
\fill[color=blue!50!black] (0,0.6) circle (4pt) node (meet2) {};
\end{tikzpicture} 
= 
\begin{tikzpicture}[very thick,scale=0.75,color=blue!50!black, baseline]
\draw (0,1) node[right] (Xu) {};
\draw (0,-1) -- (0,1); 
\fill[color=green!50!black] (-0.9,-0.3) circle (2.5pt) node (mult1) {};
\fill[color=green!50!black] (-0.45,-0.3) circle (2.5pt) node (mult2) {};
\draw[color=green!50!black] (-1.3,-0.8) node[Odot] (D1) {}; 
\draw[color=green!50!black] (D1) .. controls +(0,0.25) and +(0,0) .. (-0.9,-0.3);
\draw[color=green!50!black] (-0.4,-1) .. controls +(0,0.25) and +(0.25,-0.25) .. (-0.9,-0.3);
\draw[color=green!50!black] (-0.9,-0.3) .. controls +(0,0.5) and +(0,0) .. (0,0.6);
\fill[color=green!50!black] (-0.45,-0.3) circle (2.5pt) node (mult2) {};
\draw[color=white, line width=4pt] (-0.9,-1) .. controls +(0,0.25) and +(-0.2,-0.2) .. (-0.45,-0.3);
\fill[color=green!50!black] (-0.45,-0.3) circle (2.5pt) node (mult2) {};
\draw[color=green!50!black] (-0.9,-1) .. controls +(0,0.25) and +(-0.2,-0.2) .. (-0.45,-0.3);
\draw[color=green!50!black] (-0.2,-0.8) node[Odot] (D3) {}; 
\draw[color=green!50!black] (D3) .. controls +(-0.1,0.2) and +(0.2,-0.2) .. (-0.45,-0.3);
\draw[color=green!50!black] (-0.45,-0.3) .. controls +(0,0.5) and +(0,0) .. (0,0.55);
\fill[color=blue!50!black] (0,0.6) circle (4pt) node (meet2) {};
\end{tikzpicture} 
= 
\begin{tikzpicture}[very thick,scale=0.75,color=blue!50!black, baseline]
\draw (0,1) node[right] (Xu) {};
\draw (0,-1) -- (0,1); 
\draw[color=green!50!black] (-0.4,-1) .. controls +(0,1) and +(-0.3,0) .. (0,0.65);
\draw[color=white, line width=4pt] (-0.9,-1) .. controls +(0,0.5) and +(-0.2,-0.75) .. (0,0.65);
\draw[color=green!50!black] (-0.9,-1) .. controls +(0,0.5) and +(-0.2,-0.75) .. (0,0.65);
\draw (0,-1) -- (0,1); 
\fill[color=green!50!black] (0,0.6) circle (4pt) node (meet2) {};
\fill[color=blue!50!black] (0,0.6) circle (4pt) node (meet2) {};
\end{tikzpicture} 
\overset{\eqref{eq:A1A2-resolve}}{=} 
\begin{tikzpicture}[very thick,scale=0.75,color=blue!50!black, baseline]
\draw (0,-1) node[below] (X) {{\scriptsize$M$}};
\draw[color=green!50!black] (-0.5,-1) node[below] (A1) {{\scriptsize$A_1$}};
\draw[color=green!50!black] (-1,-1) node[below] (A2) {{\scriptsize$A_2$}};
\draw (0,1) node[right] (Xu) {};
\draw[color=green!50!black] (A1) .. controls +(0,0.5) and +(-0.5,-0.5) .. (0,0.6);
\draw[color=white, line width=4pt] (A2) .. controls +(0,0.5) and +(-0.25,-0.25) .. (0,-0.25);
\draw[color=green!50!black] (A2) .. controls +(0,0.5) and +(-0.25,-0.25) .. (0,-0.25);
\draw (0,-1) -- (0,1); 
\fill[color=blue!50!black] (0,-0.25) circle (2.9pt) node (meet) {};
\fill[color=blue!50!black] (0,0.6) circle (2.9pt) node (meet2) {};
\end{tikzpicture} 
\, .
\ee

(ii) $\Rightarrow$ (i): 
Given actions of $A_1$ and $A_2$ on $M$ which satisfy \eqref{eq:A1A2comp} we define an action of $A_1 \otimes A_2$ via
\be
\label{eq:A1A2act3}
\begin{tikzpicture}[very thick,scale=0.75,color=blue!50!black, baseline]
\draw (0,-1) node[below] (X) {{\scriptsize$M$}};
\draw[color=green!50!black] (-1,-1) node[below] (A1) {{\scriptsize$A_1$}};
\draw[color=green!50!black] (-0.5,-1) node[below] (A2) {{\scriptsize$A_2$}};
\draw (0,1) node[right] (Xu) {};
\draw (0,-1) -- (0,1); 
\draw[color=green!50!black] (-1,-1) .. controls +(0,0.5) and +(-0.5,-0.5) .. (0,0.55);
\draw[color=green!50!black] (-0.5,-1) .. controls +(0,0.25) and +(-0.25,-0.25) .. (0,0.45);
\fill[color=blue!50!black] (0,0.5) circle (4pt) node (meet2) {};
\end{tikzpicture} 
\stackrel{\textrm{def}}{=}
\begin{tikzpicture}[very thick,scale=0.75,color=blue!50!black, baseline]
\draw (0,-1) node[below] (X) {{\scriptsize$M$}};
\draw[color=green!50!black] (-0.5,-1) node[below] (X) {{\scriptsize$A_2$}};
\draw[color=green!50!black] (-1,-1) node[below] (X) {{\scriptsize$A_1$}};
\draw (0,1) node[right] (Xu) {};
\draw (0,-1) -- (0,1); 
\draw[color=green!50!black] (-0.5,-1) .. controls +(0,0.25) and +(-0.25,-0.25) .. (0,-0.25);
\draw[color=green!50!black] (-1,-1) .. controls +(0,0.5) and +(-0.5,-0.5) .. (0,0.6);
\fill[color=blue!50!black] (0,-0.25) circle (2.9pt) node (meet) {};
\fill[color=blue!50!black] (0,0.6) circle (2.9pt) node (meet2) {};
\end{tikzpicture} 
\, .
\ee
This is indeed an $A_1 \otimes A_2$-action: The unit property is clear and for associativity one computes
\be
\begin{tikzpicture}[very thick,scale=0.75,color=blue!50!black, baseline]
\draw (0,1) node[right] (Xu) {};
\draw (0,-1) -- (0,1); 
\fill[color=green!50!black] (-0.9,-0.3) circle (2.5pt) node (mult1) {};
\fill[color=green!50!black] (-0.45,-0.3) circle (2.5pt) node (mult2) {};
\draw[color=green!50!black] (-0.4,-1) .. controls +(-0,0.3) and +(0.25,-0.25) .. (-0.9,-0.3);
\draw[color=green!50!black] (-1.25,-1) .. controls +(0,0.5) and +(0,0) .. (-0.9,-0.3);
\draw[color=green!50!black] (-0.9,-0.3) .. controls +(0,0.5) and +(0,0) .. (0,0.6);
\draw[color=white, line width=4pt] (-0.75,-1) .. controls +(0,0.3) and +(-0.2,-0.2) .. (-0.45,-0.3);
\fill[color=green!50!black] (-0.45,-0.3) circle (2.5pt) node (mult2) {};
\draw[color=green!50!black] (-0.75,-1) .. controls +(0,0.3) and +(-0.2,-0.2) .. (-0.45,-0.3);
\draw[color=green!50!black] (-0.15,-1) .. controls +(0,0.4) and +(0.2,-0.2) .. (-0.45,-0.3);
\draw[color=green!50!black] (-0.45,-0.3) .. controls +(0,0.5) and +(0,0) .. (0,0.55);
\fill[color=blue!50!black] (0,0.6) circle (4pt) node (meet2) {};
\end{tikzpicture} 
= 
\begin{tikzpicture}[very thick,scale=0.75,color=blue!50!black, baseline]
\draw (0,-1) node[below] (X) {};
\draw[color=green!50!black] (-0.5,-1) node[below] (A1) {};
\draw[color=green!50!black] (-1,-1) node[below] (A2) {};
\draw[color=green!50!black] (-1.2,-1) node[below] (A3) {};
\draw[color=green!50!black] (-0.2,-1) node[below] (A4) {};
\draw (0,1) node[right] (Xu) {};
\draw[color=green!50!black] (A1) .. controls +(0,0.5) and +(-0.5,-0.5) .. (0,0.5);
\draw[color=white, line width=4pt] (A2) .. controls +(0,0.5) and +(-0.25,-0.25) .. (0,-0.25);
\draw[color=green!50!black] (A2) .. controls +(0,0.5) and +(-0.25,-0.25) .. (0,-0.25);
\draw (0,-1) -- (0,1); 
\draw[color=green!50!black] (A3) .. controls +(0,0.5) and +(-0.5,-0.5) .. (0,0.8);
\draw[color=green!50!black] (A4) .. controls +(0,0.25) and +(-0.1,-0.1) .. (0,-0.6);
\fill[color=blue!50!black] (0,-0.25) circle (2.9pt) node (meet) {};
\fill[color=blue!50!black] (0,0.5) circle (2.9pt) node (meet2) {};
\fill[color=blue!50!black] (0,0.8) circle (2.9pt) node (meet3) {};
\fill[color=blue!50!black] (0,-0.6) circle (2.9pt) node (meet4) {};
\end{tikzpicture} 
\overset{\eqref{eq:A1A2comp}}= 
\begin{tikzpicture}[very thick,scale=0.75,color=blue!50!black, baseline]
\draw (0,-1) node[below] (X) {};
\draw[color=green!50!black] (-0.5,-1) node[below] (A1) {};
\draw[color=green!50!black] (-1,-1) node[below] (A2) {};
\draw[color=green!50!black] (-1.2,-1) node[below] (A3) {};
\draw[color=green!50!black] (-0.2,-1) node[below] (A4) {};
\draw (0,1) node[right] (Xu) {};
\draw[color=green!50!black] (A1) .. controls +(0,0.25) and +(-0.25,-0.25) .. (0,-0.25);
\draw[color=green!50!black] (A2) .. controls +(0,0.5) and +(-0.25,-0.25) .. (0,0.5);
\draw (0,-1) -- (0,1); 
\draw[color=green!50!black] (A3) .. controls +(0,0.5) and +(-0.5,-0.5) .. (0,0.8);
\draw[color=green!50!black] (A4) .. controls +(0,0.25) and +(-0.1,-0.1) .. (0,-0.6);
\fill[color=blue!50!black] (0,-0.25) circle (2.9pt) node (meet) {};
\fill[color=blue!50!black] (0,0.5) circle (2.9pt) node (meet2) {};
\fill[color=blue!50!black] (0,0.8) circle (2.9pt) node (meet3) {};
\fill[color=blue!50!black] (0,-0.6) circle (2.9pt) node (meet4) {};
\end{tikzpicture} 
= 
\begin{tikzpicture}[very thick,scale=0.75,color=blue!50!black, baseline]
\draw (0,-1) node[below] (X) {};
\draw[color=green!50!black] (-0.5,-1) node[below] (A1) {};
\draw[color=green!50!black] (-1,-1) node[below] (A2) {};
\draw[color=green!50!black] (-1.2,-1) node[below] (A3) {};
\draw[color=green!50!black] (-0.2,-1) node[below] (A4) {};
\draw (0,1) node[right] (Xu) {};
\draw[color=green!50!black] (A1) .. controls +(0,0.25) and +(-0.25,-0.25) .. (0,-0.25);
\draw[color=green!50!black] (A2) .. controls +(0,0.5) and +(-0.25,-0.25) .. (0,0.55);
\draw (0,-1) -- (0,1); 
\draw[color=green!50!black] (A3) .. controls +(0,0.5) and +(-0.5,-0.5) .. (0,0.65);
\draw[color=green!50!black] (A4) .. controls +(0,0.25) and +(-0.1,-0.1) .. (0,-0.35);
\fill[color=blue!50!black] (0,-0.3) circle (4pt) node (meet) {};
\fill[color=blue!50!black] (0,0.6) circle (4pt) node (meet2) {};
\end{tikzpicture} 
 .
\ee
If the $A_i$-actions are those in \eqref{eq:A1A2act2}, then  the calculation in \eqref{eq:A1A2-resolve} shows that \eqref{eq:A1A2act3} recovers \eqref{eq:A1A2act}. 
\end{proof}

We can iterate the characterisation of Lemma~\ref{lem:A1A2} for any number $n\in \Z_+$ of algebras $A_1, \dots, A_n$ and $A_1 \otimes \dots \otimes A_n$-modules. 
To stress the importance of the individual factors~$A_i$ we will use a special name for such modules: 

\begin{definitionlemma}
Let $A_1, \dots, A_n \in \mathcal C$ be algebras. 
A \textsl{multi-module} for the list $(A_1, \dots, A_n)$ is an $A_1 \otimes \dots \otimes A_n$-module. 
Equivalently, a multi-module for  $(A_1, \dots, A_n)$ is 
\begin{itemize}
\item
an object $M \in \mathcal C$ with an $A_i$-module structure for all $i \in \{ 1,\dots, n \}$, such that
\item 
for all $i<j$ we have 
\be
\begin{tikzpicture}[very thick,scale=0.75,color=blue!50!black, baseline]
\draw (0,-1) node[below] (X) {{\scriptsize$M$}};
\draw[color=green!50!black] (-0.5,-1) node[below] (A1) {{\scriptsize$A_j$}};
\draw[color=green!50!black] (-1,-1) node[below] (A2) {{\scriptsize$A_i$}};
\draw (0,1) node[right] (Xu) {};
\draw[color=green!50!black] (A2) .. controls +(0,0.5) and +(-0.25,-0.25) .. (0,-0.25);
\draw[color=white, line width=4pt] (A1) .. controls +(0,0.5) and +(-0.5,-0.5) .. (0,0.6);
\draw[color=green!50!black] (A1) .. controls +(0,0.5) and +(-0.5,-0.5) .. (0,0.6);
\draw (0,-1) -- (0,1); 
\fill[color=blue!50!black] (0,-0.25) circle (2.9pt) node (meet) {};
\fill[color=blue!50!black] (0,0.6) circle (2.9pt) node (meet2) {};
\end{tikzpicture} 
= 
\begin{tikzpicture}[very thick,scale=0.75,color=blue!50!black, baseline]
\draw (0,-1) node[below] (X) {{\scriptsize$M$}};
\draw[color=green!50!black] (-0.5,-1) node[below] (X) {{\scriptsize$A_j$}};
\draw[color=green!50!black] (-1,-1) node[below] (X) {{\scriptsize$A_i$}};
\draw (0,1) node[right] (Xu) {};
\draw (0,-1) -- (0,1); 
\draw[color=green!50!black] (-0.5,-1) .. controls +(0,0.25) and +(-0.25,-0.25) .. (0,-0.25);
\draw[color=green!50!black] (-1,-1) .. controls +(0,0.5) and +(-0.5,-0.5) .. (0,0.6);
\fill[color=blue!50!black] (0,-0.25) circle (2.9pt) node (meet) {};
\fill[color=blue!50!black] (0,0.6) circle (2.9pt) node (meet2) {};
\end{tikzpicture} 
\, . 
\ee
\end{itemize}
\end{definitionlemma}

We will often abbreviate a multi-module as above as ${}_{A_1\dots A_n}M$. 
For another multi-module ${}_{A_1\dots A_n}M'$, the natural notion of structure-preserving map from ${}_{A_1\dots A_n}M$ to ${}_{A_1\dots A_n}M'$ is an $A_1\otimes \dots \otimes A_n$-module map $M \to M'$. 
Equivalently: 

\begin{definition}
\label{definition:mapmulti}
A \textsl{map of multi-modules} from ${}_{A_1\dots A_n}M$ to ${}_{A_1\dots A_n}M'$ is a morphism $\varphi \in \Hom_{\mathcal C}(M,M')$ such that 
\be
\begin{tikzpicture}[very thick,scale=0.75,color=blue!50!black, baseline]
\draw (0,-1) node[below] (X) {{\scriptsize$M$}};
\draw[color=green!50!black] (-0.5,-1) node[below] (X) {{\scriptsize$A_i$}};
\draw (0,1) node[right] (Xu) {};
\draw (0,-1) -- (0,1); 
\draw[color=green!50!black] (-0.5,-1) .. controls +(0,0.25) and +(-0.25,-0.25) .. (0,-0.25);
\fill[color=black] (0,0.6) circle (2.9pt) node[right] (meet2) {{\scriptsize$\varphi$}};
\fill[color=blue!50!black] (0,-0.25) circle (2.9pt) node (meet) {};
\end{tikzpicture} 
=
\begin{tikzpicture}[very thick,scale=0.75,color=blue!50!black, baseline]
\draw (0,-1) node[below] (X) {{\scriptsize$M$}};
\draw[color=green!50!black] (-0.5,-1) node[below] (X) {{\scriptsize$A_i$}};
\draw (0,1) node[right] (Xu) {};
\draw (0,-1) -- (0,1); 
\fill[color=black] (0,-0.25) circle (2.9pt) node[right] (meet2) {{\scriptsize$\varphi$}};
\draw[color=green!50!black] (-0.5,-1) .. controls +(0,0.5) and +(-0.5,-0.5) .. (0,0.6);
\fill[color=blue!50!black] (0,0.6) circle (2.9pt) node (meet) {};
\end{tikzpicture} 
\quad\text{for all } 
i \in \{1,\dots,n \} 
\, . 
\ee
	The vector space of these maps is denoted  $\Hom_{A_{1} \otimes \ldots \otimes A_{n}}(M,M')$. 

\end{definition}

\begin{example}
\label{ex:Amultmod}
Let $A \in \mathcal C$ be a commutative algebra, that is, an algebra for which $\mu \circ c_{A,A} = \mu$.
Then its multiplication makes ${}_{AA}A$ a multi-module. 
Indeed, we have 
\be\label{eq:comm-A-multi-example}
\begin{tikzpicture}[very thick,scale=0.75,color=blue!50!black, baseline]
\draw (0,-1) node[below] (X) {};
\draw[color=green!50!black] (-0.5,-1) node[below] (X) {};
\draw[color=green!50!black] (-1,-1) node[below] (X) {};
\draw[color=green!50!black] (0,-1) -- (0,1); 
\fill[color=green!50!black] (0,-0.25) circle (2.5pt) node (meet) {};
\fill[color=green!50!black] (0,0.6) circle (2.5pt) node (meet2) {};
\draw[color=green!50!black] (-0.5,-1) .. controls +(0,0.25) and +(-0.25,-0.25) .. (0,-0.25);
\draw[color=green!50!black] (-1,-1) .. controls +(0,0.5) and +(-0.5,-0.5) .. (0,0.6);
\end{tikzpicture} 
=
\begin{tikzpicture}[very thick,scale=0.75,color=blue!50!black, baseline]
\draw[color=green!50!black] (0,-1) -- (0,1); 
\fill[color=green!50!black] (0,0.6) circle (2.5pt) node (meet2) {};
\fill[color=green!50!black] (-0.85,0) circle (2.5pt) node (mult1) {};
\draw[color=green!50!black] (-0.5,-1) .. controls +(0,0.25) and +(0.25,-0.25) .. (-0.85,0);
\draw[color=green!50!black] (-1.2,-1) .. controls +(0,0.5) and +(0,0) .. (-0.85,0);
\draw[color=green!50!black] (-0.85,0) .. controls +(0,0.5) and +(0,0) .. (0,0.6);
\end{tikzpicture} 
=
\begin{tikzpicture}[very thick,scale=0.75,color=blue!50!black, baseline]
\draw[color=green!50!black] (0,-1) -- (0,1); 
\fill[color=green!50!black] (0,0.6) circle (2.5pt) node (meet2) {};
\fill[color=green!50!black] (-0.85,0) circle (2.5pt) node (mult1) {};
\draw[color=green!50!black] (-1.2,-0.3) .. controls +(0,0.25) and +(0,0) .. (-0.85,0);
\draw[color=green!50!black] (-1.2,-0.3) .. controls +(0,-0.25) and +(0,0.25) .. (-0.5,-1);
\draw[color=green!50!black] (-0.5,-0.3) .. controls +(0,0.25) and +(0,0) .. (-0.85,0);
\draw[color=white, line width=4pt] (-0.5,-0.3) .. controls +(0,-0.25) and +(0,0.25) .. (-1.2,-1);
\draw[color=green!50!black] (-0.5,-0.3) .. controls +(0,-0.25) and +(0,0.25) .. (-1.2,-1);
\draw[color=green!50!black] (-0.85,0) .. controls +(0,0.5) and +(0,0) .. (0,0.6);
\end{tikzpicture} 
=
\begin{tikzpicture}[very thick,scale=0.75,color=blue!50!black, baseline]
\fill[color=green!50!black] (0,-1) node[below] (X) {};
\draw[color=green!50!black] (-0.5,-1) node[below] (A1) {};
\draw[color=green!50!black] (-1,-1) node[below] (A2) {};
\draw[color=green!50!black] (A2) .. controls +(0,0.5) and +(-0.25,-0.25) .. (0,-0.25);
\draw[color=white, line width=4pt] (A1) .. controls +(0,0.5) and +(-0.5,-0.5) .. (0,0.6);
\draw[color=green!50!black] (A1) .. controls +(0,0.5) and +(-0.5,-0.5) .. (0,0.6);
\draw[color=green!50!black] (0,-1) -- (0,1); 
\fill[color=green!50!black] (0,-0.25) circle (2.5pt) node (meet) {};
\fill[color=green!50!black] (0,0.6) circle (2.5pt) node (meet2) {};
\end{tikzpicture} 
\ee
so the claim follows by Lemma~\ref{lem:A1A2}. 
In the same way one checks that ${}_{A\dots A}A$ is a multi-module for any number of algebras that act.
In fact, multiplication endows $A$ with an ${}_{AA}A$-multi-module structure if and only if $A$ is commutative, for, if $A$ is an ${}_{AA}A$-multi-module, the left- and right-hand side of \eqref{eq:comm-A-multi-example} are equal. 
Precomposing with $\id_A \otimes \id_A \otimes \eta$ gives $\mu = \mu \circ c_{A,A}$.
\end{example}

\begin{remark}
\label{rem:sym} 
In Section~\ref{sec:RTdef} we will use  multi-modules ${}_{A_1\dots A_n}M$ with certain extra data (a cyclic symmetry) to label line defects with $n$ incident surface defects, cf.\ Figure~\ref{fig:intro-linedef} for the case $n=3$. As illustrated in that figure, there is no total order, only a cyclic order, as one can ``flip'' surface defects from the back to the front
and vice versa. 
This will be modelled algebraically by the following manipulation of ribbon graphs: 
\be
\label{eq:twistmot}
\begin{tikzpicture}[very thick,scale=0.75,color=blue!50!black, baseline]
\draw (0,-1) node[below] (X) {};
\draw[color=green!50!black] (-1,-1) node[below] (A1) {};
\draw[line width=1pt, color=black] (0,0.5) node[inner sep=4pt,draw, rounded corners=1pt] (R) {{\scriptsize$\;\rho\;$}};
\draw[line width=3pt] (0,-1) -- (R); 
\draw[line width=3pt] (0,1.5) -- (R); 
\draw[color=green!50!black, line width=3pt] (A1) .. controls +(0,0.5) and +(0,-0.5) .. (-0.28,0.2);
\draw (0,-1) node[below] (X) {{\scriptsize$M$}};
\draw[color=green!50!black] (-1,-1) node[below] (A1) {{\scriptsize$A$}};
\draw[line width=1pt, color=black] (0,0.5) node[inner sep=4pt,draw, rounded corners=1pt] (R) {{\scriptsize$\;\rho\;$}};
\end{tikzpicture} 
\; \; \;
\stackrel{\text{flip}}{\rightsquigarrow}
\; \; \;
\begin{tikzpicture}[line width=3pt, scale=0.75, color=blue!50!black, baseline]
\fill[pattern=north west lines, opacity=0.3]
	(-1.2,0.575) .. controls +(0.15,0) and +(-0.5,0) .. (0,1.275)
	-- (0,1.125)  .. controls +(-0.5,0) and +(0.15,0) .. (-1.2,0.425);
\draw[color=green!50!black, line width=3pt, opacity=0.5]
	(-1.2,0.5) .. controls +(0.15,0) and +(-0.5,0) .. (0,1.2);
\draw[line width=3pt] (0,-1) -- (R); 
\draw[color=white, line width=5pt] (1.2,0.5) .. controls +(-0.15,0) and +(0.5,0) .. (0,-0.4);
\fill[pattern=north east lines, opacity=0.3]
	(1.2,0.575) .. controls +(-0.15,0) and +(0.5,0) .. (0,1.275)
	-- (0,1.125)  .. controls +(0.5,0) and +(-0.15,0) .. (1.2,0.425);
\draw[color=green!50!black, line width=3pt, opacity=0.5]
	(1.2,0.5) .. controls +(-0.15,0) and +(0.5,0) .. (0,1.2);
\draw[color=green!50!black, line width=3pt] (1.2,0.5) .. controls +(-0.15,0) and +(0.5,0) .. (0,-0.4);
\draw[color=white, line width=5pt] (0,-0.4) .. controls +(-0.15,0) and +(0,0.5) .. (-1,-1);
\draw[color=green!50!black, line width=3pt] (0,-0.4) .. controls +(-0.15,0) and +(0,0.5) .. (-1,-1);
\draw[color=green!50!black, line width=3pt]
	(-0.28,0.2) .. controls +(0,-0.5) and +(0.15,0) .. (-1.2,0.5);
\draw[color=green!50!black] (-1,-1) node[below] (A1) {};
\draw[line width=1pt, color=black] (0,0.5) node[inner sep=4pt,draw, rounded corners=1pt] (R) {{\scriptsize$\;\rho\;$}};
\draw[color=white, line width=5pt] (0,1.5) -- (R); 
\draw[line width=3pt] (0,1.5) -- (R); 
\draw (0,-1) node[below] (X) {{\scriptsize$M$}};
\draw[color=green!50!black] (-1,-1) node[below] (A1) {{\scriptsize$A$}};
\draw[line width=1pt, color=black] (0,0.5) node[inner sep=4pt,draw, rounded corners=1pt] (R) {{\scriptsize$\;\rho\;$}};
\end{tikzpicture}
\; 
\stackrel{\text{deform}}{=} 
\; 
\begin{tikzpicture}[very thick,scale=0.75,color=blue!50!black, looseness=0.5, baseline]
\draw[color=green!50!black] (-1,-1) node[below] (A1) {};
\draw[line width=1pt, color=black] (0,0.5) node[inner sep=4pt,draw, rounded corners=1pt] (R) {{\scriptsize$\;\rho\;$}};
%
\fill (-0.075,-1) 
	-- (-0.075,-0.8) 
	to[out=90, in=180] (0,-0.7)
	to[out=0, in=90] (0.075,-0.8)
	-- (0.075,-1);
\fill [pattern=north west lines, opacity=0.3] (0,-0.7)
	to[out=0, in=270] (0.075,-0.6)
	-- (0.075,-0.5)
	to[out=90, in=0] (0,-0.4)
	to[out=180, in=90] (-0.075,-0.5)
	-- (-0.075,-0.6)
	to[out=270, in=180] (0,-0.7);
\fill[opacity = 0.2] (0,-0.7)
	to[out=0, in=270] (0.075,-0.6)
	-- (0.075,-0.5)
	to[out=90, in=0] (0,-0.4)
	to[out=180, in=90] (-0.075,-0.5)
	-- (-0.075,-0.6)
	to[out=270, in=180] (0,-0.7);
\fill (0,-0.4)
	to[out=180, in=270] (-0.075,-0.3) 
	-- (-0.075,0.2)
	-- (0.075,0.2)
	-- (0.075,-0.3) 
	to[out=270, in=0] (0,-0.4); 
\draw[line width=0.3pt, color=black] (-0.075,-1) 		
							-- (-0.075,-0.8) 
							to[out=90, in=180] (0,-0.7)
							to[out=0, in=270] (0.075,-0.6)
							-- (0.075,-0.5);
\node[inner sep=0.75pt,fill,white] at (0,-0.7) {};
\draw[line width=0.3pt, color=black] (0.075,-1) 		
							-- (0.075,-0.8) 
							to[out=90, in=0] (0,-0.7)
							to[out=180, in=270] (-0.075,-0.6)
							-- (-0.075,-0.5)
							to[out=90, in=180] (0,-0.4)
							to[out=0, in=270] (0.075,-0.3) 
							-- (0.075,0.2);
\node[inner sep=0.75pt,fill,white] at (0,-0.4) {};
\draw[line width=0.3pt, color=black] 	(0.075,-0.5)	
							to[out=90, in=0] (0,-0.4)
							to[out=180, in=270] (-0.075,-0.3) 
							-- (-0.075,0.2);
%
%
\fill (0.075,0.8) 
	-- (0.075,1) 
	to[out=90, in=0] (0,1.1)
	to[out=180, in=90] (-0.075,1.0)
	-- (-0.075,0.8);
\fill [pattern=north west lines, opacity=0.3] (0,1.1)
	to[out=0, in=270] (0.075,1.2)
	-- (0.075,1.3)
	to[out=90, in=0] (0,1.4)
	to[out=180, in=90] (-0.075,1.3)
	-- (-0.075,1.2)
	to[out=270, in=180] (0,1.1); 
\fill[opacity = 0.2] (0,1.1)
	to[out=0, in=270] (0.075,1.2)
	-- (0.075,1.3)
	to[out=90, in=0] (0,1.4)
	to[out=180, in=90] (-0.075,1.3)
	-- (-0.075,1.2)
	to[out=270, in=180] (0,1.1); 
\fill (0,1.4)
	to[out=180, in=270] (-0.075,1.5) 
	-- (-0.075,1.7)
	-- (0.075,1.7)
	-- (0.075,1.5)
	to[out=270, in=0] (0,1.4);
\draw[line width=0.3pt, color=black] (0.075,0.8) 		
							-- (0.075,1) 
							to[out=90, in=0] (0,1.1)
							to[out=180, in=270] (-0.075,1.2)
							-- (-0.075,1.3);
\node[inner sep=0.75pt,fill,white] at (0,1.1) {};
\draw[line width=0.3pt, color=black] (-0.075,0.8) 		
							-- (-0.075,1) 
							to[out=90, in=180] (0,1.1)
							to[out=0, in=270] (0.075,1.2)
							-- (0.075,1.3)
							to[out=90, in=0] (0,1.4)
							to[out=180, in=270] (-0.075,1.5) 
							-- (-0.075,1.7);
\node[inner sep=0.75pt,fill,white] at (0,1.4) {};
\draw[line width=0.3pt, color=black] (-0.075,1.3) 		
							to[out=90, in=180] (0,1.4)
							to[out=0, in=270] (0.075,1.5) 
							-- (0.075,1.7);
%
\draw[color=green!50!black, line width=3pt] (A1) .. controls +(0,0.5) and +(0,-0.5) .. (-0.28,0.2);
\draw (0,-1) node[below] (X) {{\scriptsize$M$}};
\draw[color=green!50!black] (-1,-1) node[below] (A1) {{\scriptsize$A$}};
\draw[line width=1pt, color=black] (0,0.5) node[inner sep=4pt,draw, rounded corners=1pt] (R) {{\scriptsize$\;\rho\;$}};
\end{tikzpicture} 
\ee
The right-hand side amounts to the morphism $A \otimes M \to M$ given by $\theta_M \circ \rho \circ (\id_A \otimes \theta_M^{-1})$ and motivates the next definition.
\end{remark}
 
\begin{definition}
\label{def:twistedmodule}
The \textsl{twist} of an $A$-module $(M,\rho)$ is the $A$-module $(M,\rho^{\textrm{tw}})$ with 
\be
\rho^{\textrm{tw}}
\;
\stackrel{\textrm{def}}{=} 
\begin{tikzpicture}[very thick,scale=0.75,color=blue!50!black, baseline]
\draw (0,-1) node[below] (X) {{\scriptsize$M$}};
\draw[color=green!50!black] (-0.5,-1) node[below] (X) {{\scriptsize$A$}};
\draw (0,-1) -- (0,1); 
\draw[color=green!50!black] (-0.5,-1) .. controls +(0,0.25) and +(-0.25,-0.25) .. (0,0.1);
\fill[color=black] (0,0.1) circle (2.9pt) node[right] (meet) {{\scriptsize$\rho$}};
\fill[color=black] (0,0.6) circle (2.9pt) node[right] (meet2) {{\scriptsize$\theta_M$}};
\fill[color=black] (0,-0.6) circle (2.9pt) node[right] (meet2) {{\scriptsize$\theta^{-1}_M$}};
\end{tikzpicture} 
\, . 
\ee
\end{definition}

It immediately follows from this definition that 
\be\label{eq:thetaiso}
\theta_M \colon (M,\rho) \lra (M,\rho^{\textrm{tw}})
\ee
provides an $A$-module isomorphism between a module and its twist. We also note that via~\eqref{eq:twist} we can rewrite
\be\label{eq:alternative-rhotw}
\rho^{\textrm{tw}}
\; 
=
\begin{tikzpicture}[very thick,scale=0.75,color=blue!50!black, baseline]
\draw (0,-1) node[below] (X) {{\scriptsize$M$}};
\draw[color=green!50!black] (-0.5,-1) node[below] (X) {{\scriptsize$A$}};
\draw[color=green!50!black] (0.3,-0.2) .. controls +(0,0.25) and +(0,-0.25) .. (-0.3,0.2);
\draw[color=white, line width=4pt] (0,-0.2) -- (0,1); 
\draw (0,-1) -- (0,1); 
\draw[color=green!50!black] (-0.3,0.2) .. controls +(0,0.25) and +(0,-0.25) .. (0,0.6);
\draw[color=white, line width=4pt] (0.3,-0.2) .. controls +(0,-0.25) and +(0,0.6) .. (-0.5,-1);
\draw[color=green!50!black] (0.3,-0.2) .. controls +(0,-0.25) and +(0,0.6) .. (-0.5,-1);
%
%
\fill[color=black] (0,0.6) circle (2.9pt) node[right] (meet2) {{\scriptsize$\rho$}};
\fill[color=green!50!black] (-0.37,-0.65) circle (2.5pt) node[left] (meet2) {{\scriptsize$\theta_A$}};
\end{tikzpicture} 
\, .
\ee

\begin{lemma}
Let ${}_{A_1A_2}M$ be a multi-module with $A_i$-action $\rho_i \colon A_i \otimes M \to M$ for $i \in \{ 1,2 \}$. 
Then~$M$ is also a multi-module ${}_{A_2A_1}M^{\textrm{tw}_{1}}$ with $A_2$-action~$\rho_2$ and $A_1$-action $\rho_1^{\textrm{tw}}$. 
\end{lemma}

\begin{proof}
According to Lemma~\ref{lem:A1A2} we have to check that 
\be
\begin{tikzpicture}[very thick,scale=0.75,color=blue!50!black, baseline]
\draw (0,-1) node[below] (X) {{\scriptsize$M$}};
\draw[color=green!50!black] (-0.5,-1) node[below] (X) {{\scriptsize$A_2$}};
\draw[color=green!50!black] (-1,-1) node[below] (X) {{\scriptsize$A_1$}};
\draw (0,1) node[right] (Xu) {};
\draw (0,-1) -- (0,1); 
\draw[color=green!50!black] (-0.5,-1) .. controls +(0,0.25) and +(-0.25,-0.25) .. (0,-0.25);
\draw[color=green!50!black] (-1,-1) .. controls +(0,0.5) and +(-0.5,-0.5) .. (0,0.6);
\fill[color=black] (0,-0.25) circle (2.9pt) node[right] (meet) {{\scriptsize$\rho_2$}};
\fill[color=black] (0,0.6) circle (2.9pt) node[right] (meet2) {{\scriptsize$\rho_1^{\textrm{tw}}$}};
\end{tikzpicture} 
= 
\begin{tikzpicture}[very thick,scale=0.75,color=blue!50!black, baseline]
\draw (0,-1) node[below] (X) {{\scriptsize$M$}};
\draw[color=green!50!black] (-0.5,-1) node[below] (A1) {{\scriptsize$A_2$}};
\draw[color=green!50!black] (-1,-1) node[below] (A2) {{\scriptsize$A_1$}};
\draw (0,1) node[right] (Xu) {};
\draw[color=green!50!black] (A1) .. controls +(0,0.5) and +(-0.5,-0.5) .. (0,0.6);
\draw[color=white, line width=4pt] (A2) .. controls +(0,0.5) and +(-0.25,-0.25) .. (0,-0.25);
\draw[color=green!50!black] (A2) .. controls +(0,0.5) and +(-0.25,-0.25) .. (0,-0.25);
\draw (0,-1) -- (0,1); 
\fill[color=black] (0,-0.25) circle (2.9pt) node[right] (meet) {{\scriptsize$\rho_1^{\textrm{tw}}$}};
\fill[color=black] (0,0.6) circle (2.9pt) node[right] (meet2) {{\scriptsize$\rho_2$}};
\end{tikzpicture} 
\, . 
\ee
Indeed, using the form \eqref{eq:alternative-rhotw} of the twisted action and \eqref{eq:A1A2comp} for the $(A_1,A_2)$-multi-module~$M$ gives
\be
\begin{tikzpicture}[very thick,scale=1,color=blue!50!black, baseline]
\draw (0,-1) node[below] (X) {{\scriptsize$M$}};
\draw[color=green!50!black] (-0.7,-1) node[below] (X) {{\scriptsize$A_1$}};
\draw[color=green!50!black] (-0.35,-1) node[below] (X) {{\scriptsize$A_2$}};
\draw[color=green!50!black] (0.3,-0.2) .. controls +(0,0.25) and +(0,-0.25) .. (-0.3,0.2);
\draw[color=white, line width=4pt] (0,-0.2) -- (0,1); 
\draw (0,-1) -- (0,1); 
\draw[color=green!50!black] (-0.3,0.2) .. controls +(0,0.25) and +(-0.1,-0.25) .. (0,0.6);
\draw[color=white, line width=4pt] (0.3,-0.2) .. controls +(0,-0.25) and +(0,0.6) .. (-0.7,-1);
\draw[color=green!50!black] (0.3,-0.2) .. controls +(0,-0.25) and +(0,0.6) .. (-0.7,-1);
\fill[color=green!50!black] (-0.53,-0.65) circle (2.5pt) node[above] (meet2) {{\scriptsize$\theta_{A_1}$}};
\draw[color=green!50!black] (-0.35,-1) .. controls +(0,0.2) and +(-0.2,-0.2) .. (0,-0.65);
\fill[color=black] (0,-0.65) circle (2.5pt) node[right] (meet2) {{\scriptsize$\rho_2$}};
\fill[color=black] (0,0.6) circle (2.5pt) node[right] (meet2) {{\scriptsize$\rho_1$}};
\end{tikzpicture} 
= 
\begin{tikzpicture}[very thick,scale=1,color=blue!50!black, baseline]
\draw (0,-1) node[below] (X) {{\scriptsize$M$}};
\draw[color=green!50!black] (-0.7,-1) node[below] (X) {{\scriptsize$A_1$}};
\draw[color=green!50!black] (-0.35,-1) node[below] (X) {{\scriptsize$A_2$}};
\draw[color=green!50!black] (0.3,-0.2) .. controls +(0,0.25) and +(0,-0.25) .. (-0.3,0.2);
\draw[color=white, line width=4pt] (0,-0.2) -- (0,1); 
\draw (0,-1) -- (0,1); 
\draw[color=green!50!black] (-0.3,0.2) .. controls +(0,0.25) and +(-0.1,-0.25) .. (0,0.6);
\draw[color=green!50!black] (-0.35,-1) .. controls +(0,0.2) and +(0,-0.2) .. (-0.2,-0.2);
\fill[color=black] (0,0.3) circle (2.5pt) node[right] (meet2) {};
\draw[color=white, line width=4pt] (0.3,-0.2) .. controls +(0,-0.25) and +(0,0.6) .. (-0.7,-1);
\draw[color=green!50!black] (0.3,-0.2) .. controls +(0,-0.25) and +(0,0.6) .. (-0.7,-1);
\draw[color=white, line width=4pt] (-0.2,-0.2) .. controls +(0,0.2) and +(-0.2,-0.2) .. (0,0.3);
\draw[color=green!50!black] (-0.2,-0.2) .. controls +(0,0.2) and +(-0.2,-0.2) .. (0,0.3);
\fill[color=green!50!black] (-0.53,-0.65) circle (2.5pt) node[above] (meet2) {{\scriptsize$\theta_{A_1}$}};
\fill[color=black] (0,0.6) circle (2.5pt) node[right] (meet2) {{\scriptsize$\rho_1$}};
\fill[color=black] (0,0.3) circle (2.5pt) node[right] (meet2) {{\scriptsize$\rho_2$}};
\end{tikzpicture} 
= 
\begin{tikzpicture}[very thick,scale=1,color=blue!50!black, baseline]
\draw (0,-1) node[below] (X) {{\scriptsize$M$}};
\draw[color=green!50!black] (-0.7,-1) node[below] (X) {{\scriptsize$A_1$}};
\draw[color=green!50!black] (-0.35,-1) node[below] (X) {{\scriptsize$A_2$}};
\draw[color=green!50!black] (0.3,-0.2) .. controls +(0,0.25) and +(0,-0.25) .. (-0.2,0.2);
\draw[color=white, line width=4pt] (0,-0.2) -- (0,1); 
\draw (0,-1) -- (0,1); 
\draw[color=green!50!black] (-0.2,0.2) .. controls +(0,0.25) and +(-0.1,-0.25) .. (0,0.6);
\draw[color=green!50!black] (-0.35,-1) .. controls +(0,0.2) and +(0,-0.2) .. (-0.25,-0.2);
\draw[color=green!50!black] (-0.25,-0.2) .. controls +(0,0.2) and +(-0.5,-0.2) .. (0,0.85);
\draw[color=white, line width=4pt] (0.3,-0.2) .. controls +(0,-0.25) and +(0,0.6) .. (-0.7,-1);
\draw[color=green!50!black] (0.3,-0.2) .. controls +(0,-0.25) and +(0,0.6) .. (-0.7,-1);
\fill[color=green!50!black] (-0.53,-0.65) circle (2.5pt) node[above] (meet2) {{\scriptsize$\theta_{A_1}$}};
\fill[color=black] (0,0.6) circle (2.5pt) node[right] (meet2) {{\scriptsize$\rho_1$}};
\fill[color=black] (0,0.85) circle (2.5pt) node[right] (meet2) {{\scriptsize$\rho_2$}};
\end{tikzpicture} 
= 
\begin{tikzpicture}[very thick,scale=1,color=blue!50!black, baseline]
\draw (0,-1) node[below] (X) {{\scriptsize$M$}};
\draw[color=green!50!black] (-0.5,-1) node[below] (A1) {{\scriptsize$A_2$}};
\draw[color=green!50!black] (-1,-1) node[below] (A2) {{\scriptsize$A_1$}};
\draw (0,1) node[right] (Xu) {};
\draw[color=green!50!black] (A1) .. controls +(0,0.5) and +(-0.5,-0.5) .. (0,0.6);
\draw[color=white, line width=4pt] (A2) .. controls +(0,0.5) and +(-0.25,-0.25) .. (0,-0.25);
\draw[color=green!50!black] (A2) .. controls +(0,0.5) and +(-0.25,-0.25) .. (0,-0.25);
\draw (0,-1) -- (0,1); 
\fill[color=black] (0,-0.25) circle (2.5pt) node[right] (meet) {{\scriptsize$\rho_1^{\textrm{tw}}$}};
\fill[color=black] (0,0.6) circle (2.5pt) node[right] (meet2) {{\scriptsize$\rho_2$}};
\end{tikzpicture} 
\, . 
\ee
\end{proof}

The above directly carries over to multi-modules for an arbitrary number of algebras: 

\begin{definition}
\label{def:twmm}
Let ${}_{A_1\dots A_n}M$ be a multi-module with $A_i$-action $\rho_i \colon A_i \otimes M \to M$ for all $i \in \{ 1,\dots,n \}$. 
Then for every $j \in \{ 0,\dots,n \}$ the \textsl{twisted multi-module} 
\be
{}_{A_{j+1}\dots A_n A_1\dots A_j}M^{\textrm{tw}_{j}}
\ee
by definition has $M^{\textrm{tw}_{j}} := M$ and 
\begin{itemize}
\item
$A_i$-action $\rho_i$ for $i \in \{ j+1, \dots, n \}$, 
\item
$A_k$-action $\rho_k^{\textrm{tw}}$ for $k \in \{ 1, \dots, j \}$. 
\end{itemize}
\end{definition}

Note that for $j=0$ the second condition is empty, and so $M^{\textrm{tw}_0} =M$.
The next lemma generalises the module isomorphism \eqref{eq:thetaiso}.

\begin{lemma}
\label{lem:thetaisomulti}
$\theta_M \colon {}_{A_1\dots A_n}M 
	\to
{}_{A_1\dots A_n}M^{\textrm{tw}_{n}}$ is an isomorphism of multi-modules. 
\end{lemma}

\begin{proof}
Since $\rho_i^{\textrm{tw}} = \theta_M \circ \rho_i \circ (1_{A_i} \otimes \theta_M^{-1} )$ for $i \in \{1,\dots,n \}$ it is immediate that $\theta_M \circ \rho_i = \rho_i^{\textrm{tw}} \circ (\id_A \otimes \theta_M)$ for all $i$.
\end{proof}

\medskip

Finally we discuss the cyclic symmetry that will be needed to consistently label line defects in Section~\ref{sec:RTdef}. 
We will first treat the case that all algebras $A_i$ in the multi-module ${}_{A_1\dots A_n}M$ are equal to some $B$, so that we have a $B^{\otimes n}$-module. Then we pass to the general case where  
$A_i = A_{i + k}$ for all $i$ and some divisor~$k$ of~$n$, with indices taken modulo~$n$.

\medskip

Let $B \in \Cat{C}$ be an algebra and write 
	\be
	\mathcal{M} := B^{\otimes n}\text{-mod}_{\Cat{C}}
	\ee	
for the category of $B^{\otimes n}$-modules in $\Cat{C}$. 
We will define an action of the cyclic group $C_n$ on $\mathcal{M}$, that is, we will specify a monoidal functor
\be\label{eq:tw-as-monfun}
	\mathrm{tw} \colon C_n \lra \FEnd(\mathcal M) \, .
\ee 
Here, $C_n$ is understood as a strict monoidal category with only identity morphisms and the group operation as tensor product. The category of endofunctors $\FEnd(\mathcal M)$ is strict monoidal via composition.
For $a \in C_n$ denote by $\rp{a}$ the representative in $\{0,1,\dots,n-1\}$.
Then the action of \eqref{eq:tw-as-monfun} on objects is given by, for $a \in C_n$,
\be\label{eq:tw-as-monfun_obj}
	\mathrm{tw}_a(M) := M^{\mathrm{tw}_{\rp a}} \, , 
\ee
and for a morphism $f \colon M \rightarrow N$ we set $f^{\mathrm{tw}_{ a}} = f \colon M^{\mathrm{tw}_{ a}} \rightarrow N^{\mathrm{tw}_{ a}}$.
To define the monoidal structure, we set $\sigma(a,b) = \tfrac{1}{n}( \rp a + \rp b - \rp{(a+b)} ) \in \{0,1\}$ for all $a,b \in C_n$ and then define the natural isomorphisms 
\be
	\tau_{a,b} \colon \mathrm{tw}_a \circ \mathrm{tw}_b \longrightarrow \mathrm{tw}_{a+b}
	\, , \quad 
	(\tau_{a,b})_M = (\theta_M^{-1})^{\sigma(a,b)} \, .
\ee
By Lemma~\ref{lem:thetaisomulti}, $(\tau_{a,b})_M$ is indeed a $B^{\otimes n}$-module homomorphism.

\begin{lemma}\label{lem:monfun-Cn-EndM}
$(\mathrm{tw},\tau) \colon  C_n \to \FEnd(\mathcal M)$ is a monoidal functor.
\end{lemma}

\begin{proof}
It remains to check the hexagon diagram for the coherence isomorphisms, which in the present case boils down to the identity 
\be
	\theta_M^{-\sigma(a,b+c)} \circ
	\theta_M^{-\sigma(b,c)} 
	=
	\theta_M^{-\sigma(a+b,c)} \circ
	\theta_M^{-\sigma(a,b)} \, ,
\ee
for all $a,b,c \in C_n$ and $M \in \mathcal{M}$, which in turn is easily checked.
\end{proof}

Thanks to the monoidal functor in Lemma~\ref{lem:monfun-Cn-EndM} we can talk about $C_n$-equivariant objects in $\mathcal{M}$. By definition, these are tuples $(M , \{ \varphi_a \}_{a \in C_n})$ where $M \in \mathcal{M}$ and the $\varphi_a \colon \mathrm{tw}_a(M) \to M$ are module isomorphisms. The $\varphi_a$ must satisfy $\varphi_0 = \id_M$ and commutativity of 
\be
\begin{tikzpicture}[
			     baseline=(current bounding box.base), 
			     descr/.style={fill=white,inner sep=3.5pt}, 
			     normal line/.style={->}
			     ] 
\matrix (m) [matrix of math nodes, row sep=3.5em, column sep=4em, text height=1.5ex, text depth=0.1ex] {%
\mathrm{tw}_a\big(\mathrm{tw}_b(M)\big) & \mathrm{tw}_a(M)
\\
\mathrm{tw}_{a+b}(M)   & M
\\
};
\path[font=\footnotesize] (m-1-1) edge[->] node[above] {$\mathrm{tw}_a(\varphi_b)$} (m-1-2);
\path[font=\footnotesize] (m-1-1) edge[->] node[left] {$(\tau_{a,b})_M$} (m-2-1);
\path[font=\footnotesize] (m-2-1) edge[->] node[below] {$\varphi_{a+b}$} (m-2-2);
\path[font=\footnotesize] (m-1-2) edge[->] node[right] {$\varphi_{a}$} (m-2-2);
\end{tikzpicture}
\ee
For more details on equivariantisation we refer e.\,g.\ to~\cite[Sect.\,2.7]{EGNO-book}. 

Let $g \in C_n$ be the generator such that $\rp g=1$. It is immediate from the above diagram that all $\varphi_a$ are determined by $\varphi_g$ and that the following lemma holds.

\begin{lemma}
\label{lem:Cn-equiv-via-gen}
Giving a $C_n$-equivariant object in $\mathcal{M}$ is the same as giving a pair $(M,\varphi)$ with $\varphi \colon M^{\mathrm{tw}_1} \to M$ a module isomorphism such that $\varphi^n = \theta_M^{-1}$.
\end{lemma}

\medskip

We now return to general $(A_1,\dots,A_n)$-multi-modules. Let $M$ be such a module and suppose that the $A_i$ are periodic in the sense that $A_i = A_{i+k}$ for some $k>0$ and all~$i$, with indices taken modulo $n$. 
If we write $B = A_1 \otimes \cdots \otimes A_k$, then~$M$ is a $B^{\otimes n/k}$-module and can be equipped with a $C_{n/k}$-equivariant structure. 
In Section~\ref{sec:RTdef} we will need the case that $k>0$ is minimal, so that the cyclic symmetry is maximal. 

Using the simplified description in Lemma~\ref{lem:Cn-equiv-via-gen}, we finally arrive at the following definition.

\begin{definition}
\label{def:symmetricmodule}
Let $A_1,\dots,A_n$ be algebras and let $k\in \Z_+$ divide~$n$, such that $A_i = A_{i+k}$ for all~$i$ (where $i+k$ is taken modulo~$n$). 
A \textsl{$k$-cyclic $(A_1,\dots ,A_n)$-multi-module} is a pair $(M,\varphi)$ where~$M$ is an $(A_1,\dots ,A_n)$-multi-module and $\varphi \colon M^{\mathrm{tw}_k} \to M$ is a module isomorphism such that $\varphi^{n/k} = \theta_M^{-1}$.
\end{definition}

\begin{example}
\begin{enumerate}
\item
Let $(A,\mu)$ be a commutative algebra as in Example~\ref{ex:Amultmod}. 
Then ${}_{AA}A$ is 1-cyclic iff $\theta_A = 1_A$. 
Indeed, a cyclic structure on ${}_{AA}A$ amounts to a module isomorphism 
$\varphi \colon {}_{AA}A^{\textrm{tw}} \to {}_{AA}A$, 
i.\,e.
\be
\label{eq:Asys}
\begin{tikzpicture}[very thick,scale=0.75,color=green!50!black, baseline]
\draw (0,-1) node[below] (X) {};
\draw[color=green!50!black] (-0.5,-1) node[below] (X) {};
\draw (0,1) node[right] (Xu) {};
\draw (0,-1) -- (0,1); 
\fill[color=green!50!black] (0,-0.25) circle (2.9pt) node[right] (meet) {{\scriptsize$\mu$}};
	\fill[color=black] (0,0.6) circle (2.9pt) node[right] (meet2) {{\scriptsize$\varphi$}};
\draw[color=green!50!black] (-0.5,-1) .. controls +(0,0.25) and +(-0.25,-0.25) .. (0,-0.25);
\end{tikzpicture} 
=
\begin{tikzpicture}[very thick,scale=0.75,color=green!50!black, baseline]
\draw (0,-1) node[below] (X) {};
\draw[color=green!50!black] (-0.5,-1) node[below] (X) {};
\draw (0,1) node[right] (Xu) {};
\draw (0,-1) -- (0,1); 
\fill[color=green!50!black] (0,0.6) circle (2.9pt) node[right] (meet) {{\scriptsize$\mu$}};
	\fill[color=black] (0,-0.25) circle (2.9pt) node[right] (meet2) {{\scriptsize$\varphi$}};
\draw[color=green!50!black] (-0.5,-1) .. controls +(0,0.5) and +(-0.5,-0.5) .. (0,0.6);
\draw[color=green!50!black] (-0.5,-1) .. controls +(0,0.5) and +(-0.5,-0.5) .. (0,0.6);
\end{tikzpicture} 
\quad\text{and}\quad
\begin{tikzpicture}[very thick,scale=0.75,color=green!50!black, baseline]
\draw (0,-1) node[below] (X) {};
\draw[color=green!50!black] (-0.5,-1) node[below] (X) {};
\draw (0,1) node[right] (Xu) {};
\draw (0,-1) -- (0,1); 
\fill[color=green!50!black] (0,-0.25) circle (2.9pt) node[right] (meet) {{\scriptsize$\mu^{\textrm{tw}}$}};
	\fill[color=black] (0,0.6) circle (2.9pt) node[right] (meet2) {{\scriptsize$\varphi$}};
\draw[color=green!50!black] (-0.5,-1) .. controls +(0,0.25) and +(-0.25,-0.25) .. (0,-0.25);
\end{tikzpicture} 
=
\begin{tikzpicture}[very thick,scale=0.75,color=green!50!black, baseline]
\draw (0,-1) node[below] (X) {};
\draw[color=green!50!black] (-0.5,-1) node[below] (X) {};
\draw (0,1) node[right] (Xu) {};
\draw (0,-1) -- (0,1); 
\fill[color=green!50!black] (0,0.6) circle (2.9pt) node[right] (meet) {{\scriptsize$\mu$}};
	\fill[color=black] (0,-0.25) circle (2.9pt) node[right] (meet2) {{\scriptsize$\varphi$}};
\draw[color=green!50!black] (-0.5,-1) .. controls +(0,0.5) and +(-0.5,-0.5) .. (0,0.6);
\end{tikzpicture} 
\, . 
\ee
The first equation states that~$\varphi$ is an ${}_A A$-automorphism. 
Using this and postcomposing the second equation of~\eqref{eq:Asys} with $\varphi^{-1}$ we compute 
\be
\begin{tikzpicture}[very thick,scale=0.75,color=green!50!black, baseline]
\draw (0,-1) node[below] (X) {};
\draw[color=green!50!black] (-0.5,-1) node[below] (X) {};
\draw (0,-1) -- (0,1); 
\draw[color=green!50!black] (-0.5,-1) .. controls +(0,0.25) and +(-0.25,-0.25) .. (0,0.1);
\fill (0,0.1) circle (2.9pt) node[right] (meet) {{\scriptsize$\mu$}};
\end{tikzpicture} 
\stackrel{\eqref{eq:Asys}}{=}
\begin{tikzpicture}[very thick,scale=0.75,color=green!50!black, baseline]
\draw (0,-1) node[below] (X) {};
\draw[color=green!50!black] (-0.5,-1) node[below] (X) {};
\draw (0,-1) -- (0,1); 
\draw[color=green!50!black] (-0.5,-1) .. controls +(0,0.25) and +(-0.25,-0.25) .. (0,0.1);
\fill (0,0.1) circle (2.9pt) node[right] (meet) {{\scriptsize$\mu^{\textrm{tw}}$}};
\end{tikzpicture} 
\;
\overset{\eqref{eq:alternative-rhotw}}=
\begin{tikzpicture}[very thick,scale=0.75,color=green!50!black, baseline]
\draw[color=green!50!black] (-0.85,0.5) -- (-0.85,1); 
\fill[color=green!50!black] (-0.85,0.5) circle (2.5pt) node[left] (mult1) {{\scriptsize$\mu$}};
\draw[color=green!50!black] (-1.2,0.2) .. controls +(0,0.25) and +(0,0) .. (-0.85,0.5);
\draw[color=green!50!black] (-0.5,0.2) .. controls +(0,0.25) and +(0,0) .. (-0.85,0.5);
\draw[color=green!50!black] (-1.2,-0.4) .. controls +(0,-0.25) and +(0,0.25) .. (-0.5,-1);
\draw[color=green!50!black] (-0.5,-0.4) .. controls +(0,0.25) and +(0,-0.25) .. (-1.2,0.2);
\draw[color=white, line width=4pt] (-1.2,-0.4) .. controls +(0,0.25) and +(0,-0.25) .. (-0.5,0.2);
\draw[color=green!50!black] (-1.2,-0.4) .. controls +(0,0.25) and +(0,-0.25) .. (-0.5,0.2);
\draw[color=white, line width=4pt] (-0.5,-0.4) .. controls +(0,-0.25) and +(0,0.25) .. (-1.2,-1);
\draw[color=green!50!black] (-0.5,-0.4) .. controls +(0,-0.25) and +(0,0.25) .. (-1.2,-1);
\fill (-1.05,-0.8) circle (2.9pt) node[left] (meet2) {{\scriptsize$\theta^{-1}_A$}};
\end{tikzpicture} 
\,
\overset{\text{$A$ comm.}}=
\begin{tikzpicture}[very thick,scale=0.75,color=green!50!black, baseline]
\draw[color=green!50!black] (-0.85,0.1) -- (-0.85,1); 
\fill[color=green!50!black] (-0.85,0.1) circle (2.5pt) node[right] (mult1) {{\scriptsize$\mu$}};
\draw[color=green!50!black] (-1.2,-1) .. controls +(0,1) and +(0,0) .. (-0.85,0.1);
\draw[color=green!50!black] (-0.5,-1) .. controls +(0,1) and +(0,0) .. (-0.85,0.1);
\fill (-1.18,-0.6) circle (2.9pt) node[left] (meet2) {{\scriptsize$\theta^{-1}_A$}};
\end{tikzpicture} 
\quad .
\ee
Precomposing with $1_A \otimes 
\begin{tikzpicture}[thick,scale=0.4,color=green!50!black, baseline=-0.2cm]
\draw (0,-0.5) node[Odot] (D) {}; 
\draw (D) -- (0,0.3); 
\end{tikzpicture} 
$ then proves the claim. 
\item 
Let~$A$ be a commutative algebra with $\theta_A = 1_A$, and let~$M$ be an $A$-module. 
Then ${}_{A\dots A}M$ canonically is a multi-module for any number of $A$-factors. 
Furthermore, ${}_{A\dots A}M$ has a canonical cyclic structure (meaning one with $\varphi = 1_M$) iff ${}_A M$ is \textsl{local}, i.\,e.
\be
\begin{tikzpicture}[very thick,scale=0.75,color=blue!50!black, baseline]
\draw[color=blue!50!black] (-0.85,0.5) -- (-0.85,1); 
\draw[color=green!50!black] (-1.2,0.2) .. controls +(0,0.25) and +(0,0) .. (-0.85,0.5);
\draw[color=blue!50!black] (-0.5,0.2) .. controls +(0,0.25) and +(0,0) .. (-0.85,0.5);
\draw[color=blue!50!black] (-1.2,-0.4) .. controls +(0,-0.25) and +(0,0.25) .. (-0.5,-1);
\draw[color=green!50!black] (-0.5,-0.4) .. controls +(0,0.25) and +(0,-0.25) .. (-1.2,0.2);
\draw[color=white, line width=4pt] (-1.2,-0.4) .. controls +(0,0.25) and +(0,-0.25) .. (-0.5,0.2);
\draw[color=blue!50!black] (-1.2,-0.4) .. controls +(0,0.25) and +(0,-0.25) .. (-0.5,0.2);
\draw[color=white, line width=4pt] (-0.5,-0.4) .. controls +(0,-0.25) and +(0,0.25) .. (-1.2,-1);
\draw[color=green!50!black] (-0.5,-0.4) .. controls +(0,-0.25) and +(0,0.25) .. (-1.2,-1);
\fill[color=black] (-0.85,0.5) circle (2.9pt) node[left] (mult1) {};
\draw (-0.5,-1) node[below] (X) {{\scriptsize$M$}};
\draw[color=green!50!black] (-1.2,-1) node[below] (A1) {{\scriptsize$A$}};
\end{tikzpicture} 
=
\begin{tikzpicture}[very thick,scale=0.75,color=blue!50!black, baseline]
\draw[color=green!50!black] (-1,-1) node[below] (A1) {};
\draw[color=green!50!black] (-0.5,-1) node[below] (A2) {};
\draw (0,1) node[right] (Xu) {};
\draw (0,-1) -- (0,1); 
\draw[color=green!50!black] (-1,-1) .. controls +(0,0.5) and +(-0.5,-0.5) .. (0,0.5);
\draw (0,-1) node[below] (X) {{\scriptsize$M$}};
\draw[color=green!50!black] (-1,-1) node[below] (A1) {{\scriptsize$A$}};
\fill[color=black] (0,0.5) circle (2.9pt) node[right] (meet2) {};
\end{tikzpicture} 
\, . 
\ee
This follows along the same lines as part (i). 
\item 
Let~$A$ be any algebra in~$\mathcal C$, and let $M,N$ be $A$-modules. 
Then ${}_{A_{1}A_{2}}(M \otimes N)$ is a multi-module with component actions 
\be
\label{eq:actions-M}
\begin{tikzpicture}[very thick,scale=0.75,color=blue!50!black, baseline]
\draw[color=green!50!black] (-1,-1) node[below] (A1) {};
\draw[color=green!50!black] (-0.5,-1) node[below] (A2) {};
\draw (0,-1) -- (0,1); 
\draw (0.75,-1) -- (0.75,1); 
\draw[color=green!50!black] (-1,-1) .. controls +(0,0.5) and +(-0.5,-0.5) .. (0,0.5);
\draw (0,-1) node[below] (M) {{\scriptsize$M$}};
\draw (0.75,-1) node[below] (M) {{\scriptsize$N$}};
\draw[color=green!50!black] (-1,-1) node[below] (A1) {{\scriptsize$A_{1}$}};
\fill[color=black] (0,0.5) circle (2.9pt) node[right] (meet2) {};
\end{tikzpicture} 
\, , \quad
\begin{tikzpicture}[very thick,scale=0.75,color=blue!50!black, baseline]
\draw[color=green!50!black] (-1,-1) node[below] (A1) {};
\draw[color=green!50!black] (-0.5,-1) node[below] (A2) {};
\draw (0,-1) -- (0,1); 
\draw (0.75,-1) -- (0.75,1); 
\draw[color=white, line width=4pt] (-1,-1) .. controls +(0,0.5) and +(-0.5,-0.5) .. (0.75,0.5);
\draw[color=green!50!black] (-1,-1) .. controls +(0,0.5) and +(-0.5,-0.5) .. (0.75,0.5);
\draw (0,-1) node[below] (M) {{\scriptsize$M$}};
\draw (0.75,-1) node[below] (M) {{\scriptsize$N$}};
\draw[color=green!50!black] (-1,-1) node[below] (A1) {{\scriptsize$A_{2}$}};
\fill[color=black] (0.75,0.5) circle (2.9pt) node[right] (meet2) {};
\end{tikzpicture} 
\, ,
\ee
where  $A_1= A_2 = A$.
Furthermore, ${}_{AA}(M \otimes M)$ has a cyclic structure with isomorphism 
\be
\varphi
\; :=  \; 
(1_{M} \otimes \theta_{M}^{-1}) \circ c_{M,M}^{-1}  
\colon 
\quad
{}_{AA}(M \otimes M)^{\textrm{tw}_1} \lra {}_{AA}(M \otimes M) \ .
\ee
Indeed, 
$\varphi$ 
is clearly an isomorphism.  To see that 
$\varphi$ 
 map of multi-modules  according to Definition \ref{definition:mapmulti},
we need to see that it intertwines the two actions of ${}_{AA}(M \otimes M)^{\textrm{tw}_1}$ with those of ${}_{AA}(M \otimes M)$ from \eqref{eq:actions-M}. 
This in turn is a straightforward computation in string diagram notation.
\end{enumerate}
\end{example}

\section{Bordisms with defects in three dimensions}
\label{sec:defectTQFTs}

In this section we review 
the category of 3-dimensional bordisms with defects which we will use in Section~\ref{sec:RTdef} to extend the Reshetikhin-Turaev TQFT to line and surface defects. We specialise the $n$-dimensional setup in \cite{CRS1} to $n=3$, which yields the bordism category already used in \cite{CMS}.
We first describe a collection of sets and maps called 3-dimensional defect data, and then turn to the stratified manifolds decorated by these data from which the bordism category is built.

\medskip

A \textsl{set of 3-dimensional defect data} (\textsl{defect data} for short) is a tuple 
\be
	\D = \big(D_3,D_2,D_1;s,t,j\big) \, . 
\ee
Here $D_3, D_2, D_1$ are sets which will label strata of the corresponding dimension. 
The remaining entries $s,t,j$ are called \textsl{source}, \textsl{target} and \textsl{junction map}, respectively, for reasons that will be clear when we label stratified manifolds (cf.\ Figure~\ref{fig:local-model-2mf}). 
The source and target maps are
\be
	s,t \colon D_2 \times \{ \pm \} \longrightarrow D_3 \, ,
\ee
and 	they must satisfy, for all $f \in D_2$ and $\eps \in \{ \pm \}$,
\be
	s(f,\eps) = t(f,-\eps) \, .
\ee
The junction map (which is called ``folding map''~$f$ in \cite{CMS}) 
is
\be\label{eq:junction}
j \colon D_1  \times \{ \pm \}  \longrightarrow D_3 \sqcup \bigsqcup_{m\in\Z_{+}} P_m/C_m
\,,\qquad 
P_m \subset \big(D_2 \times \{ \pm \}\big)^m \ .
\ee
The subsets $P_m$ in the target of $j$ are defined by the condition that source and target maps between entries match:
\be
P_m = 
\big\{\, (d_1,d_2,\dots,d_m)  \,\big|\, s(d_j) = t(d_{j+1})
\text{ for } j \in \{1,\dots,m\} \,\big\} \, ,
\ee
and where we took $d_{m+1}:=d_1$.
When changing the sign argument, the value of~$j$ behaves as follows: 
\be
j(l,+) = \big( (f_1,\eps_1) ,\dots, (f_m,\eps_m) \big)
	\quad \Longleftrightarrow \quad
	j(l,-) = \big( (f_m,-\eps_m) ,\dots, (f_1,-\eps_1) \big) \, .
\ee
The change of the sign argument $\pm$ of $s,t,j$ will later describe orientation reversal. 

\medskip

Next we define 3-dimensional $\D$-decorated bordisms. In short, these are stratified manifolds with parametrised boundaries decorated by a set of defect data~$\D$. 
All our manifolds will be smooth and oriented. 

A \textsl{stratified $n$-manifold $M$ with boundary} is an oriented topological $n$-manifold together with a filtration into submanifolds $M = F_n \supset F_{n-1} \supset \cdots \supset F_0 \supset F_{-1} = \emptyset$, subject to conditions for which we refer to \cite[Sect.\,2.1]{CMS} or \cite[Sect.\,2.1]{CRS1}, and such that $F_j \setminus F_{j-1}$ is an oriented smooth manifold for each~$j$. 
One of the conditions is that $\partial M$ is a stratified $(n{-}1)$-manifold with induced orientations, and the strata have to meet the boundary transversally.
A connected component of $F_j \setminus F_{j-1}$ is called a \textsl{$j$-stratum} of~$M$.

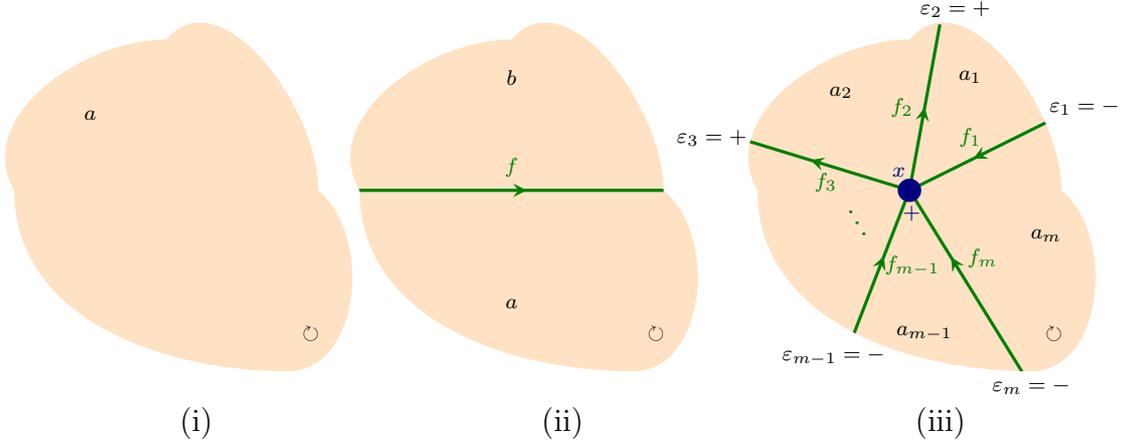
\begin{figure}[t]
\centering
\begin{tabular}{ccc}
\hspace{-0.75cm}
\tikzzbox{\begin{tikzpicture}[very thick,scale=2,color=blue!50!black, baseline=0]
\fill[orange!23] 	(1,0) 
	to[out=90, in=50] (0,1) 
	to[out=180, in=120] (-1,0) 
	to[out=270, in=180] (0.8,-1.2)
	to[out=0, in=-40] (1,0);
\clip 	(1,0) 
	to[out=90, in=50] (0,1) 
	to[out=180, in=120] (-1,0) 
	to[out=270, in=180] (0.8,-1.2)
	to[out=0, in=-40] (1,0);
\fill[color=black] (-0.5,0.5) node (X) {{\scriptsize$a$}};
\fill[color=black] (0.95,-0.95) node (X) {{\scriptsize$\circlearrowright$}};
\end{tikzpicture}}
& 
\hspace{-1.3cm}
\tikzzbox{\begin{tikzpicture}[very thick,scale=2,color=blue!50!black, baseline=0]
\fill[orange!23] 	(1,0) 
	to[out=90, in=50] (0,1) 
	to[out=180, in=120] (-1,0) 
	to[out=270, in=180] (0.8,-1.2)
	to[out=0, in=-40] (1,0);
\clip 	(1,0) 
	to[out=90, in=50] (0,1) 
	to[out=180, in=120] (-1,0) 
	to[out=270, in=180] (0.8,-1.2)
	to[out=0, in=-40] (1,0);
\draw[color=black!50!green, >=stealth, decoration={markings, mark=at position 0.7 with {\arrow{>}},}, postaction={decorate}] 
		(-2,0) -- (1,0); 
\fill[color=black!50!green] (0,0.15) node (X) {{\scriptsize$f$}};
\fill[color=black] (0,0.75) node (X) {{\scriptsize$b$}};
\fill[color=black] (0,-0.75) node (X) {{\scriptsize$a$}};
\fill[color=black] (0.95,-0.95) node (X) {{\scriptsize$\circlearrowright$}};
\end{tikzpicture}}
& 
\hspace{-1.25cm}
\tikzzbox{\begin{tikzpicture}[very thick,scale=2,color=blue!50!black, baseline=0]
\fill[orange!23] 	(1,0) 
	to[out=90, in=50] (0,1) 
	to[out=180, in=120] (-1,0) 
	to[out=270, in=180] (0.8,-1.2)
	to[out=0, in=-40] (1,0);
\fill[color=black] (1.15,0.55) node (X) {{\scriptsize$\eps_1=-$}};
\fill[color=black] (0.3,1.2) node (X) {{\scriptsize$\eps_2=+$}};
\fill[color=black] (-1.3,0.35) node (X) {{\scriptsize$\eps_3=+$}};
\fill[color=black] (-0.5,-1.1) node (X) {{\scriptsize$\eps_{m-1}=-$}};
\fill[color=black] (0.8,-1.3) node (X) {{\scriptsize$\eps_m=-$}};
\clip 	(1,0) 
	to[out=90, in=50] (0,1) 
	to[out=180, in=120] (-1,0) 
	to[out=270, in=180] (0.8,-1.2)
	to[out=0, in=-40] (1,0);
\draw[color=black!50!green, >=stealth, decoration={markings, mark=at position 0.25 with {\arrow{<}},}, postaction={decorate}] 	(0,0) -- (2,1); 
\draw[color=black!50!green, >=stealth, decoration={markings, mark=at position 0.5 with {\arrow{>}},}, postaction={decorate}] 	(0,0) -- (-1.3,0.4); 
\draw[color=black!50!green, >=stealth, decoration={markings, mark=at position 0.5 with {\arrow{>}},}, postaction={decorate}] 	(0,0) -- (0.2,1.1); 
\draw[color=black!50!green, >=stealth, decoration={markings, mark=at position 0.4 with {\arrow{<}},}, postaction={decorate}] 	(0,0) -- (-0.5,-1.3); 
\draw[color=black!50!green, >=stealth, decoration={markings, mark=at position 0.4 with {\arrow{<}},}, postaction={decorate}] 	(0,0) -- (0.8,-1.3); 
\fill[color=black!50!green] (0.4,0.35) node (X) {{\scriptsize$f_1$}};
\fill[color=black!50!green] (-0.05,0.55) node (X) {{\scriptsize$f_2$}};
\fill[color=black!50!green] (-0.55,0.05) node (X) {{\scriptsize$f_3$}};
\fill[color=black!50!green] (0.02,-0.5) node (X) {{\scriptsize$f_{m-1}$}};
\fill[color=black!50!green] (0.48,-0.45) node (X) {{\scriptsize$f_m$}};
\fill[color=black] (0.4,0.75) node (X) {{\scriptsize$a_1$}};
\fill[color=black] (-0.45,0.65) node (X) {{\scriptsize$a_2$}};
\fill[color=black!50!green] (-0.35,-0.2) node (X) {{\small$\rotatebox{120}{\dots}$}};
\fill[color=black] (0.1,-0.95) node (X) {{\scriptsize$a_{m-1}$}};
\fill[color=black] (0.9,-0.3) node (X) {{\scriptsize$a_m$}};
\fill[color=black] (0.95,-0.95) node (X) {{\scriptsize$\circlearrowright$}};
\fill (0,0) circle (2.2pt) node[above, color=black] (X) {};
\fill (-0.07,0.12) node (X) {{\scriptsize$x$}};
\fill (0.015,-0.15) node (X) {{\scriptsize$+$}};
\end{tikzpicture}}
\\
(i) & (ii) & (iii)
\\
\end{tabular}
\caption{Local models for decorated stratified 3-manifolds are cylinders over the types of decorated stratified 2-manifolds shown. 
Here, $m \in \Z_{\geqslant 0}$, $a,b,a_1,\dots, a_m \in D_3$, $f,f_1,\dots,f_m \in D_2$, and $x \in D_1$. 
For model (ii) we must have $a = s(f,+)=t(f,-)$ and $b=t(f,+)=s(f,-)$. 
For model (iii) the number of 2-strata can vary, as well as their orientations. 
The label of each 2-stratum has to have as source and target the labels on the two neighbouring 3-strata, e.\,g.~$s(f_1,-)=a_m$. 
Finally, the junction map has to match the neighbouring 2-strata via $j(x,+) = ((f_1,\eps_1),\dots,(f_m,\eps_m))$. 
For the relation of the orientations in (iii) and of the strata in the corresponding cylinder, see \cite[Ex.\,2.3]{CRS1}.
}
\label{fig:local-model-2mf}
\end{figure}

Let $\D = (D_3,D_2,D_1;s,t,j)$ be a set of defect data. We first describe the labelling for stratified closed surfaces, then for 3-dimensional bordisms. 
For a surface, 
each $j$-stratum, $j \in \{0,1,2\}$, is labelled by an element from $D_{j+1}$. The index shift arises as these 2-manifolds will be used to parametrise boundaries of 3-manifolds. 
The labelling must be chosen such that each point has a neighbourhood isomorphic (as an oriented stratified decorated manifold\footnote{The notion of an isomorphism of stratified manifolds is slightly delicate. Given stratified manifolds $M = F_n \supset \cdots \supset F_0$ and $M' = F'_n \supset \cdots \supset F'_0$, an isomorphism $f \colon M \to M'$ is, first of all, a homeomorphism $M \to M'$ such that $f(F_j) = F'_j$. Secondly, $f$, restricted to $(F_j \setminus F_{j-1}) \to (F'_j \setminus F'_{j-1})$, is an orientation preserving diffeomorphism. For more details, see e.\,g.\ \cite[Ch.\,1]{Pflaum} and \cite[Sect.\,2.1]{CRS1}.

The reason not to restrict to stratifications where~$M$ itself is already a smooth manifold and where~$f$ is a diffeomorphism on all of~$M$ is illustrated in Figure~\ref{fig:intro-linedef}: 
There is no diffeomorphism between these two stratifications (embedded in $\R^3$), since the differential $df$ on the 1-stratum~$M$ would need to preserve the line~$M$ and the planes $A_2$ and $A_3$, and so cannot map the plane $A_1$ from its location on the left to that on the right.
})
to one of the three local models in Figure~\ref{fig:local-model-2mf}.
For 3-manifolds we will only consider stratifications without 0-strata (see Remark~\ref{rem:no-0-strata} below). Each $j$-stratum, $j \in \{1,2,3\}$, is now labelled by an element from $D_j$, and the local model is again as in Figure~\ref{fig:local-model-2mf}, but now taken times the open interval $(-1,1)$.

Let $X,Y$ be compact decorated stratified $2$-manifolds. 
A \textsl{(3-dimensional) defect bordism} $N \colon X \to Y$ is a compact decorated stratified $3$-manifold $N$, together with an isomorphism $X^{\mathrm{rev}} \sqcup Y \to \partial N$ of (germs of collars around) decorated stratified 2-manifolds. Here, $X^{\mathrm{rev}}$ is~$X$ with reversed orientation for all strata (but with the same labelling).
We call two defect bordisms $N,N' \colon X \to Y$ \textsl{equivalent} if there is an isomorphism $N \to N'$ compatible with stratification, orientations, decoration and boundary parametrisation.

\begin{definition}\label{def:defect-bordisms}
Let $\D$ be a set of defect data. The  
\textsl{decorated bordism category} $\Borddefn{3}(\D)$ has as objects compact closed $\D$-decorated stratified 2-manifolds. The set of morphisms between two objects $X,Y$ consists of the equivalence classes of decorated bordisms $[N] \colon X \to Y$. Composition is via gluing.
\end{definition}

The category $\Borddefn{3}(\D)$ is symmetric monoidal via the disjoint union as tensor product; for well-definedness of the gluing see \cite[Sect.\,2.1]{CRS1}. 
As in \cite[Def.\,3.1]{CMS} we fix some field $\Bbbk$ and define:

\begin{definition}
\label{def:defeTQFT}
 A \textsl{$3$-dimensional defect TQFT} with defect data~$\D$ is a symmetric monoidal functor 
\be
  \label{eq:def-TQFT}
  \zz\colon \Borddefn{3}(\D) \longrightarrow \Vectk \, . 
\ee
\end{definition}

\begin{remark}\label{rem:no-0-strata}
The theory of 3-dimensional defect TQFT as developed in \cite{CMS} has been extended to $n$ dimensions in \cite{CRS1}. In our description above, as well as in \cite{CMS}, bordisms are not allowed to have 0-strata. In \cite{CRS1}, 0-strata are included, resulting in an additional label set $D_0$ and additional ``adjacency maps'' \cite[Def.\,2.4]{CRS1}.
However, for a given defect TQFT without 0-strata there is a canonical way to add a ``complete set of labels $D_0$'' for 0-strata, see \cite[Sect.\,2.4]{CRS1} for details. In this sense, while omitting 0-strata in bordisms simplifies our discussion considerably, it is not a restriction as $D_0$ can be added back in. This is important as the orbifold construction from \cite[Sect.\,3]{CRS1} requires 0-strata, and we will indeed add them back in the follow-up work \cite{CRS3}.
\end{remark}

\section{Reshetikhin-Turaev TQFT}
\label{sec:MTCRT}

In this section we briefly recall some of the geometric and algebraic aspects of the Reshetikhin-Turaev construction of 3-dimensional TQFTs \cite{retu2,tur}, mostly following the conventions in \cite[Sect.\,4.4]{baki}. 
After recalling the definition of a modular tensor category~$\Cat{C}$, we review the category $\Bordriben{3}(\Cat{C})$ of $\Cat{C}$-marked bordisms. Then we state Turaev's theorem that from any modular tensor category~$\Cat{C}$ one can construct a TQFT with domain $\Bordriben{3}(\Cat{C})$. 

\medskip

Fix an algebraically closed field $\Bbbk$ of characteristic zero. 
A \textsl{modular tensor category (over~$\Bbbk$)} is a $\Bbbk$-linear abelian ribbon category~$\Cat{C}$ which 
\begin{enumerate}
\item 
is finitely semisimple with simple tensor unit,
\item 
has a non-degenerate braiding.
\end{enumerate}
Let us comment on these two points. 
The ``finitely'' in point~(i) means that the number of isomorphism classes of simple objects is finite, and that all objects are isomorphic to finite direct sums of the simple ones. 
Point~(ii) is the most crucial. 
One way to phrase it is that all transparent objects are isomorphic to direct sums of the tensor unit.\footnote{%
The original definition of modularity is via invertibility of the matrix whose entries are given by the invariant of the Hopf link coloured by simple objects~\cite[Sect.\,II.1.4]{tur}. Equivalence to the formulation in terms of transparent objects is shown in \cite[Prop.\,1.1]{Bruguieres:2000}.}
By definition, an object $T \in \Cat{C}$ is \textsl{transparent} if for all $X \in \Cat{C}$ we have $c_{X,T} \circ c_{T,X} = \id_{T \otimes X}$, i.\,e.\ $T$ has trivial monodromy with respect to all other objects.

\medskip

We now fix a modular tensor category~$\Cat{C}$, and move on	 to define $\Cat{C}$-marked 2- and 3-manifolds as in \cite[Def.\,4.4.1]{baki}. 
As in Section~\ref{sec:defectTQFTs}, all our manifolds are oriented. 
As topological and smooth $n$-manifolds are equivalent for $n \leqslant 3$, we may and will assume smoothness. 
A \textsl{$\Cat{C}$-marked 2-manifold} $X$ is a compact closed 2-manifold with finitely many  marked points~$p$, each equipped with a non-zero tangent vector $v_p$ and a label $(U_p,\eps_p)$ with $U_p \in \Cat{C}$ and $\eps_p \in \{ \pm \}$.
By $X^{\mathrm{rev}}$ we mean $X$ with reversed orientation, the same set of marked points $p$, but with $v_p$ replaced by $-v_p$ and $(U_p,\eps_p)$ replaced by $(U_p,-\eps_p)$.

A \textsl{$\Cat{C}$-marked 3-manifold} $N$ is a compact 3-manifold with possibly non-empty boundary, together with an embedded $\Cat{C}$-coloured ribbon tangle.
The ribbon tangle consists of ribbons and coupons. A ribbon has a core, that is, a 1-dimensional oriented submanifold. The ribbon itself carries a 2-orientation.\footnote{%
	Alternatively, instead of ribbons one can work with framed 1-manifolds. In this case one only remembers the core and a framing of the core obtained from picking a tangent vector to the ribbon which together with the orientation of the core induces the 2-orientation of the ribbon.}
Each ribbon is labelled by an element $U \in \Cat{C}$, and each coupon is labelled by a morphism from $\Cat{C}$ 
(compatible with the ribbons that end on it, as illustrated in Figure \ref{fig:C-marked-mfld-convention}, see \cite[Sect.\,2.3]{baki} for details). 
Each ribbon may form an annulus, or end on a coupon and/or the boundary $\partial N$. 

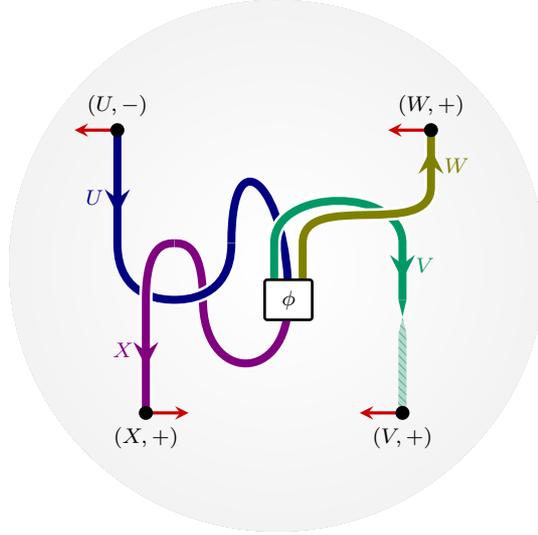
\begin{figure}[tb]
$$
\begin{tikzpicture}[line width=3pt, scale=0.75, color=blue!50!black, baseline]
\fill [inner color=white, outer color=white!90!gray] (-0.2,0.6) circle (4.7 cm);
\draw[color=blue!50!red] (0,0) .. controls +(0,-1.5) and +(0,-1.5) .. (-1.5,0)
		-- (-1.5,0) .. controls +(0,1) and +(0.15,0) .. (-2,1);
\draw (0,0) .. controls +(0,2) and +(0,2) .. (-1,1);
\draw[color=white, line width=5pt] (-1,1) .. controls +(0,-1) and +(0.15,0) .. (-2,0);
\draw[>=stealth, decoration={markings, mark=at position 0.8 with {\arrow{<}},}, postaction={decorate}] 
		(-1,1) .. controls +(0,-1) and +(0.15,0) .. (-2,0)
		-- (-2,0) .. controls +(-0.15,0) and +(0,-1) .. (-3,1) 
		-- (-3,3);
\draw[color=white, line width=5pt] (-2,1) .. controls +(-0.15,0) and +(0,1) .. (-2.5,0) -- (-2.5,-2);
\draw[color=blue!50!red, >=stealth, decoration={markings, mark=at position 0.75 with {\arrow{>}},}, postaction={decorate}] 
		(-2,1) .. controls +(-0.15,0) and +(0,1) .. (-2.5,0) 
		-- (-2.5,-2);
\draw[color=white, line width=5pt] (-0.25,0) -- (-0.25,1) -- (-0.25,1) .. controls +(0,1) and +(0,1) .. (2,1);
\draw[color=blue!40!green, >=stealth, decoration={markings, mark=at position 0.93 with {\arrow{>}},}, postaction={decorate}] 
		(-0.25,0) -- (-0.25,1)
		-- (-0.25,1) .. controls +(0,1) and +(0,1) .. (2,1) -- (2,0);
%
\newcommand{\brei}{0.07}
\fill[color=blue!40!green] (2-\brei,0) -- (2,-0.3) -- (2+\brei,0);
\fill [pattern=north west lines, opacity=0.3] 
		(2-\brei,-0.6) -- (2,-0.3) -- (2+\brei,-0.6) -- (2+\brei,-2) -- (1.99-\brei,-2) -- (2-\brei,-0.6);
\fill[color=blue!40!green, opacity=0.3] (2-\brei,-0.6) -- (2,-0.3) -- (2+\brei,-0.6);
\draw[color=blue!40!green, opacity=0.3] (2,-0.6) -- (2,-2);
%
\draw[color=white, line width=5pt] (0.25,0) -- (0.25,1) to[out=90, in=270] (2.5,2) -- (2.5,3); 
\draw[color=red!50!green, >=stealth, decoration={markings, mark=at position 0.93 with {\arrow{>}},}, postaction={decorate}] 
		(0.25,0) -- (0.25,1) to[out=90, in=270] (2.5,2) -- (2.5,3); 
%
\fill[color=white] (0,0) node[inner sep=4pt,draw, rounded corners=1pt, fill, color=white] (R2) {{\scriptsize$\rho_2$}};
\draw[line width=1pt, color=black] (0,0) node[inner sep=4pt,draw, rounded corners=1pt] (R) {{\scriptsize$\;\phi\;$}};
%
\draw[color=black!20!red, very thick, ->, >=stealth] (-2.5,-2) -- (-1.75,-2); 
\draw[color=black!20!red, very thick, ->, >=stealth] (-3,3) -- (-3.75,3); 
\draw[color=black!20!red, very thick, ->, >=stealth] (2,-2) -- (1.25,-2); 
\draw[color=black!20!red, very thick, ->, >=stealth] (2.5,3) -- (1.75,3); 
%
\fill[color=black] (-2.5,-2) circle (3.5pt) node[below, color=black] (X) {{\scriptsize$(X,+)$}};
\fill[color=blue!50!red] (-2.9,-0.5) node[below] (X) {{\scriptsize$X$}};
\fill[color=black] (-3,3) circle (3.5pt) node[above, color=black] (X) {{\scriptsize$(U,-)$}};
\fill (-3.4,2.2) node[below] (X) {{\scriptsize$U$}};
\fill[color=black] (2,-2) circle (3.5pt) node[below, color=black] (X) {{\scriptsize$(V,+)$}};
\fill[color=blue!40!green] (2.4,1) node[below] (X) {{\scriptsize$V$}};
\fill[color=black] (2.5,3) circle (3.5pt) node[above, color=black] (X) {{\scriptsize$(W,+)$}};
\fill[color=red!50!green] (2.95,2) node[above] (X) {{\scriptsize$W$}};
\end{tikzpicture}
$$
\caption{%
A ribbon diagram with $\phi \in \Hom_{\Cat{C}}(X^\vee,V \otimes U^\vee \otimes W)$, embedded into a 3-ball with four marked points on the boundary.%
}
\label{fig:C-marked-mfld-convention}
\end{figure}

The boundary $\partial N$ of a $\Cat{C}$-marked 3-manifold becomes a $\Cat{C}$-marked 2-manifold as follows. 
The marked points $p$ are the points where the core of a ribbon intersects $\partial N$; the tangent vector $v_p$ is a choice of nonzero vector in the intersection of the tangent planes of the ribbon and of $\partial N$ at $p$, such that $v_p$ and the orientation of the core induce the orientation of the ribbon; 
if the ribbon is labelled~$U$ and the core is oriented towards the boundary, the marked point is labelled $(U,+)$, and $(U,-)$ else, see \cite[Fig.\,4.8]{baki} and Figure~\ref{fig:C-marked-mfld-convention} for an illustration.

Let $X,Y$ be $\Cat{C}$-marked 2-manifolds. A \textsl{$\Cat{C}$-marked bordism} $N \colon X \to Y$ is a $\Cat{C}$-marked 3-manifold together with an isomorphism $X^{\mathrm{rev}} \sqcup Y \to \partial N$ of $\Cat{C}$-marked 2-manifolds (the isomorphism needs to respect the tangent vectors only up to a positive scalar). Two such bordisms $N,N' \colon X\to Y$ are \textsl{equivalent} if there is an isomorphism 
$N \to N'$, compatible with the boundary parametrisation, orientation, ribbon tangle and $\Cat{C}$-marking.

\begin{definition}
Let $\Cat{C}$ be a modular tensor category. 
The \textsl{$\Cat{C}$-marked bordism category} $\Bordribn{3}(\Cat{C})$ has as objects $\Cat{C}$-marked 2-manifold and as morphisms equivalence classes of $\Cat{C}$-marked bordisms. 
\end{definition}

It turns out that the Reshetikhin-Turaev TQFT is typically anomalous, that is, the gluing axiom only holds up to scalars. 
In order to get a functor one can extend $\Bordribn{3}(\Cat{C})$ as is done in \cite[Sect.\,IV.9]{tur}, to which we refer for details. We call the extended category 
\be\label{eq:bord-rib-ext}
\Bordriben{3}(\Cat{C}) \, .
\ee
Its objects are pairs $(\Sigma,\mathcal{L})$, where $\Sigma \in \Bordribn{3}(\Cat{C})$ and $\mathcal{L}$ is a Lagrangian subspace of $H_1(\Sigma,\mathbb{R})$. 
Its morphisms are pairs $([N],m)$, where $[N]$ is a morphism in $\Bordribn{3}(\Cat{C})$ and $m \in \mathbb{Z}$. The integers are additive under tensor products (i.\,e.\ disjoint union), but they behave in a more complicated way under gluing (which depends also on the subspaces $\mathcal{L}$)	
and allow one to absorb the gluing anomaly into scalar $m$-dependent weights.

\medskip

We have the following important result due to \cite{tur}.

\begin{theorem}
\label{thm:RTTQFT}
A modular tensor category $\Cat{C}$ over $\Bbbk$ gives rise -- via the construction in \cite[Sect.\,IV.9]{tur} -- to a symmetric monoidal functor (i.\,e.\ a TQFT) 
\begin{equation}
  \label{eq:RT}
  \zrt \colon \Bordriben{3}(\Cat{C}) \lra \Vect_{\Bbbk} \, .
\end{equation}
\end{theorem}

We will not need the details of this construction, but we will mention two properties of $\zrt$ which are important to us as they allow one to manipulate the ribbon tangles inside a bordism. 

For the first property, let $\phi \colon X \to Y$ be a morphism in $\Cat{C}$ and let $B^3(\phi)$ be a 3-ball with embedded ribbon tangle, such that the ribbon tangle contains a single coupon labelled~$\phi$, and such that each ribbon has one end on the coupon and one on the boundary $\partial B^3(\phi)$. Then the linear map
\be
	\phi \longmapsto \zrt \big( B^3(\phi) \big) 
\ee
is an isomorphism from $\Cat{C}(X,Y)$ to $\zrt \big( \partial B^3(\phi) \big)$.
This is a consequence of the definition of the state spaces, i.\,e.\ of how $\zrt$ is evaluated on objects.

Before stating the second property we need to
recall the following relation between ribbon tangles and morphisms in $\Cat{C}$ (which holds for every ribbon category, see e.\,g.\ \cite[Sect.\,2.3]{baki}). Consider a $\Cat{C}$-coloured ribbon tangle $R$ inside $\mathbb{R}^2 \times [0,1]$, such that ribbons which end on the boundary only do so on ``on the real axis'', that is on $\mathbb{R} \times \{0\} \times \{0\}$ or $\mathbb{R} \times \{0\} \times \{1\}$. To such a configuration one can assign a unique morphism $\mathcal{F}(R)$ in $\Cat{C}$ whose source and target are given by the tensor product of the objects labelling the ribbons intersecting the boundary. Importantly, $\mathcal{F}(R)$ depends only on the isotopy class of $R$ \cite{retu1}, see also \cite[Thm.\,2.3.8]{baki}. This, ultimately, is also the justification behind the graphical calculus employed in Section~\ref{sec:modules}.

For the second property, let $B^3(R)$ be a 3-ball with embedded $\Cat{C}$-coloured ribbon tangle~$R$. 
Isotope the ribbon tangle such that some part -- call it $R'$ -- of it sits inside a box, which we identify with a subset of $\mathbb{R}^2 \times [0,1]$, and suppose the ribbons intersect the box as required above. Then the ribbon tangle $R'$ can be replaced by a coupon labelled $\mathcal{F}(R')$ without changing the value of $\zrt$.

\begin{remark}
\label{remark:rel-rib-def}
\begin{enumerate}
\item
Given a modular tensor category $\Cat{C}$, denote by~$I$ a choice of representatives of the isomorphism classes of simple objects. 
For $i \in I $ write $\dim(i) = d_i \cdot \id_\one \in \End_{\Cat{C}}(\one) \cong \Bbbk$ with $d_i \in \Bbbk^\times$ for the categorical dimension of~$i$ (i.\,e.\ the trace of $\id_i$ as computed via the duality morphisms), and let $\vartheta_i \in \Bbbk^\times$ be defined by $\theta_i = \vartheta_i \cdot \id_i$. 
The modular tensor category $\Cat{C}$ is called \textsl{anomaly free} if
\be
\sum_{i\in I } \vartheta_i \cdot d_i
= 
\sum_{i\in I } \vartheta_i^{-1} \cdot d_i \, .
\ee
The important point about anomaly free $\Cat{C}$ is that in this case the gluing anomaly which required us to extend the bordism category in \eqref{eq:RT} vanishes \cite[Thm.\,IV.7.1]{tur}. In other words, for anomaly free $\Cat{C}$ the functor \eqref{eq:RT} factors as
\be
	  \zrt \colon \Bordriben{3}(\Cat{C}) 
	  \lra 
	  \Bordribn{3}(\Cat{C}) 
	  \lra
	  \Vect_{\Bbbk} \, .
\ee
\item
Just as we extended $\Bordribn{3}(\Cat{C})$ to $\Bordriben{3}(\Cat{C})$, we can extend the bordism category $\Borddefn{3}(\D)$ 
introduced in Definition~\ref{def:defect-bordisms} to 
\be
\label{eq:bord-def-ext}
	\Borddefen{3}(\D) \, .
\ee
Its objects are pairs $(\Sigma, \mathcal L)$ with $\Sigma \in \Borddefn{3}(\D)$ and~$\mathcal L$ a Lagrangian subspace of $H_1(\Sigma',\mathbb{R})$, with~$\Sigma'$ the unstratified manifold underlying~$\Sigma$. 
The morphisms are pairs $([N],m)$ with~$N$ now a $\D$-decorated defect bordism and $m \in \Z$. 
The numbers~$m$ behave under gluing as in $\Bordriben{3}(\Cat{C})$. 
This extension will be important in the next section.
\end{enumerate}
\end{remark}

\section{Reshetikhin-Turaev TQFT with line and surface defects}
\label{sec:RTdef}

In this section we introduce a set of defect data~$\DC$ made up of Frobenius algebras (to decorate 2-strata) and their cyclic multi-modules (to decorate 1-strata) internal to a fixed modular tensor category~$\Cat C$. 
Then we construct a defect TQFT 
\be
\label{eq:whatwedo}
\zzc \colon \Borddefen{3}(\DC) \lra \Vectk 
\ee
which augments the Reshetikhin-Turaev TQFT of Theorem~\ref{thm:RTTQFT} to include non-trivial surface defects. 
Finally, we show that $\zzc$ can detect some non-isotopic surface embeddings. 

\medskip

Let $\Cat{C}$ be a modular tensor category. 
A \textsl{Frobenius algebra} in $\Cat{C}$ is a tuple $(A,\mu,\eta,\Delta,\eps)$ such that $(A,\mu,\eta)$ is an associative unital algebra, $(A,\Delta,\eps)$ is a coassociative counital coalgebra, and the Frobenius property
\be
\tikzzbox{\begin{tikzpicture}[very thick,scale=0.4,color=green!50!black, baseline=0cm]
\draw[-dot-] (0,0) .. controls +(0,-1) and +(0,-1) .. (-1,0);
\draw[-dot-] (1,0) .. controls +(0,1) and +(0,1) .. (0,0);
\draw (-1,0) -- (-1,1.5); 
\draw (1,0) -- (1,-1.5); 
\draw (0.5,0.8) -- (0.5,1.5); 
\draw (-0.5,-0.8) -- (-0.5,-1.5); 
\end{tikzpicture}}
=
\tikzzbox{\begin{tikzpicture}[very thick,scale=0.4,color=green!50!black, baseline=0cm]
\draw[-dot-] (0,1.5) .. controls +(0,-1) and +(0,-1) .. (1,1.5);
\draw[-dot-] (0,-1.5) .. controls +(0,1) and +(0,1) .. (1,-1.5);
\draw (0.5,-0.8) -- (0.5,0.8); 
\end{tikzpicture}}
=
\tikzzbox{\begin{tikzpicture}[very thick,scale=0.4,color=green!50!black, baseline=0cm]
\draw[-dot-] (0,0) .. controls +(0,1) and +(0,1) .. (-1,0);
\draw[-dot-] (1,0) .. controls +(0,-1) and +(0,-1) .. (0,0);
\draw (-1,0) -- (-1,-1.5); 
\draw (1,0) -- (1,1.5); 
\draw (0.5,-0.8) -- (0.5,-1.5); 
\draw (-0.5,0.8) -- (-0.5,1.5); 
\end{tikzpicture}}
\ee
holds, i.\,e.\ $\Delta$ is an intertwiner of $A$-$A$-bimodules. A Frobenius algebra is called \textsl{symmetric} if $\eps \circ \mu = \eps \circ \mu \circ c_{A,A} \circ (\id_A \otimes \theta_A)$. 
Some equivalent ways of writing the right-hand side are
\be
\begin{tikzpicture}[very thick,scale=0.4,color=green!50!black, baseline=0cm]
\draw (-1.5,2) node[Odot] (end) {}; 
\draw[color=green!50!black] (-1.5,1) -- (end); 
\fill[color=green!50!black] (-1.5,1) circle (4.5pt) node (mult1) {};
\draw[color=green!50!black] (-2.2,0) .. controls +(0,0.75) and +(0,0) .. (-1.5,1);
\draw[color=green!50!black] (-2.2,0) .. controls +(0,-0.25) and +(0,0.25) .. (-0.8,-1.0);
\draw[color=green!50!black] (-0.8,0) .. controls +(0,0.75) and +(0,0) .. (-1.5,1);
\draw[color=white, line width=4pt] (-0.8,0) .. controls +(0,-0.25) and +(0,0.25) .. (-2.2,-1.0);
\draw[color=green!50!black] (-0.8,0) .. controls +(0,-0.25) and +(0,0.25) .. (-2.2,-1.0);
\draw[color=green!50!black] (-2.2,-1.0) -- (-2.2,-1.5);
\draw[color=green!50!black] (-0.8,-1.0) -- (-0.8,-1.5);
\fill (-0.8,-1.0) circle (4.5pt) node[right] (meet2) {{\scriptsize$\theta_A$}};
\end{tikzpicture} 
=
\;
\tikzzbox{\begin{tikzpicture}[very thick,scale=0.4,color=green!50!black, baseline=0cm]
\draw[-dot-] (0,0) .. controls +(0,1) and +(0,1) .. (-1,0);
\draw[directedgreen, color=green!50!black] (1,0) .. controls +(0,-1) and +(0,-1) .. (0,0);
\draw[redirectedgreen, color=green!50!black] (1,1.5) .. controls +(0,1.25) and +(0,1.25) .. (-2,1.5);
\draw (-1,0) -- (-1,-1.5); 
\draw (1,0) -- (1,1.5); 
\draw (-2,1.5) -- (-2,-1.5);
\draw (-0.5,1.4) node[Odot] (end) {}; 
\draw (-0.5,0.8) -- (end); 
\end{tikzpicture}}
\;
=
\;
\tikzzbox{\begin{tikzpicture}[very thick,scale=0.4,color=green!50!black, baseline=0cm]
\draw[redirectedgreen, color=green!50!black] (0,0) .. controls +(0,-1) and +(0,-1) .. (-1,0);
\draw[redirectedgreen, color=green!50!black] (-1,1.5) .. controls +(0,1.25) and +(0,1.25) .. (2,1.5);
\draw[-dot-] (1,0) .. controls +(0,1) and +(0,1) .. (0,0);
\draw (2,-1.5) -- (2,1.5); 
\draw (-1,0) -- (-1,1.5); 
\draw (1,0) -- (1,-1.5); 
\draw (0.5,1.4) node[Odot] (end) {}; 
\draw (0.5,0.8) -- (end); 
\end{tikzpicture}}
\;
=
\;
\begin{tikzpicture}[very thick,scale=0.4,color=green!50!black, baseline=0cm]
\draw (-1.5,2) node[Odot] (end) {}; 
\draw[color=green!50!black] (-1.5,1) -- (end); 
\fill[color=green!50!black] (-1.5,1) circle (4.5pt) node (mult1) {};
\draw[color=green!50!black] (-2.2,0) .. controls +(0,0.75) and +(0,0) .. (-1.5,1);
\draw[color=green!50!black] (-0.8,0) .. controls +(0,-0.25) and +(0,0.25) .. (-2.2,-1.0);
\draw[color=green!50!black] (-0.8,0) .. controls +(0,0.75) and +(0,0) .. (-1.5,1);
\draw[color=green!50!black] (-2.2,-1.0) -- (-2.2,-1.5);
\draw[color=green!50!black] (-0.8,-1.0) -- (-0.8,-1.5);
\draw[color=white, line width=4pt] (-2.2,0) .. controls +(0,-0.25) and +(0,0.25) .. (-0.8,-1.0);
\draw[color=green!50!black] (-2.2,0) .. controls +(0,-0.25) and +(0,0.25) .. (-0.8,-1.0);
\fill (-0.8,-1.0) circle (4.5pt) node[right] (meet2) {{\scriptsize$\theta_A^{-1}$}};
\end{tikzpicture} 
\, . 
\ee
$A$ is called
\textsl{$\Delta$-separable} if $\mu \circ \Delta = \id_A$, 
\be
\tikzzbox{\begin{tikzpicture}[very thick,scale=0.4,color=green!50!black, baseline=0cm]
\draw[-dot-] (0,0) .. controls +(0,-1) and +(0,-1) .. (1,0);
\draw[-dot-] (0,0) .. controls +(0,1) and +(0,1) .. (1,0);
\draw (0.5,-0.8) -- (0.5,-1.5); 
\draw (0.5,0.8) -- (0.5,1.5); 
\end{tikzpicture}}
\, = \, 
\tikzzbox{\begin{tikzpicture}[very thick,scale=0.4,color=green!50!black, baseline=0cm]
\draw (0.5,-1.5) -- (0.5,1.5); 
\end{tikzpicture}}
\, . 
\ee

If $A_1$ and $A_2$ are Frobenius algebras, then so is $A_1 \otimes A_2$ with product and unit as in Section~\ref{sec:modules}, and with coproduct
\be
\Delta_{A_1 \otimes A_2} = ( 1_{A_1} \otimes c^{-1}_{A_1,A_2} \otimes 1_{A_2} ) \circ ( \Delta_1 \otimes \Delta_2 )
=
\begin{tikzpicture}[very thick,scale=0.75,color=blue!50!black, baseline]
\draw[color=green!50!black] (-1,1) node[above] (A1) {{\scriptsize$A_1$}};
\draw[color=green!50!black] (0,1) node[above] (A2) {{\scriptsize$A_2$}};
\draw[color=green!50!black] (1,1) node[above] (A1r) {{\scriptsize$A_1$}};
\draw[color=green!50!black] (2,1) node[above] (A2r) {{\scriptsize$A_2$}};
\draw[color=green!50!black] (0,-1) node[below] (A1up) {{\scriptsize$A_1$}};
\draw[color=green!50!black] (1,-1) node[below] (A2up) {{\scriptsize$A_2$}};
\fill[color=green!50!black] (0,0) circle (2.5pt) node (mult1) {};
\fill[color=green!50!black] (1,0) circle (2.5pt) node (mult2) {};
\draw[color=green!50!black] (A1) -- (0,0);
\draw[color=green!50!black] (A1r) -- (0,0);
\draw[color=white, line width=4pt] (A2) -- (1,0);
\draw[color=green!50!black] (A2) -- (1,0);
\draw[color=green!50!black] (A2r) -- (1,0);
\fill[color=green!50!black] (1,0) circle (2.5pt) node (mult2) {};
\draw[color=green!50!black] (0,0) -- (A1up);
\draw[color=green!50!black] (1,0) -- (A2up);
\end{tikzpicture} 
\ee 
and counit $\eps_{A_1 \otimes A_2} = \eps_1 \otimes \eps_2$.
It is easy to check that if $A_1$ and $A_2$ are symmetric (resp.\ $\Delta$-separable), then also $A_1 \otimes A_2$ is symmetric (resp.\ $\Delta$-separable).

Recall the definition of $A^{\mathrm{op}}$ from \eqref{eq:muop}. 
For a Frobenius algebra $A$, $A^{\mathrm{op}}$ is a Frobenius algebra with  algebra structure \eqref{eq:muop} and coproduct $\Delta^{\mathrm{op}} = c_{A,A}^{-1} \circ \Delta$. 
It is $\Delta$-separable iff $A$ is.
Below we will use the notation 
\be
A^+ := A 
\, , \quad 
A^- := A^{\mathrm{op}} 
\, . 
\ee
For more details on Frobenius algebras and proofs of the above statements, see e.\,g.~\cite[Sect.\,3]{tft1}

\medskip

Recall that the junction map \eqref{eq:junction} takes values in cyclically symmetrised lists in $(D_2 \times \{ \pm \})^m$. To implement this in the present setting we require the following variant of Definition~\ref{def:symmetricmodule}:

\begin{definition}
\label{def:cyclic-with-signs}
For an $m>0$ and $i \in \{1,\dots,m\}$ let $A_i \in \Cat{C}$ be algebras and $\eps_i \in \{ \pm \}$. 
A \textsl{(maximally) cyclic multi-module $M$ for $((A_1,\eps_1),\dots,(A_m,\eps_m))$} is a $k$-cyclic $(A_1^{\eps_1},\dots, A_m^{\eps_m})$-multi-module $M$, where $k>0$ is minimal such that $(A_{i},\eps_{i}) = (A_{i+k},\eps_{i+k})$ for all $i$ (where $i+k$ is taken modulo $m$). 
\end{definition}

Since we will only ever use the maximal cyclic symmetry below, we will drop the qualifier ``maximally''.
We can now give the explicit description of line and surface defects in Reshetikhin-Turaev TQFT.

\begin{definition}
\label{def:defectdataMTC}
The set of defect data $\DC \equiv (D_1^{\Cat C},D_2^{\Cat C},D_3^{\Cat C},s,t,j)$ associated to a modular tensor category~$\Cat C$ is:
\begin{itemize}
\item
$D_3^{\Cat C} := \{ \Cat C \}$, 
\item 
$D_2^{\Cat C} := \{ \Delta\text{-separable symmetric Frobenius algebras in } \Cat C \}$, 
\item
$
D_1^{\Cat C} := \bigsqcup_{n \in \mathbb{Z}_{\geqslant 0}} L_n
$, 
where $L_0 = \big\{ X \in \Cat{C} \,\big|\, \theta_X = \id_X \big\}$,  
and, for $n>0$,
\begin{align*}
	L_n =\, &\big\{ 
	\big((A_1,\eps_1), (A_2,\eps_2), \dots, (A_n,\eps_n), M\big) \;\big|\; 
					A_i \in D_2^{\Cat C} , \; 
					\eps_i \in \{ \pm \} , \; 
	 \\ & \qquad
					M \text{ is a cyclic multi-module for }
					\big((A_1,\eps_1), (A_2,\eps_2), \dots, (A_n,\eps_n)\big)
					\big\}
\end{align*}
\item 
$s(A,\pm) \stackrel{\text{def}}{=} \Cat C \stackrel{\text{def}}{=} t(A,\pm)$ 
for all $A \in D_2^{\Cat C}$, 
\item 
$
j(M)
\stackrel{\text{def}}{=}  \Cat C$ 
for $n=0$, and  
\be
j \big( ((A_1,\eps_1), \dots, (A_n,\eps_n), M) \big) 
\stackrel{\text{def}}{=} 
((A_1,\eps_1), \dots, (A_n,\eps_n))/C_n
\ee
for $n>0$. 
 \end{itemize}
\end{definition}

To construct our defect TQFT \eqref{eq:whatwedo}, we want to reduce the evaluation of $\zzc$ on a $\DC$-decorated bordism\footnote{As noted after \eqref{eq:bord-rib-ext} and \eqref{eq:bord-def-ext}, the objects and morphisms 
in the domains $\Bordriben{3}(\Cat{C})$ and $\Borddefen{3}(\DC)$ of~$\zrt$ and~$\zzc$, respectively, 
are of the form $(\Sigma,\mathcal L)$ and $([N],m)$.
Since the Lagrangian subspaces~$\mathcal L$ and integers~$m$ are only spectators in our construction, we will suppress them in the notation.}~$N$ 
to the Reshetikhin-Turaev construction $\zrt(N)$ reviewed in Section~\ref{sec:MTCRT}. 
A brief summary of the construction below is this (see also Figure~\ref{fig:roughoutline} for an illustration): 
\begin{itemize}
\item[(1)]
Replace every 1-stratum in~$N$ decorated by $((A_1,\eps_1), \dots, (A_n,\eps_n), M) \in D_1^{\Cat C}$ with an $M$-decorated ribbon. 
\item[(2)]
For every 2-stratum decorated by $A_i \equiv (A_i,\mu_i,\Delta_i) \in D_2^{\Cat C}$, choose a triangulation~$t_i$, decorate the 1- and 0-strata of the Poincar\'{e} dual of~$t_i$ with~$A_i$ and~$\mu_i$ or~$\Delta_i$, respectively, and view the result as a ribbon graph embedded in~$N$, using the action $\rho_i \colon A_i \otimes M \to M$ to connect $A_i$-ribbons to the $M$-ribbon as dictated by the (dual of the) triangulation~$t_i$. 
\item[(3)]
Evaluate the resulting bordism $N(\{t_i\},\dots)$ with embedded ribbon graph with the functor $\zrt$. 
(The notation indicates that $N(\{t_i\},\dots)$ depends on the choice of triangulations~$t_i$, as well as other choices discussed below.) 
\item[(4)]
Take the limit of the resulting inverse system to produce $\zzc(N)$, independent of all choices. 
\end{itemize}

\begin{figure}[t]
$$
\begin{tikzpicture}[
			     baseline=(current bounding box.base), 
			     descr/.style={fill=white,inner sep=3.5pt}, 
			     normal line/.style={->}
			     ] 
\matrix (m) [matrix of math nodes, row sep=7.5em, column sep=0.6em, text height=0.5ex, text depth=0.1ex] {%
	\begin{minipage}{3cm}%
	\begin{tikzpicture}[very thick, scale=0.75, color=blue!50!black, baseline]
	%
	\draw[color=black] (1.8,1.4) node (X) {{\scriptsize$N$}};
	\clip (2,0) arc [radius=2, start angle=0, delta angle=360];
	\draw[color=gray, line width=0.5pt, densely dashed] (0,0) circle (2);
	%
	\fill [green!50!black,opacity=0.25] (0,2.2) -- (-2,2) -- (-2,-2) -- (0,-2) to[out=210, in=270] (2,-1.4) 
			to[out=180, in=270] (1.2,-0.3) to[out=90, in=300] (0.7,1.1) to[out=120, in=300] (0,2.2);
	%
	\draw[color=blue!50!black, >=stealth, decoration={markings, mark=at position 0.87 with {\arrow{>}},}, postaction={decorate}] 
			(2,-1.4) 
			to[out=180, in=270] (1.2,-0.3)
			to[out=90, in=300] (0.7,1.1)
			to[out=120, in=300] (0,2.2);
	%
	\draw (0.7,1.5) node (X) {{\scriptsize$M$}};
	\draw[color=green!50!black] (-0.4,-0.3) node (X) {{\scriptsize$A_i$}};
	\end{tikzpicture} 
	\end{minipage}%
 & & 
\begin{tikzpicture}[very thick, scale=0.75, color=blue!50!black, baseline]
%
\draw[color=black] (1.8,1.4) node (X) {{\scriptsize$N$}};
\clip (2,0) arc [radius=2, start angle=0, delta angle=360];
\draw[color=gray, line width=0.5pt, densely dashed] (0,0) circle (2);
%
\fill [green!50!black,opacity=0.25] (0,2.2) -- (-2,2) -- (-2,-2) -- (0,-2) to[out=210, in=270] (2,-1.4) 
		to[out=180, in=270] (1.2,-0.3) to[out=90, in=300] (0.7,1.1) to[out=120, in=300] (0,2.2);
%
\draw[color=blue!50!black, line width=3pt, >=stealth, decoration={markings, mark=at position 0.87 with {\arrow{>}},}, postaction={decorate}] 
		(2,-1.4) 
		to[out=180, in=270] (1.2,-0.3)
		to[out=90, in=300] (0.7,1.1)
		to[out=120, in=300] (0,2.2);
%
\draw (0.95,1.35) node (X) {{\scriptsize$M$}};
\draw[color=green!50!black] (-0.4,-0.3) node (X) {{\scriptsize$A_i$}};
\end{tikzpicture} 
& 
\hspace{1.5cm}
\begin{tikzpicture}[very thick, scale=0.75, color=blue!50!black, baseline]
%
\draw[color=black] (1.8,1.4) node (X) {{\scriptsize$N$}};
\clip (2,0) arc [radius=2, start angle=0, delta angle=360];
\draw[color=gray, line width=0.5pt, densely dashed] (0,0) circle (2);
%
\fill[color=black] (-1,-1.4) circle (1pt) node (v1) {};
\draw[line width=0.3pt, color=black, >=stealth, decoration={markings, mark=at position 0.7 with {\arrow{>}},}, postaction={decorate}] 
		(1.4,-1.1) -- (-1,-1.4); 
\draw[line width=0.3pt, color=black, >=stealth, decoration={markings, mark=at position 0.7 with {\arrow{>}},}, postaction={decorate}] 
		(1.05,0.4) -- (-1,-1.4); 
\draw[line width=0.3pt, color=black, >=stealth, decoration={markings, mark=at position 0.5 with {\arrow{>}},}, postaction={decorate}] 
		(-1,-1.4) -- (-2.5,1); 
\draw[line width=0.3pt, color=black, >=stealth, decoration={markings, mark=at position 0.5 with {\arrow{>}},}, postaction={decorate}] 
		(-1,-1.4) -- (-0.8,-1.9); 
%
\fill[color=green!50!black] (0.2,-1.02) circle (4.5pt) node (Delta) {};
\fill[color=green!50!black] (-0.2,0.3) circle (4.5pt) node (mu) {};
\draw[line width=3pt, color=black!50!green, >=stealth, decoration={markings, mark=at position 0.6 with {\arrow{>}},}, postaction={decorate}] 
		(0.2,-1.02) to[out=130, in=-60] (-0.2,0.3); 
\draw[line width=3pt, color=black!50!green, >=stealth, decoration={markings, mark=at position 0.7 with {\arrow{>}},}, postaction={decorate}] 
		(0.8,-1.9) to[out=90, in=270] (0.2,-1.02); 
\draw[line width=3pt, color=black!50!green, >=stealth, decoration={markings, mark=at position 0.7 with {\arrow{>}},}, postaction={decorate}] 
		(0.2,-1.02) to[out=66, in=245] (1.1,-0.3); 
\draw[line width=3pt, color=black!50!green, >=stealth, decoration={markings, mark=at position 0.7 with {\arrow{>}},}, postaction={decorate}] 
		(-0.2,0.3) to[out=90, in=235] (0.7,1.1); 
\draw[line width=3pt, color=black!50!green, >=stealth, decoration={markings, mark=at position 0.8 with {\arrow{>}},}, postaction={decorate}] 
		(-2.4,-1.9) to[out=40, in=250] (-0.2,0.3); 
%
\fill [green!50!black,opacity=0.07] (0,2.2) -- (-2,2) -- (-2,-2) -- (0,-2) to[out=210, in=270] (2,-1.4) 
		to[out=180, in=270] (1.2,-0.3) to[out=90, in=300] (0.7,1.1) to[out=120, in=300] (0,2.2);
%
\draw[color=blue!50!black, line width=3pt, >=stealth, decoration={markings, mark=at position 0.92 with {\arrow{>}},}, postaction={decorate}] 
		(2,-1.4) 
		to[out=180, in=270] (1.2,-0.3)
		to[out=90, in=300] (0.7,1.1)
		to[out=120, in=300] (0,2.2);
%
\draw (0.76,1.48) node (X) {{\scriptsize$M$}};
\draw[color=green!50!black] (-1,-0.4) node (X) {{\scriptsize$A_i$}};
%
\draw[line width=1pt, color=white, fill] (-0.2,0.2) node[inner sep=2.5pt, draw,fill=white, rounded corners=1pt] (R2) {{\tiny$\mu$}};
\draw[line width=1pt, color=green!50!black] (-0.2,0.2) node[inner sep=2.5pt,draw, rounded corners=1pt] (R) {{\tiny$\mu$}};
%
\draw[line width=1pt, color=white, fill] (0.2,-1.05) node[inner sep=2pt, draw,fill=white, rounded corners=1pt] (R2) {{\tiny$\Delta$}};
\draw[line width=1pt, color=green!50!black] (0.2,-1.05) node[inner sep=2pt,draw, rounded corners=1pt] (R) {{\tiny$\Delta$}};
\draw[color=black] (-1.6,-0.1) node (X) {{\scriptsize$t_i$}};
\draw[line width=1pt, color=white, fill] (1.2,-0.3) node[inner sep=2.5pt, draw,fill=white, rounded corners=1pt] (R2) {{\tiny$\rho_i$}};
\draw[line width=1pt, color=blue!50!black] (1.2,-0.3) node[inner sep=2.5pt,draw, rounded corners=1pt] (R) {{\tiny$\rho_i$}};
\draw[line width=1pt, color=white, fill] (0.7,1.1) node[inner sep=2.5pt, draw,fill=white, rounded corners=1pt] (R2) {{\tiny$\rho_i$}};
\draw[line width=1pt, color=blue!50!black] (0.7,1.1) node[inner sep=2.5pt,draw, rounded corners=1pt] (R) {{\tiny$\rho_i$}};
\end{tikzpicture} 
&
\\
 & \mathcal{Z}^{\mathcal C}( N ) &   & \mathcal{Z}^{\text{RT,}\mathcal C} \big( N(\{t_i\},\dots) \big) &
\\
};
\path[font=\footnotesize] ($(m-1-1)+(1.6,0)$) edge[->] node[above] {$(1)$} (m-1-3);
\path[font=\footnotesize] ($(m-1-3)+(1.6,0)$) edge[->] node[above] {$(2)$} ($(m-1-4)+(-0.95,0)$);
\path[font=\footnotesize] ($(m-1-4)+(0.75,-1.7)$) edge[->] node[right] {$(3)$} ($(m-2-4)+(0.75,0.3)$);
\path[font=\footnotesize] ($(m-2-4.west)+(0,0)$) edge[->] node[below] {$(4)$} (m-2-2);
\end{tikzpicture}
$$
\caption{Rough outline of the construction of $\mathcal{Z}^{\mathcal C}(N)$. Only a patch of the defect bordism~$N$ with one 1-stratum and one 2-stratum is shown.} 
\label{fig:roughoutline} 
\end{figure}
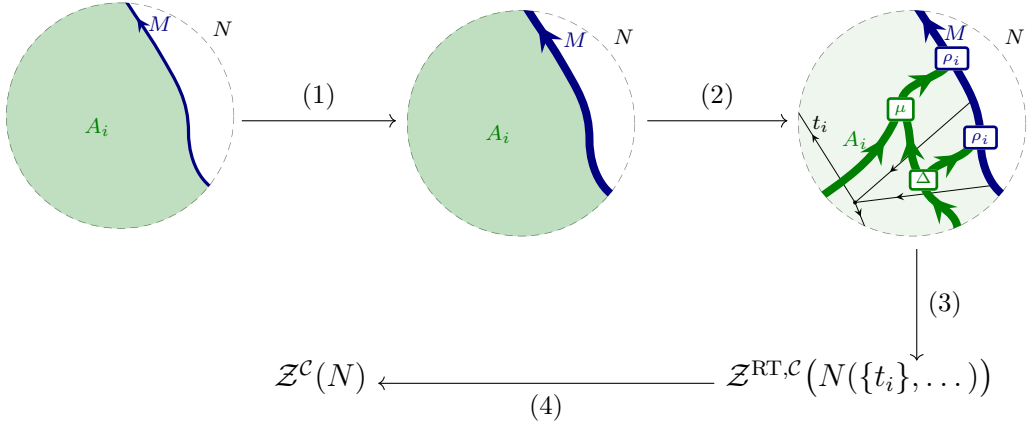

\begin{remark}
\begin{enumerate}
\item
Since the properties of $\Delta$-separable symmetric Frobenius algebras encode invariance under 2-dimensional Pachner moves, the value of $\zrt(N(\{t_i\},\dots))$ does not depend on the choice of triangulation in the interior of 2-strata, but only on the induced triangulation on the boundary $\partial N$, see \cite[Sect.\,5.1]{tft1}.
\item 
Since an $((A_1,\eps_1), \dots, (A_n,\eps_n), M)$-decorated 1-stratum ``sees'' the adjacent $A_i$-decorated 2-strata only up to cyclic symmetry, it is not immediately clear which of the twisted multi-modules $M^{\textrm{tw}_{j}}$ (recall Definition~\ref{def:twmm}) should be used to decorate the ribbons in steps (1) and (2) above. 
However, the notion of cyclic 
multi-modules is tailored to address this ambiguity. 
\end{enumerate}
\end{remark}

To begin to carry out the steps sketched above with proper care, we need the following auxiliary notion. 
It will serve as an intermediate ``reference point'' to deal with 
point (ii)
above and to evaluate $\zzc$ on $\DC$-decorated surfaces and bordisms. 

\begin{definition}
\label{def:stardec}
A \textsl{$*$-decoration} of a bordism~$N$ in $\Borddc$ is a choice of 3-stratum~$U$ for each $((A_1,\eps_1), \dots, (A_n,\eps_n), M)$-decorated 1-stratum~$L$, such that
\begin{enumerate}
\item
$U$ is adjacent to~$L$, 
\item 
$U$ is adjacent to 2-strata decorated by~$A_n$ and $A_1$ (there are~$n/k$ such pairs if $C_{n/k}$ is the maximal cyclic symmetry of~$M$).  
\end{enumerate}
For an object $\Sigma \in \Borddc$, a \textsl{$*$-decoration} is a choice of 2-stratum for each 0-stratum in~$\Sigma$ which is induced from a $*$-decoration of $\Sigma \times [0,1]$. 
\end{definition}

For example, for a cyclic multi-module ${}_{A_1\dots A_9}M$ with $k=3$ only the 3-strata between two 2-strata decorated~$A_3$ and~$A_1$ are allowed choices for a $*$-decoration. 
Hence a local neighbourhood of an ${}_{A_1\dots A_9}M$-decorated 1-stratum in a $*$-decorated bordism is a cylinder over one of these three configurations: 
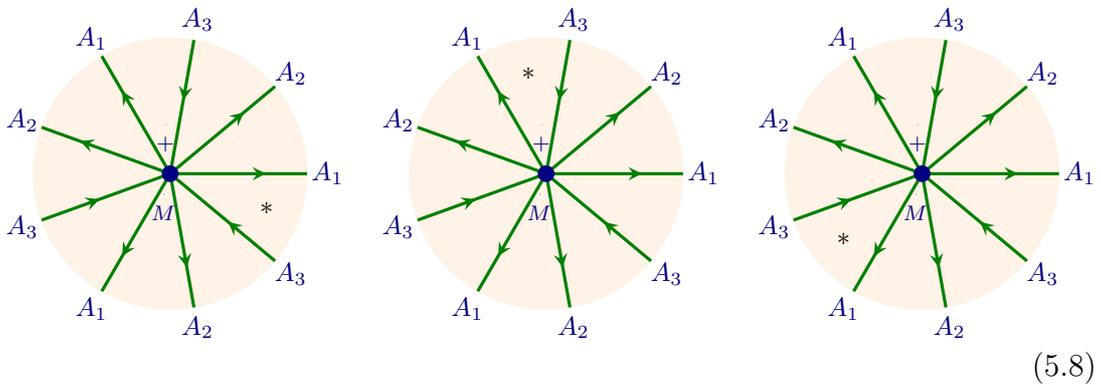
\be
\begin{tikzpicture}[very thick,scale=0.9,color=blue!50!black, baseline=-0.1cm]
\filldraw[orange!10, line width=0] (0,0) circle (2);
\foreach \z in {0, 1, ..., 2}
{
	\draw[color=green!50!black, >=stealth, decoration={markings, mark=at position 0.7 with {\arrow{>}},}, postaction={decorate}] (0,0) -- (\z * 120 + 0 :2);
}
\foreach \z in {0, 1, ..., 2}
{
	\draw[color=green!50!black, >=stealth, decoration={markings, mark=at position 0.7 with {\arrow{>}},}, postaction={decorate}] (0,0) -- (\z * 120 + 40 :2);
}
\foreach \z in {0, 1, ..., 2}
{
	\draw[color=green!50!black, >=stealth, decoration={markings, mark=at position 0.65 with {\arrow{<}},}, postaction={decorate}] (0,0) -- (\z * 120 + 80 :2);
}
\fill[color=blue!50!black] (0,0) circle (3.5pt) node[below] {};
\fill[color=blue!50!black] (-0.1,-0.3) circle (0pt) node[below] {{\scriptsize $M$}};
\fill[color=blue!50!black] (-0.07,0.72) circle (0pt) node[below] {{\scriptsize $+$}};
\foreach \x in {0,...,2}
	\draw[line width=1] (120*\x:2.3) node[line width=0pt] (Xbottom) {{\footnotesize $A_1$}};
\foreach \x in {0,...,2}
	\draw[line width=1] (120*\x + 40:2.3) node[line width=0pt] (Xbottom) {{\footnotesize $A_2$}};
\foreach \x in {0,...,2}
	\draw[line width=1] (120*\x + 80:2.3) node[line width=0pt] (Xbottom) {{\footnotesize $A_3$}};
\draw[line width=1, color=black] (-20:1.5) node[line width=0pt] (Xbottom) {{\footnotesize $*$}};
\end{tikzpicture}%
\;\;
\begin{tikzpicture}[very thick,scale=0.9,color=blue!50!black, baseline=-0.1cm]
\filldraw[orange!10, line width=0] (0,0) circle (2);
\foreach \z in {0, 1, ..., 2}
{
	\draw[color=green!50!black, >=stealth, decoration={markings, mark=at position 0.7 with {\arrow{>}},}, postaction={decorate}] (0,0) -- (\z * 120 + 0 :2);
}
\foreach \z in {0, 1, ..., 2}
{
	\draw[color=green!50!black, >=stealth, decoration={markings, mark=at position 0.7 with {\arrow{>}},}, postaction={decorate}] (0,0) -- (\z * 120 + 40 :2);
}
\foreach \z in {0, 1, ..., 2}
{
	\draw[color=green!50!black, >=stealth, decoration={markings, mark=at position 0.65 with {\arrow{<}},}, postaction={decorate}] (0,0) -- (\z * 120 + 80 :2);
}
\fill[color=blue!50!black] (0,0) circle (3.5pt) node[below] {};
\fill[color=blue!50!black] (-0.1,-0.3) circle (0pt) node[below] {{\scriptsize $M$}};
\fill[color=blue!50!black] (-0.07,0.72) circle (0pt) node[below] {{\scriptsize $+$}};
\foreach \x in {0,...,2}
	\draw[line width=1] (120*\x:2.3) node[line width=0pt] (Xbottom) {{\footnotesize $A_1$}};
\foreach \x in {0,...,2}
	\draw[line width=1] (120*\x + 40:2.3) node[line width=0pt] (Xbottom) {{\footnotesize $A_2$}};
\foreach \x in {0,...,2}
	\draw[line width=1] (120*\x + 80:2.3) node[line width=0pt] (Xbottom) {{\footnotesize $A_3$}};
\draw[line width=1, color=black] (100:1.5) node[line width=0pt] (Xbottom) {{\footnotesize $*$}};
\end{tikzpicture}%
\;\;
\begin{tikzpicture}[very thick,scale=0.9,color=blue!50!black, baseline=-0.1cm]
\filldraw[orange!10, line width=0] (0,0) circle (2);
\foreach \z in {0, 1, ..., 2}
{
	\draw[color=green!50!black, >=stealth, decoration={markings, mark=at position 0.7 with {\arrow{>}},}, postaction={decorate}] (0,0) -- (\z * 120 + 0 :2);
}
\foreach \z in {0, 1, ..., 2}
{
	\draw[color=green!50!black, >=stealth, decoration={markings, mark=at position 0.7 with {\arrow{>}},}, postaction={decorate}] (0,0) -- (\z * 120 + 40 :2);
}
\foreach \z in {0, 1, ..., 2}
{
	\draw[color=green!50!black, >=stealth, decoration={markings, mark=at position 0.65 with {\arrow{<}},}, postaction={decorate}] (0,0) -- (\z * 120 + 80 :2);
}
\fill[color=blue!50!black] (0,0) circle (3.5pt) node[below] {};
\fill[color=blue!50!black] (-0.1,-0.3) circle (0pt) node[below] {{\scriptsize $M$}};
\fill[color=blue!50!black] (-0.07,0.72) circle (0pt) node[below] {{\scriptsize $+$}};
\foreach \x in {0,...,2}
	\draw[line width=1] (120*\x:2.3) node[line width=0pt] (Xbottom) {{\footnotesize $A_1$}};
\foreach \x in {0,...,2}
	\draw[line width=1] (120*\x + 40:2.3) node[line width=0pt] (Xbottom) {{\footnotesize $A_2$}};
\foreach \x in {0,...,2}
	\draw[line width=1] (120*\x + 80:2.3) node[line width=0pt] (Xbottom) {{\footnotesize $A_3$}};
\draw[line width=1, color=black] (220:1.5) node[line width=0pt] (Xbottom) {{\footnotesize $*$}};
\end{tikzpicture}%
\ee
Note that our convention is that the 2-strata adjacent to an ${}_{A_1\dots A_n}M$-decorated 1-stratum are ordered $A_1, \dots, A_n$ anticlockwise with respect to the orientation of the 1-stratum. 

\medskip

Naturally, ribbon graphs feature prominently also in our extension of the Reshetikhin-Turaev construction, and it makes a difference in which direction a ribbon is twisted. 
Below we will use the following simplified depiction of a counter-clockwise half-twist: 
\be
\label{eq:uglyhalftwistdefined}
\begin{tikzpicture}[very thick,scale=0.75,color=blue!50!black, baseline]
\newcommand{\brei}{0.15}
\fill[color=green!50!black] (\brei,1) -- (-\brei,1) -- (-\brei,0) -- (\brei,0);	
\draw[color=green!50!black, semithick, double distance=34.7*\brei pt] (0,-1) -- (0,0); 
\end{tikzpicture} 
\; 
:= 
\; 
\begin{tikzpicture}[very thick,scale=0.75,color=blue!50!black, baseline]
\newcommand{\brei}{0.15}
%
\fill[color=green!50!black] (0,0)
	to[out=60, in=270] (\brei,0.4)
	-- (\brei,1)
	-- (-\brei,1)
	-- (-\brei,0.4)
	to[out=270, in=120] (0,0);
\fill [pattern=north west lines, opacity=0.3] (\brei,-1)
	-- (\brei,-0.4)
	to[out=90, in=-60] (0,0)
	to[out=240, in=90] (-\brei,-0.4)
	-- (-\brei,-1);
\fill[color=green!50!black, opacity = 0.2] (\brei,-1)
	-- (\brei,-0.4)
	to[out=90, in=-60] (0,0)
	to[out=240, in=90] (-\brei,-0.4)
	-- (-\brei,-1);
\draw[line width=0.3pt, color=black] (\brei,-1) 		
							-- (\brei,-0.4)
							to[out=90, in=-60] (0,0)
							to[out=120, in=270] (-\brei,0.4)
							-- (-\brei,1);
\node[inner sep=1.5pt, fill, white] at (0,0) {};
\draw[line width=0.3pt, color=black] (-\brei,-1) 		
							-- (-\brei,-0.4)
							to[out=90, in=240] (0,0)
							to[out=60, in=270] (\brei,0.4)
							-- (\brei,1);
%
\end{tikzpicture} 
\ee
Here, the dark green colour is on the ``front side'' of the ribbon, i.\,e.\ the stretch where the ribbon 2-orientation matches that of the paper plane, and the lighter colour marks the ``back side'' where the ribbon orientation is opposite to that of the paper plane.
 
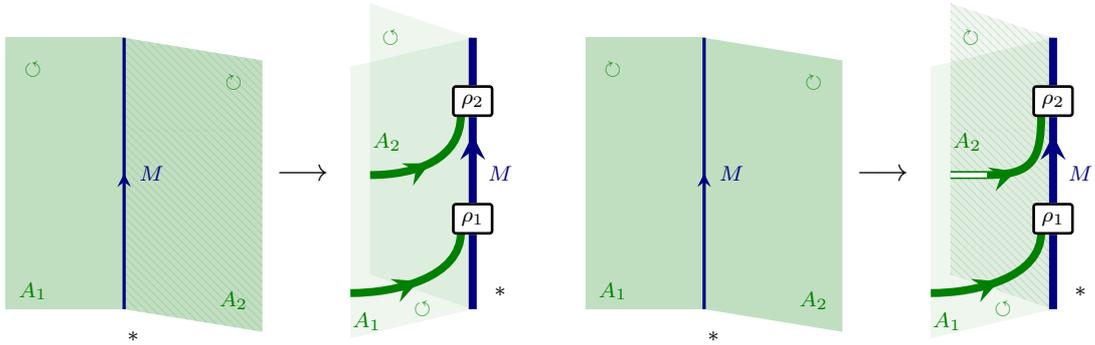
\begin{figure}[t]
\begin{align*}
\tikzzbox{\begin{tikzpicture}[very thick,scale=1.2,color=blue!50!black, baseline=-1.9cm]
%
\fill [green!50!black,opacity=0.25] (0,0) -- (-1.3,0) -- (-1.3,-3) -- (0,-3) -- (0,0);
\fill [pattern=north west lines, opacity=0.3] (0,0) -- ($(0,0)+(330:1.75 and 0.5)$) -- ($(0,-3)+(330:1.75 and 0.5)$) -- (0,-3) -- (0,0);
\fill [green!50!black,opacity=0.25] (0,0) -- ($(0,0)+(330:1.75 and 0.5)$) -- ($(0,-3)+(330:1.75 and 0.5)$) -- (0,-3) -- (0,0);
%
 \draw[
	color=blue!50!black, 
	>=stealth,
	decoration={markings, mark=at position 0.5 with {\arrow{>}},
					}, postaction={decorate}
	] 
(0,-3) -- (0,0);
%
%
\draw[line width=1, color=blue!50!black] (0.3, -1.5) node[line width=0pt] (beta) {{\scriptsize$M$}};
\draw[line width=1, color=green!50!black] (1.2, -2.9) node[line width=0pt] (beta) {{\scriptsize$A_2$}};
\draw[line width=1, color=green!50!black] (1.2, -0.5) node[line width=0pt] (beta) {{\scriptsize$\circlearrowright$}};
\draw[line width=1, color=green!50!black] (-1, -2.8) node[line width=0pt] (beta) {{\scriptsize$A_1$}};
\draw[line width=1, color=green!50!black] (-1, -0.35) node[line width=0pt] (beta) {{\scriptsize$\circlearrowleft$}};
\draw[line width=1, color=black] (0.1, -3.3) node[line width=0pt] (beta) {{\scriptsize$*$}};
\end{tikzpicture}}
\,
\lra 
\,
\tikzzbox{\begin{tikzpicture}[very thick,scale=1.2,color=blue!50!black, baseline=-1.9cm]
%
\fill [green!50!black,opacity=0.07] (0,0) -- ($(0,0)+(130:1.75 and 0.5)$) -- ($(0,-3)+(130:1.75 and 0.5)$) -- (0,-3) -- (0,0);
\fill [green!50!black,opacity=0.07] (0,0) -- ($(0,0)+(220:1.75 and 0.5)$) -- ($(0,-3)+(220:1.75 and 0.5)$) -- (0,-3) -- (0,0);
\draw[
	color=green!50!black, 
	line width=3pt, 
	>=stealth,
	decoration={markings, mark=at position 0.5 with {\arrow{>}},
					}, postaction={decorate}
	]
	($(0,-1.9)+(130:1.75 and 0.5)$) .. controls +(0.5,0) and +(0,-0.5) .. (-0.13,-0.85);
\draw[
	color=green!50!black, 
	line width=3pt, 
	>=stealth,
	decoration={markings, mark=at position 0.5 with {\arrow{>}},
					}, postaction={decorate}
	]
	($(0,-2.5)+(220:1.75 and 0.5)$) .. controls +(0.5,0) and +(0,-0.5) .. (-0.13,-2.15);
\fill[color=white] (0,-2) node[inner sep=3pt, fill, color=white, rounded corners=1pt] (R1) {{\scriptsize$\rho_1$}};
\draw[line width=1pt, color=black] (0,-2) node[inner sep=3pt,draw, rounded corners=1pt] (R1) {{\scriptsize$\rho_1$}};
\fill[color=white] (0,-0.7) node[inner sep=3pt, fill, color=white, rounded corners=1pt] (R2) {{\scriptsize$\rho_2$}};
\draw[line width=1pt, color=black] (0,-0.7) node[inner sep=3pt,draw, rounded corners=1pt] (R2) {{\scriptsize$\rho_2$}};
 \draw[
	color=blue!50!black, 
	line width=3pt, 
	>=stealth,
	decoration={markings, mark=at position 0.5 with {\arrow{<}},
					}, postaction={decorate}
	] 
(R2) -- (R1);
\draw[color=blue!50!black, line width=3pt] (0,-3) -- (R1);
\draw[color=blue!50!black, line width=3pt] (R2) -- (0,0);
%
%
\draw[line width=1, color=blue!50!black] (0.3, -1.5) node[line width=0pt] (beta) {{\scriptsize$M$}};
\draw[line width=1, color=green!50!black] (-0.94, -1.15) node[line width=0pt] (beta) {{\scriptsize$A_2$}};
\draw[line width=1, color=green!50!black] (-1.16, -3.14) node[line width=0pt] (beta) {{\scriptsize$A_1$}};
\draw[line width=1, color=green!50!black] (-0.9, 0) node[line width=0pt] (beta) {{\scriptsize$\circlearrowleft$}};
\draw[line width=1, color=green!50!black] (-0.55, -3) node[line width=0pt] (beta) {{\scriptsize$\circlearrowleft$}};
\draw[line width=1, color=black] (0.3, -2.8) node[line width=0pt] (beta) {{\scriptsize$*$}};
\end{tikzpicture}}
\qquad
\tikzzbox{\begin{tikzpicture}[very thick,scale=1.2,color=blue!50!black, baseline=-1.9cm]
%
\fill [green!50!black,opacity=0.25] (0,0) -- (-1.3,0) -- (-1.3,-3) -- (0,-3) -- (0,0);
\fill [green!50!black,opacity=0.25] (0,0) -- ($(0,0)+(330:1.75 and 0.5)$) -- ($(0,-3)+(330:1.75 and 0.5)$) -- (0,-3) -- (0,0);
%
 \draw[
	color=blue!50!black, 
	>=stealth,
	decoration={markings, mark=at position 0.5 with {\arrow{>}},
					}, postaction={decorate}
	] 
(0,-3) -- (0,0);
%
%
\draw[line width=1, color=blue!50!black] (0.3, -1.5) node[line width=0pt] (beta) {{\scriptsize$M$}};
\draw[line width=1, color=green!50!black] (1.2, -2.9) node[line width=0pt] (beta) {{\scriptsize$A_2$}};
\draw[line width=1, color=green!50!black] (1.2, -0.5) node[line width=0pt] (beta) {{\scriptsize$\circlearrowright$}};
\draw[line width=1, color=green!50!black] (-1, -2.8) node[line width=0pt] (beta) {{\scriptsize$A_1$}};
\draw[line width=1, color=green!50!black] (-1, -0.35) node[line width=0pt] (beta) {{\scriptsize$\circlearrowright$}};
\draw[line width=1, color=black] (0.1, -3.3) node[line width=0pt] (beta) {{\scriptsize$*$}};
\end{tikzpicture}}
\,
\lra 
\,
\tikzzbox{\begin{tikzpicture}[very thick,scale=1.2,color=blue!50!black, baseline=-1.9cm]
%
\fill [pattern=north west lines, opacity=0.3] (0,0) -- ($(0,0)+(130:1.75 and 0.5)$) -- ($(0,-3)+(130:1.75 and 0.5)$) -- (0,-3) -- (0,0);
\fill [green!50!black,opacity=0.07] (0,0) -- ($(0,0)+(130:1.75 and 0.5)$) -- ($(0,-3)+(130:1.75 and 0.5)$) -- (0,-3) -- (0,0);
\fill [green!50!black,opacity=0.07] (0,0) -- ($(0,0)+(220:1.75 and 0.5)$) -- ($(0,-3)+(220:1.75 and 0.5)$) -- (0,-3) -- (0,0);
\draw[
	color=green!50!black, 
	line width=3pt, 
	>=stealth, 
	semithick, double distance=1.8pt
	]
	($(0,-1.9)+(130:1.75 and 0.5)$) -- ($(0.4,-1.9)+(130:1.75 and 0.5)$);
\draw[
	color=green!50!black, 
	line width=3pt, 
	>=stealth,
	decoration={markings, mark=at position 0.3 with {\arrow{>}},
					}, postaction={decorate}
	]
	($(0.4,-1.9)+(130:1.75 and 0.5)$) .. controls +(0.5,0) and +(0,-0.5) .. (-0.13,-0.85);
\draw[
	color=green!50!black, 
	line width=3pt, 
	>=stealth,
	decoration={markings, mark=at position 0.5 with {\arrow{>}},
					}, postaction={decorate}
	]
	($(0,-2.5)+(220:1.75 and 0.5)$) .. controls +(0.5,0) and +(0,-0.5) .. (-0.13,-2.15);
\fill[color=white] (0,-2) node[inner sep=3pt, fill, color=white, rounded corners=1pt] (R1) {{\scriptsize$\rho_1$}};
\draw[line width=1pt, color=black] (0,-2) node[inner sep=3pt,draw, rounded corners=1pt] (R1) {{\scriptsize$\rho_1$}};
\fill[color=white] (0,-0.7) node[inner sep=3pt, fill, color=white, rounded corners=1pt] (R2) {{\scriptsize$\rho_2$}};
\draw[line width=1pt, color=black] (0,-0.7) node[inner sep=3pt,draw, rounded corners=1pt] (R2) {{\scriptsize$\rho_2$}};
 \draw[
	color=blue!50!black, 
	line width=3pt, 
	>=stealth,
	decoration={markings, mark=at position 0.5 with {\arrow{<}},
					}, postaction={decorate}
	] 
(R2) -- (R1);
\draw[color=blue!50!black, line width=3pt] (0,-3) -- (R1);
\draw[color=blue!50!black, line width=3pt] (R2) -- (0,0);
%
%
\draw[line width=1, color=blue!50!black] (0.3, -1.5) node[line width=0pt] (beta) {{\scriptsize$M$}};
\draw[line width=1, color=green!50!black] (-0.94, -1.15) node[line width=0pt] (beta) {{\scriptsize$A_2$}};
\draw[line width=1, color=green!50!black] (-1.16, -3.14) node[line width=0pt] (beta) {{\scriptsize$A_1$}};
\draw[line width=1, color=green!50!black] (-0.9, 0) node[line width=0pt] (beta) {{\scriptsize$\circlearrowright$}};
\draw[line width=1, color=green!50!black] (-0.55, -3) node[line width=0pt] (beta) {{\scriptsize$\circlearrowleft$}};
\draw[line width=1, color=black] (0.3, -2.8) node[line width=0pt] (beta) {{\scriptsize$*$}};
\end{tikzpicture}}
\end{align*}
\caption{Two examples of orienting ribbons relative to $*$-decoration; the right picture involves the half-twist of~\eqref{eq:halftwistAi}.} 
\label{fig:ribbonoriillu} 
\end{figure}

\medskip

We now describe how to turn a decorated stratified bordism into a bordism with ribbon graph on which the Reshetikhin-Turaev TQFT can act.

\begin{construction}
\label{con:ZRTC}
Let $N \colon \Sigma_1 \to \Sigma_2$ be a bordism in $\Borddc$. 
\begin{enumerate}
\item
Choose a $*$-decoration $*$ of $N$. 
\item 
Choose a triangulation\footnote{Here and below by ``triangulation'' we always mean ``triangulation with orientation induced from a choice of total order on the vertices'', as described e.\,g.~in \cite[Sect.\,3.1]{CRS1}. Then the Poincar\'{e} dual of a triangulation inherits an orientation, and we can decorate its 
	negatively (resp.~positively) 
oriented 0-strata with the multiplication (resp.~comultiplication) of Frobenius algebras.} $t_i$ for each $A_i$-decorated 2-stratum of $N$. 
\item 
Decorate the interior of the Poincar\'{e} dual of every triangulation~$t_i$ with the data $(A_i,\mu_i,\Delta_i)$ of the $\Delta$-separable symmetric Frobenius algebra~$A_i$. 
\item
Thicken the resulting $M$- and $A_i$-decorated lines to ribbons. 
In detail: 
	\begin{enumerate}
	\item
	In the interior of an $A_i$-decorated 2-stratum~$F$, give all $A_i$-ribbons the same 2-orientation as~$F$. 
	\item 
	In a neighbourhood of an $M$-decorated 1-stratum~$L$, consider a chart in which all adjacent 2-strata are to the left of the upward-oriented~$L$, and the $*$-decorated 3-stratum is to the right (see Figure~\ref{fig:Aiorientations}). 
	In this chart, orient $M$ in the paper plane and connect all $A_i$-ribbons to~$M$ with the coupons~$\rho_i$ as follows:
		\begin{itemize}
		\item 
		If an $A_i$-decorated 2-stratum~$F$ has the same orientation as the adjacent $M$-ribbon (i.\,e.~iff $\eps_i=+$, see Figure~\ref{fig:Aiorientations}\,(i)), connect the $A_i$-ribbons in the interior of~$F$ with those near~$M$ directly. 
		\item
		If the orientations do not agree ($\eps_i=-$, Figure~\ref{fig:Aiorientations}\,(ii)), perform the half-twist of~\eqref{eq:uglyhalftwistdefined} 
		on~$A_i$ near~$M$ before connecting: 
		\be
		
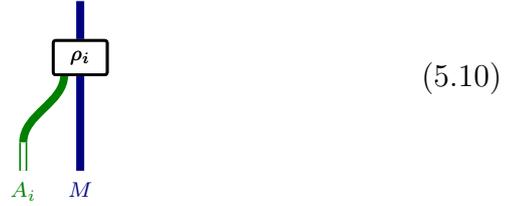
\label{eq:halftwistAi}
\begin{tikzpicture}[very thick,scale=0.75,color=blue!50!black, baseline]
\draw (0,-1) node[below] (X) {};
\draw[line width=1pt, color=black] (0,0.5) node[inner sep=4pt,draw, rounded corners=1pt] (R) {{\scriptsize$\;\rho_i\;$}};
\draw[line width=3pt] (0,-1.5) -- (R); 
\draw[line width=3pt] (0,1.5) -- (R); 
\draw[color=green!50!black, line width=3pt] (-1,-1) .. controls +(0,0.5) and +(0,-0.5) .. (-0.28,0.2);
\draw (0,-1.5) node[below] (X) {{\scriptsize$M$}};
\draw[color=green!50!black] (-1,-1.5) node[below] (A1) {{\scriptsize$A_i$}};
\draw[line width=1pt, color=black] (0,0.5) node[inner sep=4pt,draw, rounded corners=1pt] (R) {{\scriptsize$\;\rho_i\;$}};
\draw[color=green!50!black, semithick, double distance=1.8pt] (-1,-1.5) -- (-1,-1.0); 
\end{tikzpicture} 
\ee
\end{itemize}
\end{enumerate}
This produces a bordism with embedded ribbon graph 
\be
\label{eq:NSS}
N ( t, *)  \colon
\Sigma_1( \tau_1, *_1)
\lra
\Sigma_2( \tau_2, *_2) 
\ee
where $t$ denotes the totality of all triangulations~$t_i$, inducing the triangulations $\tau_1, \tau_2$ of 1-strata on the boundaries $\Sigma_1, \Sigma_2$, respectively, while the $*$-decorations $*_1, *_2$ are similarly induced by~$*$. 
\item 
Apply the Reshetikhin-Turaev construction to obtain a linear map
\be
\label{eq:ZRTCstar}
\zrt\big( N ( t, *) \big) \colon
\zrt \big( \Sigma_1( \tau_1, *_1) \big)
\lra
\zrt\big( \Sigma_2( \tau_2, *_2) \big)
\, . 
\ee
\end{enumerate}
\end{construction}

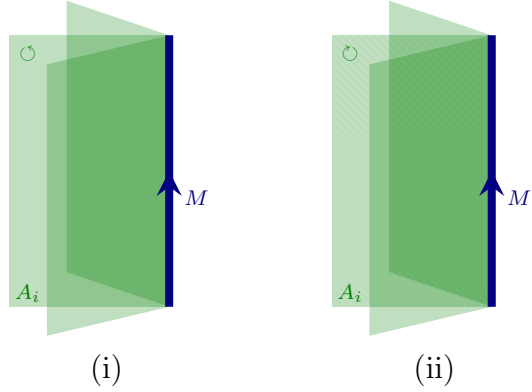
\begin{figure}[t]
\centering
\begin{tabular}{cc}
\tikzzbox{\begin{tikzpicture}[very thick,scale=1.2,color=blue!50!black, baseline=-1.9cm]
%
\fill [green!50!black,opacity=0.25] (0,0) -- ($(0,0)+(130:1.75 and 0.5)$) -- ($(0,-3)+(130:1.75 and 0.5)$) -- (0,-3) -- (0,0);
\fill [green!50!black,opacity=0.25] (0,0) -- ($(0,0)+(180:1.75 and 0.5)$) -- ($(0,-3)+(180:1.75 and 0.5)$) -- (0,-3) -- (0,0);
\fill [green!50!black,opacity=0.25] (0,0) -- ($(0,0)+(220:1.75 and 0.5)$) -- ($(0,-3)+(220:1.75 and 0.5)$) -- (0,-3) -- (0,0);
%
 \draw[
	color=blue!50!black, 
	line width=3pt, 
	>=stealth,
	decoration={markings, mark=at position 0.5 with {\arrow{>}},
					}, postaction={decorate}
	] 
(0,-3) -- (0,0);
%
\draw[line width=1, color=blue!50!black] (0.3, -1.8) node[line width=0pt] (beta) {{\scriptsize$M$}};
\draw[line width=1, color=green!50!black] (-1.55, -2.85) node[line width=0pt] (beta) {{\scriptsize$A_i$}};
\draw[line width=1, color=green!50!black] (-1.55, -0.2) node[line width=0pt] (beta) {{\scriptsize$\circlearrowleft$}};
\end{tikzpicture}}
\hspace{1cm}
& 
\tikzzbox{\begin{tikzpicture}[very thick,scale=1.2,color=blue!50!black, baseline=-1.9cm]
%
\fill [green!50!black,opacity=0.25] (0,0) -- ($(0,0)+(130:1.75 and 0.5)$) -- ($(0,-3)+(130:1.75 and 0.5)$) -- (0,-3) -- (0,0);
\fill [pattern=north west lines, opacity=0.2] (0,0) -- ($(0,0)+(180:1.75 and 0.5)$) -- ($(0,-3)+(180:1.75 and 0.5)$) -- (0,-3) -- (0,0);
\fill [green!50!black,opacity=0.25] (0,0) -- ($(0,0)+(180:1.75 and 0.5)$) -- ($(0,-3)+(180:1.75 and 0.5)$) -- (0,-3) -- (0,0);
\fill [green!50!black,opacity=0.25] (0,0) -- ($(0,0)+(220:1.75 and 0.5)$) -- ($(0,-3)+(220:1.75 and 0.5)$) -- (0,-3) -- (0,0);
%
 \draw[
	color=blue!50!black, 
	line width=3pt, 
	>=stealth,
	decoration={markings, mark=at position 0.5 with {\arrow{>}},
					}, postaction={decorate}
	] 
(0,-3) -- (0,0);
%
\draw[line width=1, color=blue!50!black] (0.3, -1.8) node[line width=0pt] (beta) {{\scriptsize$M$}};
\draw[line width=1, color=green!50!black] (-1.55, -2.85) node[line width=0pt] (beta) {{\scriptsize$A_i$}};
\draw[line width=1, color=green!50!black] (-1.55, -0.2) node[line width=0pt] (beta) {{\scriptsize$\circlearrowright$}};
\end{tikzpicture}}
\\
\addlinespace
(i) \hspace{1cm} & (ii)
\\
\end{tabular}
\caption{An $A_i$-decorated 2-stratum can have the same (i) or opposite (ii) orientation as the adjacent $M$-ribbon.} 
\label{fig:Aiorientations} 
\end{figure}

The source and target vector spaces of the map \eqref{eq:ZRTCstar} depend both on the choice of triangulations and on the $*$-decoration of~$N$. 
We will later remove this dependence via a limit construction. 
Before doing this we establish that $\zrt(N ( t, *))$ is independent of choices not visible on the boundary:
	
\begin{lemma}
\label{lem:Z-indep-of-interior-choices}
The map \eqref{eq:ZRTCstar} is invariant under isomorphisms of~$N$ and independent of the choice of triangulations in the interior of~$N$, and of the choice of $*$-decoration for each 1-stratum entirely contained in the interior of~$N$. 
\end{lemma}

\begin{proof}
The invariance under isomorphisms of the stratified manifold $N$ relative to the boundary follows from the invariance of the Reshitikhin-Turaev construction under homeomorphisms of $N$.

\textsl{Triangulation:}
If the orientation of $A_i$-ribbons agrees with that of the associated $M$-ribbon, independence of triangulation in the interior of 2- and 1-strata is a direct consequence of the~$A_i$ being $\Delta$-separable symmetric Frobenius algebras and~$M$ being $A_i$-modules, respectively. 
What remains to be verified is that the prescription \eqref{eq:halftwistAi} for oppositely oriented 2-strata makes the construction triangulation invariant also in the case $\eps_i=-$, when~$M$ is an $A_i^{\text{op}}$-module. 

Recall our convention \eqref{eq:muop} for the multiplication of opposite algebras, and the identity
\be
\label{eq:halfassed}
\begin{tikzpicture}[line width=3pt,scale=0.75,color=blue!50!black, baseline]
\fill[color=green!50!black] (-0.85,0) circle (4pt) node (mult1) {};
\draw[color=green!50!black, line width=3pt] (-1.2,-0.3) .. controls +(0,0.25) and +(0,0) .. (-0.85,0);
\draw[color=green!50!black, line width=3pt] (-1.2,-0.3) .. controls +(0,-0.25) and +(0,0.25) .. (-0.5,-1);
\draw[color=green!50!black, line width=3pt] (-0.5,-0.3) .. controls +(0,0.25) and +(0,0) .. (-0.85,0);
\draw[color=white, line width=5pt] (-0.5,-0.3) .. controls +(0,-0.25) and +(0,0.25) .. (-1.2,-1);
\draw[color=green!50!black, line width=3pt] (-0.5,-0.3) .. controls +(0,-0.25) and +(0,0.25) .. (-1.2,-1);
\draw[color=green!50!black, line width=3pt] (-0.85,0) -- (-0.85,0.75);
\draw[color=green!50!black, semithick, double distance=1.8pt] (-0.5,-1.5) -- (-0.5,-1); 
\draw[color=green!50!black, semithick, double distance=1.8pt] (-1.2,-1.5) -- (-1.2,-1); 
\draw[line width=1pt, color=white, fill] (-0.85,0) node[inner sep=1.5pt,draw,fill=white] (R2) {{\tiny$\mu$}};
\draw[line width=1.2pt, color=green!50!black] (-0.85,0) node[inner sep=1.5pt,draw, rounded corners=1pt] (R) {{\tiny$\mu$}};
\end{tikzpicture} 
=
\begin{tikzpicture}[line width=3pt,scale=0.75,color=blue!50!black, baseline]
\draw[color=green!50!black] (-0.85,0.75) -- (-0.85,0); 
\draw[color=green!50!black, semithick, double distance=1.8pt] (-0.85,-0.5) -- (-0.85,0); 
\draw[color=green!50!black, semithick, double distance=1.8pt] (-0.5,-1.5) -- (-0.5,-0.8); 
\draw[color=green!50!black, semithick, double distance=1.8pt] (-1.2,-1.5) -- (-1.2,-0.8); 
\draw[color=green!50!black, semithick, double distance=1.8pt] (-1.2,-0.8) .. controls +(0,0.25) and +(0,0) .. (-0.85,-0.5);
\draw[color=green!50!black, semithick, double distance=1.8pt] (-0.5,-0.8) .. controls +(0,0.25) and +(0,0) .. (-0.85,-0.5);
\draw[line width=1pt, color=white, fill] (-0.85,-0.5) node[inner sep=1.5pt,draw,fill=white] (R2) {{\tiny$\mu$}};
\draw[line width=1pt, color=green!50!black] (-0.85,-0.5) node[inner sep=1.5pt,draw, rounded corners=1pt] (R) {{\tiny$\mu$}};
\end{tikzpicture} 
\ee
of ribbon graphs, 
where we continue to use the notation of \eqref{eq:uglyhalftwistdefined}. 
Then we compute 
\be
\begin{tikzpicture}[very thick,scale=0.75,color=blue!50!black, baseline]
\draw (0,-1.5) node[below] (X) {{\scriptsize$M$}};
\draw[color=green!50!black] (-0.5,-1.5) node[below] (X) {{\scriptsize$A_i$}};
\draw[color=green!50!black] (-1,-1.5) node[below] (X) {{\scriptsize$A_i$}};
\draw (0,-1) node[below] (X) {};
\draw[color=green!50!black] (-0.5,-1) node[below] (X) {};
\draw[color=green!50!black] (-1,-1) node[below] (X) {};
\draw[line width=3pt] (0,-1.5) -- (0,1); 
\draw[color=green!50!black, semithick, double distance=1.8pt] (-0.5,-1.5) -- (-0.5,-1); 
\draw[color=green!50!black, semithick, double distance=1.8pt] (-1,-1.5) -- (-1,-1); 
\draw[color=green!50!black, line width=3pt] (-0.5,-1) .. controls +(0,0.25) and +(-0.25,-0.25) .. (0,-0.25);
\draw[color=green!50!black, line width=3pt] (-1,-1) .. controls +(0,0.5) and +(-0.5,-0.5) .. (0.05,0.6);
\draw[line width=1pt, color=white, fill] (0,0.6) node[inner sep=1.5pt,draw,fill=white, rounded corners=1pt] (R2) {{\tiny$\rho_i$}};
\draw[line width=1pt, color=blue!50!black] (0,0.6) node[inner sep=1.5pt,draw, rounded corners=1pt] (R) {{\tiny$\rho_i$}};
\draw[line width=1pt, color=white, fill] (0,-0.25) node[inner sep=1.5pt,draw,fill=white, rounded corners=1pt] (R2) {{\tiny$\rho_i$}};
\draw[line width=1pt, color=blue!50!black] (0,-0.25) node[inner sep=1.5pt,draw, rounded corners=1pt] (R) {{\tiny$\rho_i$}};
\end{tikzpicture} 
=
\begin{tikzpicture}[line width=3pt,scale=0.75,color=blue!50!black, baseline]
\draw (0,-1.5) node[below] (X) {{\scriptsize$M$}};
\draw[color=green!50!black] (-0.5,-1.5) node[below] (X) {{\scriptsize$A_i$}};
\draw[color=green!50!black] (-1.2,-1.5) node[below] (X) {{\scriptsize$A_i$}};
\draw (0,-1) node[below] (X) {};
\draw[color=green!50!black] (-0.5,-1) node[below] (X) {};
\draw[color=green!50!black] (-1,-1) node[below] (X) {};
\draw[line width=3pt] (0,-1.5) -- (0,1); 
\draw[color=green!50!black, semithick, double distance=1.8pt] (-0.5,-1.5) -- (-0.5,-1); 
\draw[color=green!50!black, semithick, double distance=1.8pt] (-1.2,-1.5) -- (-1.2,-1); 
\draw[color=green!50!black] (-0.5,-1) .. controls +(0,0.5) and +(0,0) .. (-0.85,0);
\draw[color=green!50!black] (-1.2,-1) .. controls +(0,0.5) and +(0,0) .. (-0.85,0);
\draw[color=green!50!black] (-0.85,0) .. controls +(0,0.5) and +(0,0) .. (0,0.6);
\draw[line width=1pt, color=white, fill] (0,0.6) node[inner sep=1.5pt,draw,fill=white, rounded corners=1pt] (R2) {{\tiny$\rho_i$}};
\draw[line width=1pt, color=blue!50!black] (0,0.6) node[inner sep=1.5pt,draw, rounded corners=1pt] (R) {{\tiny$\rho_i$}};
\draw[line width=1pt, color=white, fill] (-0.85,0) node[inner sep=1.5pt,draw,fill=white] (R2) {{\tiny$\mu$}};
\draw[line width=1pt, color=green!50!black] (-0.85,0) node[inner sep=1.5pt,draw, rounded corners=1pt] (R) {{\tiny$\mu^{\text{op}}$}};
\end{tikzpicture} 
=
\begin{tikzpicture}[line width=3pt,scale=0.75,color=blue!50!black, baseline]
\draw (0,-1.5) node[below] (X) {{\scriptsize$M$}};
\draw[color=green!50!black] (-0.5,-1.5) node[below] (X) {{\scriptsize$A_i$}};
\draw[color=green!50!black] (-1.2,-1.5) node[below] (X) {{\scriptsize$A_i$}};
\draw (0,-1.5) -- (0,1); 
\fill[color=green!50!black] (0,0.6) circle (2.5pt) node (meet2) {};
\draw[color=green!50!black, line width=3pt] (-1.2,-0.3) .. controls +(0,0.25) and +(0,0) .. (-0.85,0);
\draw[color=green!50!black, line width=3pt] (-1.2,-0.3) .. controls +(0,-0.25) and +(0,0.25) .. (-0.5,-1);
\draw[color=green!50!black, line width=3pt] (-0.5,-0.3) .. controls +(0,0.25) and +(0,0) .. (-0.85,0);
\draw[color=white, line width=5pt] (-0.5,-0.3) .. controls +(0,-0.25) and +(0,0.25) .. (-1.2,-1);
\draw[color=green!50!black, line width=3pt] (-0.5,-0.3) .. controls +(0,-0.25) and +(0,0.25) .. (-1.2,-1);
\draw[color=green!50!black, line width=3pt] (-0.85,0) .. controls +(0,0.5) and +(0,0) .. (0,0.6);
\draw[line width=1pt, color=white, fill] (0,0.6) node[inner sep=1.5pt,draw,fill=white, rounded corners=1pt] (R2) {{\tiny$\rho_i$}};
\draw[line width=1pt, color=blue!50!black] (0,0.6) node[inner sep=1.5pt,draw, rounded corners=1pt] (R) {{\tiny$\rho_i$}};
\draw[line width=1pt, color=white, fill] (-0.85,0) node[inner sep=1.5pt,draw,fill=white] (R2) {{\tiny$\mu$}};
\draw[line width=1pt, color=green!50!black] (-0.85,0) node[inner sep=1.5pt,draw, rounded corners=1pt] (R) {{\tiny$\mu$}};
\draw[color=green!50!black, semithick, double distance=1.8pt] (-0.5,-1.5) -- (-0.5,-1); 
\draw[color=green!50!black, semithick, double distance=1.8pt] (-1.2,-1.5) -- (-1.2,-1); 
\end{tikzpicture} 
=
\begin{tikzpicture}[line width=3pt,scale=0.75,color=blue!50!black, baseline]
\draw (0,-1.5) node[below] (X) {{\scriptsize$M$}};
\draw[color=green!50!black] (-0.5,-1.5) node[below] (X) {{\scriptsize$A_i$}};
\draw[color=green!50!black] (-1.2,-1.5) node[below] (X) {{\scriptsize$A_i$}};
\draw (0,-1.5) -- (0,1); 
\draw[color=green!50!black, semithick, double distance=1.8pt] (-1.2,-0.8) .. controls +(0,0.25) and +(0,0) .. (-0.85,-0.5);
\draw[color=green!50!black, semithick, double distance=1.8pt] (-0.5,-0.8) .. controls +(0,0.25) and +(0,0) .. (-0.85,-0.5);
\draw[color=green!50!black, semithick, double distance=1.8pt] (-0.5,-1.5) -- (-0.5,-0.8); 
\draw[color=green!50!black, semithick, double distance=1.8pt] (-1.2,-1.5) -- (-1.2,-0.8); 
\draw[color=green!50!black, semithick, double distance=1.8pt] (-0.85,-0.5) -- (-0.85,0); 
\fill[color=green!50!black] (-0.85,-0.5) circle (4pt) node (mult1) {};
\draw[color=green!50!black, line width=3pt] (-0.85,0) .. controls +(0,0.5) and +(0,0) .. (0,0.6);
\draw[line width=1pt, color=white, fill] (0,0.6) node[inner sep=1.5pt,draw,fill=white, rounded corners=1pt] (R2) {{\tiny$\rho_i$}};
\draw[line width=1pt, color=blue!50!black] (0,0.6) node[inner sep=1.5pt,draw, rounded corners=1pt] (R) {{\tiny$\rho_i$}};
\draw[line width=1pt, color=white, fill] (-0.85,-0.5) node[inner sep=1.5pt,draw,fill=white] (R2) {{\tiny$\mu$}};
\draw[line width=1pt, color=green!50!black] (-0.85,-0.5) node[inner sep=1.5pt,draw, rounded corners=1pt] (R) {{\tiny$\mu$}};
\end{tikzpicture} 
\ee
where we used that~$M$ is an $A_i^{\text{op}}$-module in the first step, and \eqref{eq:halfassed} in the last step. 
This shows triangulation invariance. 

\textsl{$*$-decoration:} 1-strata in the interior of $N$ are necessarily circles.
Choosing a different $*$-decoration around a 1-stratum which is labelled with a multi-module~$M$ results in a ribbon graph which is isotopic to the graph where we keep the $*$-decoration while changing the label to a twisted module $M^{\text{tw}_j}$, see Remark~\ref{rem:sym} and Definition~\ref{def:twmm}. 
Thanks to the cyclic symmetry in Definition~\ref{def:cyclic-with-signs}, 
the two modules are isomorphic and inserting a pair $\varphi_j \circ \varphi_j^{-1} = \id_M$ and dragging one of the two around the circle shows the claim.
\end{proof}

It remains to find coherent 
	linear maps between the vector spaces
associated to surfaces $\Sigma(\tau,*), \Sigma(\tilde\tau,\tilde *) \in \Borddc$ with different $*$-decorations $*$, $\tilde *$ and triangulations $\tau$, $\tau'$.
We will describe these by cylinders $C_\Sigma = \Sigma \times [0,1]$ with appropriate ribbon graphs. 

We start by describing how to treat line defects in $C_\Sigma$.
Let~$P$ be an ${}_{A_1 \dots A_n}M$-decorated 0-stratum on~$\Sigma$ with orientation~$+$. 
By part (iv) of Construction~\ref{con:ZRTC}, close to each boundary we need to insert an $M$-ribbon in the wedge labelled~$*$, oriented as in Figure~\ref{fig:relatestars} (or, equivalently, Figure~\ref{fig:Aiorientations}). 
In passing from the incoming to the outgoing boundary of $C_\Sigma$, we first rotate the $M$-ribbon starting at the incoming boundary counter-clockwise until it lies in the plane of the outgoing ribbon. Afterwards we insert the appropriate isomorphism from the equivariant structure of $M$: If the minimal cyclic generator is $k$ and we passed $jk$ surface defects in going from $*$ to $\tilde *$, we insert $\varphi^j$. 
Figure~\ref{fig:relatestars} gives an example.

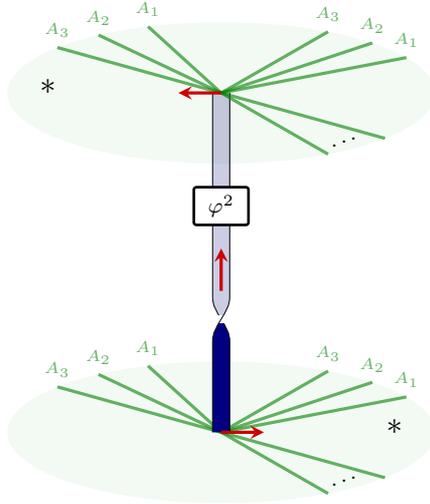
\begin{figure}[t]
\begin{align*}
\begin{tikzpicture}[very thick,scale=0.75,color=blue!50!black, baseline]
%
\newcommand{\yco}{-2}
\fill[color=green!50!black, opacity=0.05] (0,\yco) ellipse (3.75 and 1.25);
\draw[color=black] ($(5:3.05 and 1.25) + (0,\yco)$) node[line width=0pt] (beta) {{$*$}};
\draw[color=green!50!black, opacity=0.6] ($(30:3.75 and 1.25) + (0,\yco)$) -- (0,\yco);
\draw[color=green!50!black, opacity=0.6] ($(30:3.75 and 1.25) + (0,\yco)$) node[line width=0pt, above] (beta) {{\tiny$A_1$}};
\draw[color=green!50!black, opacity=0.6] ($(45:3.75 and 1.25) + (0,\yco)$) -- (0,\yco);
\draw[color=green!50!black, opacity=0.6] ($(45:3.75 and 1.25) + (0,\yco)$) node[line width=0pt, above] (beta) {{\tiny$A_2$}};
\draw[color=green!50!black, opacity=0.6] ($(60:3.75 and 1.25) + (0,\yco)$) -- (0,\yco);
\draw[color=green!50!black, opacity=0.6] ($(60:3.75 and 1.25) + (0,\yco)$) node[line width=0pt, above] (beta) {{\tiny$A_3$}};
\draw[color=green!50!black, opacity=0.6] ($(110:3.75 and 1.25) + (0,\yco)$) -- (0,\yco);
\draw[color=green!50!black, opacity=0.6] ($(110:3.75 and 1.25) + (0,\yco)$) node[line width=0pt, above] (beta) {{\tiny$A_1$}};
\draw[color=green!50!black, opacity=0.6] ($(125:3.75 and 1.25) + (0,\yco)$) -- (0,\yco);
\draw[color=green!50!black, opacity=0.6] ($(125:3.75 and 1.25) + (0,\yco)$) node[line width=0pt, above] (beta) {{\tiny$A_2$}};
\draw[color=green!50!black, opacity=0.6] ($(140:3.75 and 1.25) + (0,\yco)$) -- (0,\yco);
\draw[color=green!50!black, opacity=0.6] ($(140:3.75 and 1.25) + (0,\yco)$) node[line width=0pt, above] (beta) {{\tiny$A_3$}};
\draw[color=green!50!black, opacity=0.6] ($(300:3.75 and 1.25) + (0,\yco)$) -- (0,\yco);
\draw[color=black] ($(270:3.75 and 1.05) + (2.15,\yco+0.2)$) node[line width=0pt] (beta) {{{\scriptsize$\,\rotatebox{20}{\dots}$}}};
\draw[color=green!50!black, opacity=0.6] ($(320:3.75 and 1.25) + (0,\yco)$) -- (0,\yco);
\newcommand{\brei}{0.15}
%
\fill[opacity = 0.2] (0,0)
	to[out=60, in=270] (\brei,0.4)
	-- (\brei,2)
	-- (-\brei,2)
	-- (-\brei,0.4)
	to[out=270, in=120] (0,0);
\fill (\brei,-2)
	-- (\brei,-0.4)
	to[out=90, in=-60] (0,0)
	to[out=240, in=90] (-\brei,-0.4)
	-- (-\brei,-2);
\draw[line width=0.3pt, color=black] (\brei,-2) 		
							-- (\brei,-0.4)
							to[out=90, in=-60] (0,0)
							to[out=120, in=270] (-\brei,0.4)
							-- (-\brei,2);
\node[inner sep=1.5pt, fill, white] at (0,0) {};
\draw[line width=0.3pt, color=black] (-\brei,-2) 		
							-- (-\brei,-0.4)
							to[out=90, in=240] (0,0)
							to[out=60, in=270] (\brei,0.4)
							-- (\brei,2);
%
%
\fill[opacity = 0.2] (-\brei,2) -- (\brei,2) -- (\brei,4) -- (-\brei,4); 
\draw[line width=0.3pt, color=black] (-\brei,2) -- (-\brei,4);
\draw[line width=0.3pt, color=black] (\brei,2) -- (\brei,4);
%
\fill[color=white] (0,2) node[inner sep=3pt, fill, rounded corners=1pt, color=white] (P) {{\scriptsize$\;\varphi^{2}\;$}};; 
\draw[line width=1pt, color=black] (0,2) node[inner sep=3pt,draw, rounded corners=1pt] (P) {{\scriptsize$\;\varphi^{2}\;$}};
%
%
%
%
\renewcommand{\yco}{4}
\fill[color=green!50!black, opacity=0.05] (0,\yco) ellipse (3.75 and 1.25);
\draw[color=black] ($(175:3.05 and 1.25) + (0,\yco)$) node[line width=0pt] (beta) {{$*$}};
\draw[color=green!50!black, opacity=0.6] ($(30:3.75 and 1.25) + (0,\yco)$) -- (0,\yco);
\draw[color=green!50!black, opacity=0.6] ($(30:3.75 and 1.25) + (0,\yco)$) node[line width=0pt, above] (beta) {{\tiny$A_1$}};
\draw[color=green!50!black, opacity=0.6] ($(45:3.75 and 1.25) + (0,\yco)$) -- (0,\yco);
\draw[color=green!50!black, opacity=0.6] ($(45:3.75 and 1.25) + (0,\yco)$) node[line width=0pt, above] (beta) {{\tiny$A_2$}};
\draw[color=green!50!black, opacity=0.6] ($(60:3.75 and 1.25) + (0,\yco)$) -- (0,\yco);
\draw[color=green!50!black, opacity=0.6] ($(60:3.75 and 1.25) + (0,\yco)$) node[line width=0pt, above] (beta) {{\tiny$A_3$}};
\draw[color=green!50!black, opacity=0.6] ($(110:3.75 and 1.25) + (0,\yco)$) -- (0,\yco);
\draw[color=green!50!black, opacity=0.6] ($(110:3.75 and 1.25) + (0,\yco)$) node[line width=0pt, above] (beta) {{\tiny$A_1$}};
\draw[color=green!50!black, opacity=0.6] ($(125:3.75 and 1.25) + (0,\yco)$) -- (0,\yco);
\draw[color=green!50!black, opacity=0.6] ($(125:3.75 and 1.25) + (0,\yco)$) node[line width=0pt, above] (beta) {{\tiny$A_2$}};
\draw[color=green!50!black, opacity=0.6] ($(140:3.75 and 1.25) + (0,\yco)$) -- (0,\yco);
\draw[color=green!50!black, opacity=0.6] ($(140:3.75 and 1.25) + (0,\yco)$) node[line width=0pt, above] (beta) {{\tiny$A_3$}};
\draw[color=green!50!black, opacity=0.6] ($(300:3.75 and 1.25) + (0,\yco)$) -- (0,\yco);
\draw[color=black] ($(270:3.75 and 1.05) + (2.15,\yco+0.2)$) node[line width=0pt] (beta) {{{\scriptsize$\,\rotatebox{20}{\dots}$}}};
\draw[color=green!50!black, opacity=0.6] ($(320:3.75 and 1.25) + (0,\yco)$) -- (0,\yco);
%
%
\draw[color=black!20!red, very thick, ->, >=stealth] (0,\yco) -- (-0.75,\yco); 
\draw[color=black!20!red, very thick, ->, >=stealth] (0,-2) -- (0.75,-2); 
\draw[color=black!20!red, very thick, ->, >=stealth] (0,0.5) -- (0,1.25); 
\end{tikzpicture} 
\end{align*}
\caption{Relating different $*$-decorations of the same surface; example with minimal cyclic generator $k=3$, and $j=2$ copies of $A_1,A_2,A_3$ between the two $*$-decorations.} 
\label{fig:relatestars} 
\end{figure}

The surface defects are treated as in Construction~\ref{con:ZRTC} with two extra conditions: The triangulation~$t$ in the interior of $C_\Sigma$ 
has to restrict to~$\tau$ and~$\tilde\tau$ on the incoming and outgoing boundary, respectively, and $A$-ribbons from surface defects may be connected
to $M$-ribbons from line defects only before we started rotating $M$ (using the position of $*$ as coming from the incoming boundary) or after the insertion of $\varphi^j$ (using the position of $\tilde *$ as coming from the outgoing boundary).

We denote the bordism $C_\Sigma$ with this embedded ribbon graph as $C_{\Sigma,\tau,\tilde\tau,*,\tilde *}$. 
A straightforward but tedious calculation shows that the combination of $M$-ribbon-rotation and $\varphi^j$ commutes with attaching $A$-ribbons of a given surface defect to~$M$. This observation, together with triangulation independence as in the proof of Lemma~\ref{lem:Z-indep-of-interior-choices}, shows that the linear map (which is in general not an isomorphism)
\be
\label{eq:bigPsi}
\Psi_{\tau,\tilde\tau,*,\tilde *}^\Sigma 
:= 
\zrt \big( C_{\Sigma,\tau,\tilde\tau,*,\tilde *} \big)
\colon 
\; 
\zrt\big( \Sigma(\tau,*) \big) 
	\lra
\zrt\big( \Sigma(\tau,\tilde *) \big) 
\ee
is independent of all interior choices.

\begin{lemma}\label{lem:directed-system}
The maps $\Psi_{\tau, \widetilde\tau,*,\tilde *}^\Sigma$ in \eqref{eq:bigPsi} are compatible in the sense that
\be
\Psi_{\tau_1, \tau_3,*_1,*_3}^\Sigma = \Psi_{\tau_2, \tau_3,*_2,*_3}^\Sigma \circ \Psi_{\tau_1, \tau_2,*_1,*_2}^\Sigma
\ee 
for all triangulations~$\tau_i$ and $*$-decorations~$*_i$ for $i\in \{ 1,2,3 \}$. 
\end{lemma}

\begin{proof}
The composition $\Psi_{\tau_2, \tau_3,*_2,*_3}^\Sigma \circ \Psi_{\tau_1, \tau_2,*_1,*_2}^\Sigma$ can be computed by gluing the corresponding cylinders:  $\zrt ( C_{\Sigma,\tau_2, \tau_3,*_2,*_3} \circ C_{\Sigma,\tau_1, \tau_2,*_1,*_2})$. For each 1-stratum, labelled by $M$ say, one now proceeds as follows: using the properties of $\Delta$-separable symmetric Frobenius algebras and their modules, remove all $A$-ribbons attached from adjacent surfaces between the start of the first rotation and the insertion of the second $\varphi$. 
We are left with the overall rotation from $*_1$ to $*_3$ and the isomorphism $\varphi^{j_{23}+j_{12}}$. If the overall rotation does not exceed $2\pi$, this is precisely the ribbon graph for $C_{\Sigma,\tau_1, \tau_3,*_1,*_3}$. 
If it does exceed a rotation by $2\pi$, then we employ the condition $\varphi^{n/k} = \theta_M^{-1}$ in Definition~\ref{def:symmetricmodule} to get $\varphi^{j_{23}+j_{12}} = \varphi^{j_{23}+j_{12}-n/k} \circ \theta_M^{-1}$. 
The extra inverse twist compensates the over-rotation of the ribbon, resulting again in the ribbon graph $C_{\Sigma,\tau_1, \tau_3,*_1,*_3}$.
\end{proof}

We are now in a position to apply a standard limit construction (see e.\,g.~\cite{LawrenceIntro} or Constructions~3.7--9 of \cite{CRS1}) 
to define our defect TQFT~$\zzc$. 
By definition, for any $\Sigma \in \Borddefen{3}(\DC)$ we set $\zzc(\Sigma)$ to be the vector space which is the limit of the inverse system~\eqref{eq:bigPsi}, and the action of~$\zzc$ on morphisms is defined as in \cite[Constr.\,3.9(v)]{CRS1}. 

More explicitly, we have that up to isomorphism, $\zzc(\Sigma)$ is given by the image of an idempotent, namely the map $\Psi_{\tau,\tau,*,*}^\Sigma$ for any choice of $*$-decoration~$*$ and any triangulation~$\tau$ of the 1-strata of~$\Sigma$: 
\be
\label{eq:ImageOfPis}
\zzc(\Sigma) \cong \text{Im} \big( \Psi_{\tau,\tau,*,*}^\Sigma \big) \, . 
\ee
Up to these isomorphisms, the linear map which~$\zzc$ assigns to a bordism $N\colon \Sigma_1 \to \Sigma_2$ in $\Borddc$ 
is given by the linear map that the Reshetikhin-Turaev construction~\eqref{eq:ZRTCstar} induces on the vector spaces $\text{Im}(\Psi_{\tau_i,\tau_i,*_i,*_i}^{\Sigma_i})$ for $i\in\{1,2\}$. 
Thus in summary, we have: 

\begin{theorem}
\label{thm:mainresult}
The limit of the inverse system~\eqref{eq:bigPsi} gives a defect TQFT $\zzc \colon \Borddefen{3}(\DC) \to \Vectk$. 
\end{theorem}

\begin{remark}
\begin{enumerate}
\item
The functor $\zzc$ is indeed an extension of the defect TQFT considered in \cite{CMS}: The monoidal unit $\one \in \Cat{C}$ is a $\Delta$-separable Frobenius algebra, and every object $X \in \Cat{C}$ is canonically a cyclic multi-module in the sense that $((\one,-), \ldots, (\one,-), (\one,+), X) \in D^{\Cat{C}}_1$ for any number $n\in \Z_+$ of copies of~$\one$. 
Restricting to such defect labels yields a reformulation of the defect data considered in \cite{CMS}, and Construction~\ref{con:ZRTC} restricts to the corresponding defect TQFT. 
\item 
Let us briefly describe how the Reshetikhin-Turaev construction $\zrt$ recalled in Section~\ref{sec:MTCRT} can be recovered from~$\zzc$. 
Note first though that there are two main differences between the category $\Bordriben{3}(\Cat{C})$ and the subcategory of $\Borddefen{3}(\DC)$ without 2-strata: in the former line defects (1-strata) are framed, while in $\Borddefen{3}(\DC)$ they are not, and in $\Bordriben{3}(\Cat{C})$ the 1-strata can meet in coupons (which one could model as 0-strata), which we excluded in $\Borddefen{3}(\DC)$. 
Both ribbons and coupons can however be modelled in the defect bordism category: 

\textsl{\underline{Ribbons}}: 
In $\Borddefen{3}(\DC)$ we can model a ribbon labelled by $X\in\Cat{C}$ as a $\one$-labelled ribbon (viewed as a 2-stratum) whose left and right boundaries are 1-strata labelled by~$X$ and~$\one$ (both viewed as $\one$-modules), respectively: 
\be
\label{eq:ribbons}
\tikzzbox{\begin{tikzpicture}[very thick,scale=1.2,color=blue!50!black, 	baseline=1cm]
	%
	\draw[
	color=blue!50!black, 
	line width=3pt, 
	>=stealth,
	decoration={markings, mark=at position 0.5 with {\arrow{>}},
	}, postaction={decorate}
	] 
	(0,0) -- (0,2);
	%
	\draw[line width=1, color=blue!50!black] (-0.2, 1) node[line width=0pt] (beta) {{\scriptsize$X$}};
	\end{tikzpicture}}
\;\;\;
\lmt
\;  
\tikzzbox{\begin{tikzpicture}[very thick,scale=1.2,color=blue!50!black, baseline=1cm]
	%
	\fill [green!50!black,opacity=0.15] (0,0) -- (0,2) -- (0.3,2) -- (0.3,0) -- (0,0);
	%
	\draw[
	color=blue!50!black, 
	>=stealth,
	decoration={markings, mark=at position 0.5 with {\arrow{>}},
	}, postaction={decorate}
	] 
	(0,0) -- (0,2);
	\draw[
	color=green!50!black, 
	thick, 
	>=stealth,
	decoration={markings, mark=at position 0.5 with {\arrow{>}},
	}, postaction={decorate}
	] 
	(0.3,0) -- (0.3,2);
	%
	\draw[line width=1, color=blue!50!black] (-0.2, 1) node[line width=0pt] (beta) {{\scriptsize$X$}};
	\draw[line width=1, color=green!50!black] (0.5, 1) node[line width=0pt] (beta) {{\scriptsize$\one$}};	
	\draw[line width=1, color=green!50!black] (0.175, 1.5) node[line width=0pt] (beta) {{\scriptsize$\one$}};	
	\end{tikzpicture}}
\ee
\textsl{\underline{Coupons}}: 
As mentioned in Remark~\ref{rem:intro-whatelse}(iv) and explained in detail in \cite[Sect.\,2.4]{CRS1}, the set of defect data~$\DC$ can canonically be completed to a set of defect data $(\DC)^\bullet$ to include all compatible labels for 0-strata, and $\zzc$ lifts to its ``$D_0$-completion'' 
\be 
(\zzc)^\bullet \colon \Borddefen{3}((\DC)^\bullet) \lra \Vectk \, . 
\ee 
We claim that any $\phi$-labelled coupon in $\Bordriben{3}(\Cat{C})$ with ingoing ribbons~$X_i$ and outgoing ribbons~$Y_j$ can be modelled as a neighbourhood of a corresponding $\phi$-labelled 0-stratum in $\Borddefen{3}((\DC)^\bullet)$: 
the $X_i$- and $Y_j$-labelled 1-strata corresponding to the ribbons meet in the 0-stratum, and the associated $\one$-labelled 1- and 2-strata form a half-disc in the plane of the coupon: 
\be
\label{eq:coupons}
\begin{tikzpicture}[line width=3pt, scale=0.75, color=blue!50!black, baseline]
%
\draw[color=blue!40!green, >=stealth, decoration={markings, mark=at position 0.8 with {\arrow{>}},}, postaction={decorate}] 
(-0.25,0) -- (-0.25,0.5) to[out=90, in=270] (-1,2) -- (-1,3); 
\draw[color=blue!50!black, >=stealth, decoration={markings, mark=at position 0.78 with {\arrow{>}},}, postaction={decorate}] 
(0,0) -- (0,0.5) to[out=90, in=270] (-0.25,2) -- (-0.25,3); 
\draw[color=red!50!green, >=stealth, decoration={markings, mark=at position 0.8 with {\arrow{>}},}, postaction={decorate}] 
(0.25,0) -- (0.25,0.5) to[out=90, in=270] (1,2) -- (1,3); 
%
%
\draw[color=blue!50!red, >=stealth, decoration={markings, mark=at position 0.8 with {\arrow{<}},}, postaction={decorate}] 
(-0.25,0) -- (-0.25,-0.5) to[out=270, in=90] (-1,-2) -- (-1,-3); 
\draw[color=blue!30!red, >=stealth, decoration={markings, mark=at position 0.78 with {\arrow{<}},}, postaction={decorate}] 
(0,0) -- (0,-0.5) to[out=270, in=90] (-0.25,-2) -- (-0.25,-3); 
\draw[color=red!80!green, >=stealth, decoration={markings, mark=at position 0.8 with {\arrow{<}},}, postaction={decorate}] 
(0.25,0) -- (0.25,-0.5) to[out=270, in=90] (1,-2) -- (1,-3); 
%
%
\fill[color=white] (0,0) node[inner sep=4pt,draw, rounded corners=1pt, fill, color=white] (R2) {{\scriptsize$\rho_2$}};
\draw[line width=1pt, color=black] (0,0) node[inner sep=4pt,draw, rounded corners=1pt] (R) {{\scriptsize$\;\phi\;$}};
%
%
\fill[color=black] (-1.3,1.85) node[above] (X) {{\scriptsize$Y_1$}};
\fill[color=black] (-0.55,1.85) node[above] (X) {{\scriptsize$Y_2$}};
\fill[color=black] (0.4,1.85) node[above] (X) {{\scriptsize$\dots$}};
\fill[color=black] (1.4,1.85) node[above] (X) {{\scriptsize$Y_m$}};
\fill[color=black] (-1.3,-2.9) node[above] (X) {{\scriptsize$X_1$}};
\fill[color=black] (-0.55,-2.9) node[above] (X) {{\scriptsize$X_2$}};
\fill[color=black] (0.4,-2.9) node[above] (X) {{\scriptsize$\dots$}};
\fill[color=black] (1.45,-2.9) node[above] (X) {{\scriptsize$X_n$}};
\end{tikzpicture}
\lmt \;\;\;
\begin{tikzpicture}[very thick, scale=0.75, color=blue!50!black, baseline]
%
\fill [green!50!black,opacity=0.15] 
(0,-3) -- (0.3,-3) -- ([shift=(-79:1.5)]0,0) arc (-79:-71.5:1.5)
-- (1,-3) -- (1.3,-3) -- ([shift=(-63:1.5)]0,0) arc (-63:-45:1.5)
-- (3,-3) -- (3.3,-3) -- ([shift=(-36.5:1.5)]0,0) arc (-36.5:36.5:1.5) -- (3.3,3)
-- (3,3) -- ([shift=(45:1.5)]0,0) arc (45:63:1.5) -- (1.3,3)
-- (1,3) -- ([shift=(71.5:1.5)]0,0) arc (71.5:79:1.5) -- (0.3,3)
-- (0,3);
%
\draw[color=blue!40!green, >=stealth, decoration={markings, mark=at position 0.8 with {\arrow{>}},}, postaction={decorate}]
(0,0) -- (0,3);
\draw[color=blue!50!black, >=stealth, decoration={markings, mark=at position 0.8 with {\arrow{>}},}, postaction={decorate}]
(0,0) -- (1,3);
\draw[color=red!50!green, >=stealth, decoration={markings, mark=at position 0.8 with {\arrow{>}},}, postaction={decorate}]
(0,0) -- (3,3);
\draw[color=blue!50!red, >=stealth, decoration={markings, mark=at position 0.8 with {\arrow{<}},}, postaction={decorate}]
(0,0) -- (0,-3);
\draw[color=blue!30!red, >=stealth, decoration={markings, mark=at position 0.8 with {\arrow{<}},}, postaction={decorate}]
(0,0) -- (1,-3);
\draw[color=red!80!green, >=stealth, decoration={markings, mark=at position 0.8 with {\arrow{<}},}, postaction={decorate}]
(0,0) -- (3,-3);
\draw[color=green!50!black, thick, >=stealth, decoration={markings, mark=at position 0.5 with {\arrow{>}},}, postaction={decorate}] 
([shift=(71.5:1.5)]0,0) arc (71.5:79:1.5) -- (0.3,3);
\draw[color=green!50!black, thick, >=stealth, decoration={markings, mark=at position 0.5 with {\arrow{>}},}, postaction={decorate}] 
([shift=(45:1.5)]0,0) arc (45:63:1.5) -- (1.3,3);
\draw[color=green!50!black, thick, >=stealth, decoration={markings, mark=at position 0.5 with {\arrow{>}},}, postaction={decorate}] 
(3.3,-3) -- ([shift=(-36.5:1.5)]0,0) arc (-36.5:36.5:1.5) -- (3.3,3);
\draw[color=green!50!black, thick, >=stealth, decoration={markings, mark=at position 0.5 with {\arrow{>}},}, postaction={decorate}] 
(1.3,-3) -- ([shift=(-63:1.5)]0,0) arc (-63:-45:1.5);
\draw[color=green!50!black, thick, >=stealth, decoration={markings, mark=at position 0.5 with {\arrow{>}},}, postaction={decorate}] 
(0.3,-3) -- ([shift=(-79:1.5)]0,0) arc (-79:-71.5:1.5);
%
%
\fill[color=black] (0,0) circle (4pt) node[right] (meet2) {};
\fill[color=green!50!black] (71.5:1.5) circle (2.5pt) node[right] (meet2) {};
\fill[color=green!50!black] (45:1.5) circle (2.5pt) node[right] (meet2) {};
\fill[color=green!50!black] (-45:1.5) circle (2.5pt) node[right] (meet2) {};
\fill[color=green!50!black] (-71.5:1.5) circle (2.5pt) node[right] (meet2) {};
%
%
\fill[color=black] (-0.22,2.15) node[above] (X) {{\scriptsize$Y_1$}};
\fill[color=black] (0.55,2.15) node[above] (X) {{\scriptsize$Y_2$}};
\fill[color=black] (2,2.15) node[above] (X) {{\scriptsize$Y_m$}};
\fill[color=black] (-0.25,-2.15) node[below] (X) {{\scriptsize$X_1$}};
\fill[color=black] (0.55,-2.15) node[below] (X) {{\scriptsize$X_2$}};
\fill[color=black] (2.1,-2.15) node[below] (X) {{\scriptsize$X_n$}};
\fill[color=green!50!black] (0.2,1.8) node (X) {{\tiny$\one$}};
\fill[color=green!50!black] (0.74,1.8) node (X) {{\tiny$\one$}};
\fill[color=green!50!black] (1.35,0) node (X) {{\tiny$\one$}};
\fill[color=green!50!black] (0.2,-2.1) node (X) {{\tiny$\one$}};
\fill[color=green!50!black] (0.82,-2.04) node (X) {{\tiny$\one$}};
\fill[color=black] (-0.35,0) node (X) {{\scriptsize$\phi$}};
\end{tikzpicture} 
\ee 
By the construction of \cite[Sect.\,2.4]{CRS1}, the 0-strata on the right-hand side must be labelled by elements of the vector spaces computed from the action of~$\zzc$ on small stratified spheres around the 0-strata. 
By Lemma~\ref{lem:value-sMN} below these vector spaces are precisely the corresponding Hom spaces in~$\Cat{C}$.\footnote{The general construction of~$\mathds{D}^\bullet$ in \cite{CRS1} involves a symmetry constraint, but since the spheres around the 0-strata in~\eqref{eq:coupons} do not have such symmetries (as there are no $\one$-labelled 1-strata to the left of the $X_1$- and $Y_1$-labelled lines, for example), the label sets are indeed the Hom spaces of Lemma~\ref{lem:value-sMN}.} 
Hence the 0-stratum on which the $X_i$- and $Y_j$-labelled 1-strata are incident can indeed by labelled by $\phi \in \Hom_{\Cat{C}}(\bigotimes_i X_i, \bigotimes_j Y_j)$, and the remaining 0-strata in~\eqref{eq:coupons} are (invisibly) labelled by the corresponding unitors in~$\Cat{C}$. 

The assignments~\eqref{eq:ribbons} and~\eqref{eq:coupons} lift to a faithful functor $F\colon \Bordriben{3}(\Cat{C}) \to \Borddefen{3}((\DC)^\bullet)$ such that $\zrt \cong (\zzc)^\bullet \circ F$. 
\item 
By Lemma~\ref{lem:Z-indep-of-interior-choices}, our construction is independent of the choice of $*$-decoration for 1-strata which are entirely contained in the interior of bordisms. 
This also means that we have independence of the choice of $C_n$-equivariant structure on the modules which decorate such 1-strata in the interior. 

Furthermore, up to isomorphism also the value of~$\zzc$ on objects is independent of the choice of $C_n$-equivariant structures~$\varphi$ in the labels $(M,\varphi)$ for 1-strata. This follows directly from the fact that the isomorphisms~$\varphi$ do not appear in the idempotents $\Psi_{\tau,\tau,*,*}^\Sigma$ in~\eqref{eq:ImageOfPis}. 

Nonetheless, in general the action of~$\zzc$ on morphisms does depend on the choice of equivariant structure on modules. 
For example, consider the object $\Sigma \in \Borddefen{3}(\DC)$ which is a 2-sphere with two 0-strata at the poles decorated by $C_n$-equivariant $A^{\otimes n}$-modules $(M,\varphi)$ and $(M',\varphi')$, 
and $n\geqslant 2$ 1-strata, all decorated with the same algebra~$A$ and running along the longitudes $i\cdot 360\degree/n$ from the south pole to the north pole. Now let~$N$ be the cylinder over~$\Sigma$, where the ingoing boundary is parametrised by the identity map while the outgoing boundary by a non-trivial rotation of~$\Sigma$. This means that the $M$- and $M'$-ribbon connecting the poles are twisted by an angle between~0 and $2\pi$, and non-trivial powers of the structure maps $\varphi$, $\varphi'$  
have to be inserted as in Figure~\ref{fig:relatestars}. Hence $\zzc(N)$ depends on the choice of these maps. 
\item 
In the description of 2-dimensional conformal field theory on unoriented surfaces, so-called ``Jandl algebras'' were introduced in \cite{tft2}. 
These are algebras in $\Cat{C}$ together with a \textsl{reversion}, i.\,e.\ an algebra isomorphism $A \to A^{\mathrm{op}}$ which squares to the twist on $A$. 
In the present context, a reversion allows one to change the 2-orientation of a surface defect without affecting the value of $\zzc$; the beginning of \cite[Sect.\,2]{tft2} illustrates the mechanism. This could be used to develop a theory of unoriented surface defects.
\end{enumerate}
\label{rem:relationToRT}
\end{remark}

\medskip

\noindent 
\textbf{\textsl{State spaces associated to decorated spheres in~$\boldsymbol{\zzc}$}: }
As illustration we apply the construction    of $\zzc$ to compute the state spaces associated to spheres with decorations by multi-modules.
Let ${}_{A_1\dots A_n}M$ and ${}_{A_1\dots A_n}N$ be cyclic multi-modules over 
$\Delta$-separable symmetric Frobenius algebras $A_{1}, \ldots A_{n}$. 
Then there is an associated decorated sphere denoted $S_{M,N}$ with two 0-strata on the north and south pole decorated by $(N,-)$, $(M,+)$, respectively, and $n$ $1$-strata connecting the two $0$-strata and decorated $A_{1}, \ldots, A_{n}$. 

Recall from Definition \ref{definition:mapmulti} that $\Hom_{A_{1} \otimes \ldots \otimes A_{n}}(M,N)$ 
is the vector space of maps of multi-modules from $M$ to $N$ (which are not required to respect the cyclic structures of~$M$ and~$N$).
\begin{lemma}
\label{lem:value-sMN}

We have 
	\begin{equation}
	\label{eq:state-space-sphere}
	\zzc(S_{M,N}) \cong \Hom_{A_{1} \otimes \ldots \otimes A_{n}}(M,N) \, .
	\end{equation}
\end{lemma}
\begin{proof}
First we pick  a $\ast$-decoration $\ast$ and  a triangulation $\tau$ of the 1-strata on $S_{M,N}$  and  extend the triangulation  to the  bordism $C_{S_{M,N},\tau,\tau,\ast,\ast}$ corresponding to the cylinder $C_{S_{M,N}}=S_{M,N} \times [0,1]$. 
Then the map  
	\begin{equation}
	P := \zrt(C_{S_{M,N},\tau,\tau,\ast,\ast}) \colon \zrt(S_{M,N}(\tau,\ast)) \lra  \zrt(S_{M,N}(\tau,\ast))
	\end{equation}
	is an idempotent, and by \eqref{eq:ImageOfPis} we have 
	\begin{equation}
	\label{eq:def-state-space}
	\zzc(S_{M,N}) \cong  \text{Im}(P) \, .
	\end{equation}
To compute the linear map~$P$ we evaluate a string diagram in $\Cat{C}$ corresponding to the dual of the triangulation of $C_{S_{M,N}}$. 
One finds that the diagram is 
a horizontal product of diagrams corresponding to the 2-strata, so that we are dealing with a product of commuting projectors. 
Hence we can and will assume without loss of generality that $n=1$, i.\,e.\ there is a single 2-stratum labelled~$A$ while ${}_{A}M$ and ${}_{A}N$ are $A$-modules, and the $\ast$-decoration is unique. 
Assuming that the triangulation~$\tau$ of $S_{M,N}$ has $m$ 1-simplices, we have that by definition
	\begin{equation}
	\zrt(S_{M,N},\tau)= \Hom_{\Cat{C}}(A^{\otimes m} \otimes M, A^{\otimes m} \otimes N) \, .
	\end{equation}

For $f \in \Hom_{\Cat{C}}(A^{\otimes m} \otimes M, A^{\otimes m} \otimes N)$ the map  $P(f)$ is by definition the evaluation of a string diagram $\Gamma(f)$ in $\Cat{C}$ that corresponds to the dual of the triangulation of $C_{S_{M,N}}$ with one coupon labelled by~$f$. 
By the properties of the triangulation, the diagram~$\Gamma'$ obtained from $\Gamma(f)$ by removing the coupon is connected (though~$\Gamma'$ is not a string diagram in~$\Cat{C}$). 
It follows from the  properties of $\Delta$-separable symmetric Frobenius algebras that this property determines the evaluation of $\Gamma(f)$ uniquely:
\begin{itemize}
\item[($\#$)]
Let~$P'$ be an endomorphism of $\Hom_{\Cat{C}}(A^{\otimes m} \otimes M, A^{\otimes m} \otimes N)$ that is defined on a morphism~$f$ by the evaluation of a string diagram which is built only from one $f$-labelled coupon, the structure maps of the Frobenius algebra~$A$ and the actions of~$A$ on~$M$ and~$N$, such that the diagram remains connected after removing the coupon. 
Then $P' = P$. 
\end{itemize}
Now we can define isomorphisms 
	\begin{equation}
	\rho \colon \text{Im}(P) \rightleftarrows \Hom_{A}(M,N) :\! \varphi \, ,
	\end{equation}
by $\varphi(g)= P(1_{A^{\otimes m}}\otimes g)$ and $\rho(f)=(\varepsilon^{\otimes m} \otimes 1_{N}) \circ f \circ (\eta^{\otimes m} \otimes 1_{M})$.
Using ($\#$) and the fact that~$g$ and $\rho(f)$ are module maps it is straightforward to see that $\rho, \varphi$ are well-defined and mutually inverse. 
\end{proof}

\medskip

\noindent
\textbf{\textsl{Topological invariants}: }
Our enhancement of $\zrt$ to a defect TQFT provides new invariants of topological objects. 
In particular, given a pair of a closed 3-manifold $N$ and an embedding $\iota \colon \Sigma \rightarrow N$ of an oriented surface $\Sigma$ in $N$, every $\Delta$-separable symmetric Frobenius algebra $A$ in $\Cat{C}$ yields a number 
 \begin{equation}
   \zzc \big((N,\iota);A\big) \, ,
 \end{equation}
where, as the notation suggests, we label the surface $\Sigma$ in $N$ with~$A$ and then evaluate using $\zzc$. 
By definition of~$\zzc$, the result is invariant under isotopies of~$\iota$. 
Non-isotopic embeddings may indeed lead to different invariants, but one needs surfaces of non-zero genus:

\medskip
\noindent
\textbf{\textsl{Embedded spheres all look the same to~$\boldsymbol{\zzc}$}: }
If $\Sigma$ a 2-sphere, replacing $\Sigma$ in $N$ by a ribbon network, one can isotopically change the network by sliding it along the surface of the sphere to a small neighbourhood of one point and evaluate it to $\dim(A)$ (using that~$A$ is a $\Delta$-separable symmetric Frobenius algebra) 
without changing $\zzc((N,\iota);A)$. 
Hence for all embeddings~$\iota$,
\be
	\zzc \big((N,\iota);A \big) = \dim(A) \cdot \zrt(N) \, .
\ee

\medskip
\noindent
\textbf{\textsl{Embedded tori can be distinguished by~$\boldsymbol{\zzc}$}: }
If $\Sigma$ is a 2-torus there are situations where non-isotopic embeddings can be distinguished. 
For example, let $N = S^1 \times S^1 \times S^1$ be the 3-torus. Write $\Sigma = S^1 \times S^1$ and consider the embeddings:
\begin{itemize}
\item[$\iota_0^\pm$: ] Both $S^1$ of $\Sigma$ are contractible in $N$, and $\pm$ refers to the two possible orientations of the embedded surface. 
\item[$\iota_1^\pm$: ] One $S^1$ of $\Sigma$ is contractible, the other one winds along an $S^1$ of $N$. Again,~$\pm$ refers to the two orientations.
\item[$\iota_2$: ] Both $S^1$ of $\Sigma$ wind along separate $S^1$ in $N$ and are non-contractible. In this case, the two orientations are related by a diffeomorphism of $N$.
\end{itemize}
We will need the left and a right centre $C_{l/r}(A)$ of $A$, which are in particular commutative subalgebras of~$A$, see \cite[Sect.\,2.4]{Frohlich:2003hm} for details. 
By expressing $\zrt(N)$ as a trace of $\zrt(S^1 \times S^1 \times [0,1])$ with appropriate ribbon graphs, after a short calculation one arrives at
\begin{align}
	\zzc\big((N,\iota_0^\pm);A\big) &= \dim\! \big(C_{l/r}(A\big)) \cdot |I| \, ,
	\nonumber \\
	\zzc\big((N,\iota_1^\pm);A\big) &= \sum_{i \in I} \dim_\Bbbk \Hom_{\Cat{C}} \! \big(i \otimes C_{l/r}(A),i\big)  \, ,
\end{align}
where~$I$ is a set of representatives of simple objects as in Remark~\ref{remark:rel-rib-def}. 
One always has $\dim C_l(A) = \dim C_r(A)$ (combine \cite[Thm.\,4.5]{Kirillov:2001ti} with \cite[Thm.\,5.20]{Frohlich:2003hm}), so that $\zzc((N,\iota_0^\pm);A)$ does not distinguish orientations.\footnote{For $\zzc((N,\iota_1^\pm);A)$ we do not know an example which can distinguish the two orientations~$\pm$, but we also are not aware of a general argument excluding such examples.}

For $\iota_2$ we need another ingredient. To $A$ one can associate an $|I| \times |I|$-matrix $\tilde Z(A)_{ij}$ with non-negative integral entries as in \cite[Eqn.\,(5.44)]{tft1}. This matrix is related to modular invariant partition functions of rational conformal field theory, we refer to \cite{tft1} for more details and references. 
(It is also related to a third notion of a centre for an algebra $A$, the \textsl{full centre} \cite{Fjelstad:2006aw,Davydov:2009}.)
The computation of $\zzc$ in this case reduces to taking the trace of $\zrt$ evaluated on the 3-manifold and ribbon graph shown in \cite[Eqn.\,(5.24)]{tft1}, with the result
\be
	\zzc\big((N,\iota_2);A\big) = \sum_{i \in I} \tilde Z(A)_{ii} \, .
\ee
We note that for the label of the transparent defect, $A=\one$, we have $C_{l/r}(A)=\one$ and $\tilde Z(A)_{ij} = \delta_{i,j}$, so that in this case all three invariants reduce to $\zrt(S^1 \times S^1 \times S^1) = |I|$, as expected.

To give a concrete example where one finds different values, consider the modular tensor category associated to the affine Lie algebra $\widehat{\mathfrak{sl}}(2)_k$ at level $k=16$, with simple objects $U_0,U_1,\dots, U_{16}$. 
There are three Morita classes of simple $\Delta$-separable Frobenius algebras \cite{Ostrik:2001}, labelled by the Dynkin diagrams $\textrm{A}_{17}$, $\textrm{D}_{10}$, $\textrm{E}_7$,  with representatives 
$A(\textrm{A}_{17})=\one$, 
$A(\textrm{D}_{10}) = U_0 \oplus U_{16}$, 
$A(\textrm{E}_{7}) = U_0 \oplus U_8 \oplus U_{16}$, see also \cite[Sect.\,3.6.2]{tft1}.
Their left/right centres are
\be
	C_{l/r}(A(\textrm{D}_{10})) = C_{l/r}(A(\textrm{E}_{7})) = U_0 \oplus U_{16} \, ,
\ee
which have categorical dimension~2. The following table gives the value of $\zzc((N,\iota);A)$ in the different cases:
\begin{center}
\begin{tabular}{c|ccc}
& $A_{17}$ & $D_{10}$ & $E_{7}$ 
\\
\hline  
$\iota^\pm_0$ & $17$ & $34$ & $34$ 
\\
$\iota^\pm_1$ & $17$ & $18$ & $18$
\\
$\iota_2$ & $17$ & $10$ & $7$
\end{tabular}
\end{center}
The relevant matrices for the last line can be found in \cite{Cappelli:1987xt}, but in any case the trace is equal to the number of nodes of the Dynkin diagram labelling the Morita class.

\end{document}